\documentclass[reqno, 11pt]{amsart}
\usepackage[utf8]{inputenc}
\usepackage[all]{xy}
\usepackage{enumitem}
\usepackage[english]{babel}
\usepackage[dvipsnames]{xcolor}
\usepackage{graphicx,pstricks,pst-plot}
\usepackage[colorlinks=true, linkcolor=MidnightBlue, citecolor=ForestGreen, urlcolor=green]{hyperref}
\usepackage{setspace}
\usepackage{cancel}
\usepackage[left=2.5cm, right=2.5cm, bottom=2cm]{geometry} 
\usepackage{amsmath, amssymb, amsthm, epsfig, mathtools}
\usepackage{latexsym}
\usepackage[nameinlink,capitalise]{cleveref}
\usepackage{url}
\usepackage{bbm}

\newtheorem{theorem}{Theorem}[section]
\newtheorem{lemma}[theorem]{Lemma}
\newtheorem{proposition}[theorem]{Proposition}
\newtheorem{corollary}[theorem]{Corollary}
\theoremstyle{definition}
\newtheorem{definition}[theorem]{Definition}

\theoremstyle{remark}
\newtheorem{remark}[theorem]{Remark}
\newtheorem{assum}{Assumption}
\newenvironment{assump}[2][]
  {\renewcommand\theassum{#2}\begin{assum}[#1]}
  {\end{assum}}

\newcommand{\tr}{\text{tr}}
\newcommand{\ld}{\lambda}
\newcommand{\Ld}{\Lambda}
\newcommand{\al}{\alpha}
\newcommand{\veps}{\varepsilon}
\newcommand{\vp}{\varphi}
\newcommand{\bC}{\mathbf{C}}
\newcommand{\Op}{\textsc{Op}}
\newcommand{\bh}{\mathfrak{h}}
\newcommand{\fk}{\mathfrak{f}}

\DeclareMathOperator{\Id}{Id}
\DeclareMathOperator{\Int}{Int}
\DeclareMathOperator{\Isom}{Isom}
\DeclareMathOperator{\Ima}{Im}
\DeclareMathOperator{\loc}{loc}
\DeclareMathOperator{\adj}{Adj}

\DeclareMathOperator{\supp}{\mathrm{supp}}
\DeclareMathOperator{\lin}{\mathrm{lin}}
\DeclareMathOperator{\Nlin}{\mathrm{Nlin}}
\DeclareMathOperator{\phg}{\mathrm{phg}}
\providecommand{\norm}[1]{\lVert#1\rVert}
\newcommand{\Norm}[1]{\left\lVert#1\right\rVert}
\providecommand{\inn}[1]{\langle#1\rangle}

\providecommand{\Int}[1]{\text{Int}(#1)}

\newcommand{\mc}{\mathcal}
\newcommand{\A}{\mc{A}}
\newcommand{\B}{\mc{B}}
\newcommand{\D}{\mathcal{D}}

\newcommand{\F}{\mc{F}}
\newcommand{\G}{\mc{G}}
\renewcommand{\L}{\mc{L}}
\renewcommand{\H}{\mathcal{H}}
\newcommand{\I}{\mc{I}}
\newcommand{\J}{\mc{J}}
\newcommand{\K}{\mc{K}}
\newcommand{\M}{\mc{M}}
\renewcommand{\O}{\mathcal{O}}
\renewcommand{\P}{\mathcal{P}}
\newcommand{\Q}{\mathcal{Q}}
\newcommand{\T}{\mc{T}}

\newcommand{\Y}{\mc{Y}}

\newcommand{\R}{\mathbb{R}}
\newcommand{\N}{\mathbb{N}}
\newcommand{\Z}{\mathbb{Z}}

\begin{document}

\title[Unique continuation and stabilization for NLS under the GCC]{Unique continuation and stabilization for nonlinear Schrödinger equations under the\\ Geometric Control Condition}
\author[]{Cristóbal Loyola}
\date{\today}
\address{Sorbonne Université, Université Paris Cité, CNRS, Laboratoire Jacques-Louis Lions, LJLL, F-75005 Paris, France}
\email{cristobal.loyola@sorbonne-universite.fr}

\allowdisplaybreaks
\numberwithin{equation}{section}

\begin{abstract}
    In this article we prove global propagation of analyticity in finite time for solutions of semilinear Schrödinger equations with analytic nonlinearity from a region $\omega$ where the Geometric Control Condition holds. Our approach refines a recent technique introduced by Laurent and the author, which combines control theory techniques and Galerkin approximation, to propagate analyticity in time from a zone where observability holds. As a main consequence, we obtain unique continuation for subcritical semilinear Schrödinger equations on compact manifolds of dimension $2$ and $3$ when the solution is assumed to vanish on $\omega$. Furthermore, semiglobal control and stabilization follow only under the Geometric Control Condition on the observation zone. In particular, this answers in the affirmative an open question of Dehman, Gérard, and Lebeau from $2006$ for the nonlinear case.
\end{abstract}

\maketitle
\tableofcontents

\section{Introduction}\label{sec:intro}

The aim of this article is to study how, for a certain class of evolution PDEs, properties observed from a subset $\omega\subset \M$ over the time $(0, T)$ are propagated to the whole solution on $(0, T)\times \M$. Here, $\M$ can be, for instance, a compact Riemannian manifold without boundary. More precisely, we will study the following two properties:
\begin{enumerate}
    \item \textbf{Propagation of analyticity}: if the solution is analytic in time on $(0,T)\times \omega$, is the full solution  analytic in time on $(0,T)\times\M$?
    \item \textbf{Unique continuation}: if the solution is zero on $[0, T]\times \omega$, is the solution identically zero on $[0, T]\times\M$?
\end{enumerate}
We refine the method introduced in \cite{LL24} by broadening the class of admissible nonlinearities, thus extending its scope to a wider class of conservative equations. This method relies on observability estimates, a Galerkin procedure, and the interaction of low and high frequencies through the nonlinearity. We provide an abstract result, detailed in \cref{s:abstrNLSIntro}, that allows us to propagate analyticity in time from the observation to the full solution. We present the main applications to nonlinear Schrödinger equations, where, notably, we obtain global unique continuation under the Geometric Control Condition and give applications in control theory, see \cref{sec:introNLS}. We believe that the abstract method could be applied to several other systems and is amenable to generalizations.

To motivate the abstract result, we begin by describing the case of nonlinear Schrödinger equations.

\subsection{Main results on nonlinear Schrödinger equation}\label{sec:introNLS} Let $(\M, g)$ be a compact boundaryless smooth connected Riemannian manifold of dimension $d$. Let $\Delta_g$ be the Laplace-Beltrami operator on $\M$ associated to the metric $g$. In this section we consider the semilinear Schrödinger equation
\begin{align}\label{eq:NLS}
        \left\{\begin{array}{cc}
            i\partial_tu+\Delta_g u=f(u)     &~~ \text{ in } (0, T)\times\M,  \\
            u(0)=u_0,    & 
        \end{array}\right.
\end{align}
where $u: [0, T]\times\M\to\mathbb{C}$ and the nonlinearity $f: \mathbb{C}\to \mathbb{C}$ is \emph{real analytic} with $f(0)=0$. We assume that $u_0\in H^s(\M)$, where:
\begin{enumerate}[label=(\Alph*),ref=(\Alph*)]
    \item\label{NLS:reg:1} If there is no further assumption, we set $s>d/2$, which ensures that $H^{s}\hookrightarrow L^\infty$.
    \item\label{NLS:reg:2} Let $1<d\leq 3$ and assume $f$ to be of polynomial type, particularly $f(u)=P'(|u|^2)u$ where:
    \begin{enumerate}[label=(\Alph{enumi}-\arabic*),ref=(\Alph{enumi}-\arabic*)]
        \item\label{NLS:reg:2.1} if $d=2$ then $P$ is a polynomial function with real coefficients, satisfying $P(0)=0$ and the defocusing assumption $P'(r)\xrightarrow[r\to+\infty]{}+\infty$;
        \item\label{NLS:reg:2.2} if $d=3$, then $P'(r)=\al r+\beta$ with $\al>0$, $\beta\geq 0$, corresponding to the cubic nonlinearity.
    \end{enumerate}
    In any of these cases, we choose $s=1$ as the level of regularity.
\end{enumerate}

We will assume from now onward that the observation set $\omega\subset \M$ is open and non-empty. Moreover, we will impose that $\omega$ satisfies the Geometric Control Condition:

\begin{assump}{GCC}\label{NLS:assumGCC}
    There exists $T_0>0$ such that every geodesic of $\M$ traveling at speed $1$ meets $\omega$ in time $t\in (0, T_0)$.
\end{assump}

Our first main result is the following, corresponding to case \ref{NLS:reg:1}.

\begin{theorem}\label{thm:NLS-analytic}
Let $d\in \N$ and $s>d/2$. Let $u\in C^0([0, T], H^{s}(\M))$ be a solution of \eqref{eq:NLS}. Assume that the above setting holds and assume moreover that:
    \begin{enumerate}
        \item $\omega$ satisfies the \ref{NLS:assumGCC},
        \item $t\in (0, T)\mapsto \chi u(t,\cdot)\in H^{s}(\M)$ is analytic for any cutoff function $\chi\in C_c^\infty(\M)$ whose support is contained in $\omega$.
    \end{enumerate}
    Then $t\mapsto u(t, \cdot)$ is analytic from $(0, T)$ into $H^2(\M)$.    
\end{theorem}

The main consequence of this propagation of analyticity in time is the following unique continuation property for solutions of \eqref{eq:NLS}, notably, in the $H^1$-subcritical cases detailed in \ref{NLS:reg:2}. First, in dimension $d=2$, for solutions with finite Strichartz norm in the sense of \cite[Theorem 2]{BGT04} (see \cref{s3:wp} below for precisions), the structure of the nonlinearity dictates the type of equilibrium we will obtain.

\begin{theorem}\label{thm:UCP-NLSdim2}
    Let $d=2$. Assume that $\omega$ satisfies the \ref{NLS:assumGCC} and that $f$ is given as in \ref{NLS:reg:2.2}. If one solution $u\in C([0, T], H^1(\M))$ of \eqref{eq:NLS} with finite Strichartz norms satisfies $\partial_tu=0$ in $(0, T)\times\omega$, then $\partial_t u=0$ in $(0, T)\times \M$ and $u$ is an equilibrium point of \eqref{eq:NLS}, that is, solution of
    \begin{align}\label{thm:eq:UCP-NLSdim2}
        -\Delta_g u+P'(|u|^2)u=0,\ ~~ x\in\M.
    \end{align}
    Moreover, if there exists $C>0$ such that $P'(r)\geq C$ for $r\geq 0$, then $u=0$ in $(0, T)\times\M$.
\end{theorem}

In dimension $d=3$, let $(\M, g)$ be any of the following manifolds:
\begin{itemize}
    \item $\mathbb{T}^3$ or the irrational torus $\R^3/(\theta_1 \Z \times \theta_2 \Z \times \theta_3\Z)$ with $\theta_i \in \R$,
    \item $S^3$ or $S^2\times S^1$.
\end{itemize}
The analysis of the NLS in the subcritical case uses Bourgain spaces, see \cref{s3:wp} for precisions. We get the following unique continuation property.

\begin{theorem}\label{thm:UCP-NLSdim3}
    Let $d=3$ and let $(\M, g)$ be any of the manifolds described above. Assume that $\omega$ satisfies the \ref{NLS:assumGCC} and that $f$ is given as in \ref{NLS:reg:2.1}. Let $1/2<b\leq 1$ and $u\in X_T^{1, b}$ be a solution of \eqref{eq:NLS} which satisfies $\partial_t u=0$ in $(0, T)\times\omega$. Then $u=0$ in $(0, T)\times \M$.
\end{theorem}

\begin{remark}
    More generally, it is most likely true that the above unique continuation result holds under the more general bilinear estimates \cref{assumBilWP} on $(\M, g)$. Indeed, although not written explicetly in the hypotheses of the corresponding results in \cite{Lau10:NLS3d}, the microlocal propagation machinery applies under such an assumption, which are precisely what we need to treat the low regularity framework. In particular, any of the aforementioned manifolds with $d=3$ satisfy \cref{assumBilWP}.
\end{remark}

Our unique continuation result allows us to obtain some results in non-compact domains. For instance, we can obtain a result for $(\R^2, g)$ equipped with a smooth bounded metric $g$, see \cref{S3:UCP:unbounded} below for precisions.

\begin{proposition}\label{prop:UCP-unbounded}
    Let $\M=\R^2$ be equipped with a smooth metric $g$ bounded above and below by positive constants and whose derivatives are bounded. Let $f$ be as in \ref{NLS:reg:2.1}. Suppose that there exists $R>0$ such that $\R^2\setminus B(0, R)\subset \omega$. Let $u\in C^0([0, T], H^1(\R^2))$ be a solution with finite Strichartz norms\footnote{That is, it belongs to the space that ensures the wellposedness of the equation in the sense of \cite[Section 3, Remark 5]{BGT04}.} of
    \begin{align}
    \left\{\begin{array}{cl}
        i\partial_t u+\Delta_g u=P'(|u|^2)u     &~~ (0, T)\times \R^2,  \\
        \partial_t u=0     &~~ (0, T)\times\omega, 
    \end{array}\right.
    \end{align}
    If $\omega$ satisfies the \ref{NLS:assumGCC}, then the same conclusions of \cref{thm:UCP-NLSdim2} hold.
\end{proposition}

\subsubsection{Control and stabilization contributions} The control and stabilization of nonlinear Schrödinger equations has been shown to be a difficult but highly interesting field of research. We mention, for instance, the works of Rosier and Zhang \cite{RZ09, RZ10} on local control results, and the survey of Laurent \cite{Lau14} for a thorough summary of results up to 2014. Here, we will focus on global results.

Unique continuation often plays a key role in obtaining (semi)global results of this kind. For instance, to stabilize and control the wave equation in the subcritical case, Dehman, Lebeau, and Zuazua \cite{DLZ03} employed a method that combines two main ingredients to obtain the required observability inequality: the geometric control condition and a unique continuation property. Even though they worked on a specific geometry, their strategy is quite general and it has been successfully applied to study the controllability properties of other systems. Notably, for the control and stabilization of the Schrödinger equation on compact surfaces, Dehman, Gérard, and Lebeau \cite{DGL06} made crucial use of the Strichartz estimates due to Burq, Gérard, and Tzvetkov \cite{BGT04} which allowed them to develop microlocal propagation methods adapted to the equation. Their results hold under the following two assumptions on the pair $(\M, \omega)$:
\begin{itemize}
    \item The observation set $\omega\subset\M$ satisfies the \ref{NLS:assumGCC}.
    \item For every $T>0$, the only solution lying in the space $C([0, T], H^1(\M))$ to the system
    \begin{align}\label{NLS:UCP}\tag{UCP}
        \left\{\begin{array}{cl}
            i\partial_t u+\Delta_g u+b_1(t, x)u+b_2(t, x)\overline{u}=0     & (t, x)\in (0, T)\times\M,  \\
            u=0     & (t, x)\in (0, T)\times\omega, 
        \end{array}\right.
    \end{align}
    where $b_1$ and $b_2$ belong to $L^\infty([0, T], L^p(\M))$ for some $p>0$ large enough, is the trivial one $u=0$.
\end{itemize}
The second assumption is a unique continuation property with potentials, referred to as \ref{NLS:UCP} in what follows. Under analogous assumptions, Laurent \cite{Lau10:NLS1d, Lau10:NLS3d} extended their analysis in dimension one and in dimension three by adapting their framework to Bourgain spaces, respectively. It is worth mentioning that whenever $d\geq 2$, in these works the unique continuation property \ref{NLS:UCP} is an assumption, and it is proved to be compatible with the \ref{NLS:assumGCC} in some specific geometries. Here, we obtain the unique continuation property for the nonlinear equation as a consequence of the \ref{NLS:assumGCC}, thereby allowing us to drop the unique continuation assumption \ref{NLS:UCP} from their works and to establish that the \ref{NLS:assumGCC} is a sufficient condition for stabilization and control of these models. Notably, we answer to the positive to the open question stated in \cite[Remark 3]{DGL06} and by \cite{Lau10:NLS3d} in the nonlinear case.

Let $(\M, g)$ be any of the following manifolds:
\begin{itemize}
    \item if $d=2$ then $(\M, g)$ is a compact Riemannian surface,
    \item $\mathbb{T}^3$ or the irrational torus $\R^3/(\theta_1 \Z \times \theta_2 \Z \times \theta_3\Z)$ with $\theta_i \in \R$,
    \item $S^3$ or $S^2\times S^1$.
\end{itemize}
Let $\mc{X}_d$ denote the right functional space inherited from the wellposedness framework. Namely, for $d=2$, $\mc{X}_2$ ensures finite Strichartz norm and for $d=3$, $\mc{X}_3$ corresponds to the Bourgain space $X_T^{1, b}$ with $b\in (1/2, 1]$. This allows us to handle the subcriticality on each situation. To not over complicate the statements below, we refer to \cref{s3:wp} for more precisions.

In what follows, we will relate $\omega$ to a cutoff function $a\in C^\infty(\M, \R)$ such that $\omega=\{x\in \M\ |\ a(x)\neq 0\}$.

\begin{theorem}\label{thm:stab}
    Let $(\M, g)$ be any of the manifolds described above and $f$ be as in \ref{NLS:reg:2} according to the dimension of the manifold. Let $\omega\subset\M$ satisfy the \ref{NLS:assumGCC}. Then for every $R_0>0$, there exist two constants $C>0$ and $\gamma>0$ such that if $\norm{u_0}_{H^1}\leq R_0$ then
    \begin{align*}
        \norm{u(t)}_{H^1}\leq Ce^{-\gamma t}\norm{u_0}_{H^1},\ t>0,
    \end{align*}
    holds for every $u\in\mc{X}_d$ solution of the damped system
    \begin{align*}
        \left\{\begin{array}{cc}
            i\partial_t u+\Delta_g u-a(x)(1-\Delta_g)^{-1}a(x)\partial_t u=P'(|u|^2)u  &~~ (t, x)\in [0, \infty)\times\M,  \\
            u(0)=u_0.     & 
        \end{array}\right.
    \end{align*}
\end{theorem}

\cref{thm:stab} is a direct combination of \cite[Theorem 1]{DGL06} and \cref{thm:UCP-NLSdim2} for $d=2$, and \cite[Theorem 2]{Lau10:NLS3d} and \cref{thm:UCP-NLSdim3} for $d=3$. By time reversibility, stabilization and local control around $0$ allows us to obtain the following semiglobal exact controllability result.

\begin{theorem}\label{thm:stabNLS}
    Let $(\M, g)$ be any of the manifolds described above and $f$ be as in \ref{NLS:reg:2} according to the dimension $d$ of the manifold. Let $\omega\subset\M$ satisfy the \ref{NLS:assumGCC}. Then for every $R_0>0$, there exist $T>0$ and $C>0$ such that if for every $u_0$, $u_1\in H^1(\M)$ with $\norm{u_0}_{H^1}\leq R_0$ and $\norm{u_1}_{H^1(\M)}\leq R_0$, then there exists a control $g\in C^0([0, T], H^1(\M))$ with $\norm{g}_{L^\infty([0, T], H^1(\M))}\leq C$ supported in $[0, T]\times\overline\omega$ such that the unique solution $u\in \mc{X}_d\subset C^0([0, T], H^1(\M))$ of the system
    \begin{align*}
        \left\{\begin{array}{cc}
            i\partial_t u+\Delta_g u=P'(|u|^2)u+g  &~~ (t, x)\in [0, T]\times\M,  \\
            u(0)=u_0,     & 
        \end{array}\right.
    \end{align*}
    satisfies $u(T,\cdot)=u_1$.
\end{theorem}

\subsubsection{Literature on unique continuation} As previously described, a main motivation coming from control theory is to prove a statement of the form:
\begin{center}
    $\omega$ satisfies the \ref{NLS:assumGCC}\ $\Longrightarrow$\ Unique continuation property \ref{NLS:UCP} holds.
\end{center}
This is a key step to establish sufficient geometric conditions on the observation zone for the controllability and stabilization results to hold. A classical strategy to prove a unique continuation property for \eqref{eq:NLS} from the observation $\partial_t u=0$ in $(0, T)\times \omega$ is to take time derivative $z=\partial_t u$, which leads to establish \ref{NLS:UCP} with low-regularity potentials involving $V=f'(u)$. Below we review some already known results on unique continuation and explain why this result appears to be difficult to obtain with the current available techniques.

A first approach to unique continuation might be to employ the general theory of Hörmander \cite{Hor63:LPDE1, Hor85ALPDOIII}, where the potential $V$ involving $f'(u)$ has at most the same regularity as $u$. The main feature of Hörmander's theory is that it gives a \emph{geometric} condition on the hypersurface $S=\{\Psi=0\}$, the pseudoconvexity, sufficient for the local unique continuation across $S$. Regarding the regularity of the coefficients, as it is based on Carleman estimates, it is well-suited for potentials with rough regularity. Obtaining a global unique continuation result then requires the propagation of local unique continuation across a well-chosen family of hypersurfaces $S$ verifying the \emph{pseudoconvexity} assumption. However, this leads to a global geometric assumption which is known to be stronger than the \ref{NLS:assumGCC}, see Miller \cite{Mil03}. Furthermore, we point out that Hörmander's theorem is empty in our situation the pseudoconvexity assumptions are never satisfied for the Schrödinger operator. However, Lascar and Zuily \cite{LZ82} showed that Hörmander's theorem hold in the anisotropic case by appropriate modification of the symbol classes and Poisson bracket, by taking into account the anisotropic (or quasi-homogeneous) nature of the Schrödinger operator. We mention the works \cite{Deh84, Isa93, Tat97} for further results in this direction. As expected, the pseudoconvexity assumption is a strong local geometric condition and it naturally leads to a strong global geometric assumption on the observation set $\omega$ in its global version. When it comes to global Carleman estimates with applications to control or inverse problems, we mention \cite{BP02, TX07, MOR08, Lau10:NLS3d}. In particular, a weak pseudoconvexity condition has been proved to be sufficient in \cite{MOR08} for a flat metric. Notably, Laurent \cite{Lau10:NLS3d} employed this pseudoconvexity condition to obtain unique continuation one some compact manifolds of dimension $3$ from some particular observation zones satisfying the \ref{NLS:assumGCC}. Nevertheless, this approach forces to check the (weak)pseudoconvexity condition on each situation we consider, which is mostly impossible in general.

If we want to go beyond the pseudoconvexity assumption, another available unique continuation result in the above is the John-Holmgren theorem. Although it gives better geometric assumptions than weak pseudoconvexity, there are some strong drawbacks to this result. On the one hand, it requires analyticity on all the coefficients of the differential operator under consideration. On the other hand, it is not stable under $C^\infty$ lower order perturbations and a counterexample of Métevier \cite{M:93} showed that a nonlinear version of this theorem does not hold in general.

Regarding the propagation of analyticity for nonlinear equations, several results date back to the 1980s and 1990s. Alinhac-M\'etivier proved in \cite{AM:84,AM:84b} that if $u$ is a regular enough solution of a general nonlinear PDE, the analyticity of $u$ propagates along any hypersurface for which the real characteristics of the linearized operator cross the hypersurface transversally. Subsequently, there has been an intense activity to understand what kind of singularities propagate for nonlinear systems, of which waves  particular cases. It was found that the situation is quite complicated since microlocal analytic singularities do not remain confined to bicharacteristics as in the linear case, but can give rise to nonlinear interactions. For more details, see Godin \cite{G:86} and G\'erard \cite{G:88}. 

In our geometric context, to obtain a global result from local propagation of singularities, the \ref{NLS:assumGCC} would force to propagate from hypersurfaces of the form $S=\{\psi=0\}$ with $\psi(t,x)=\psi(x)$, in which case, the Schödinger operator is never hyperbolic with respect to $S$. Thus the propagation results from \cite{AM:84,AM:84b} do not seem to apply, despite the existence of bicharacteristics transverse to $S$.

The problem is better understood in the $C^\infty$ or $H^s$ setting. 
The first works of propagation of singularities for Schrödinger equations go back to Lascar \cite{Las77} and Boutet de Monvel \cite{BdM75} where they introduce parabolic wavefront set which propagates along the geodesics at fixed time. To study propagation and reflection of singularities on $\R_+^d$, Szeftel \cite{Sze05} develops a paradifferential calculus well-suited to the nonlinear Schrödinger equation. Global results under the \ref{NLS:assumGCC} are known to hold on compact manifolds in the subcritical case via Strichartz estimates and adapted bootstrap arguments \cite{DGL06}. By contrast, adapting such bootstrap arguments in the analytic category seems to be way more complicated: even for $f=0$, local propagation of the analytic wavefront set for Schrödinger relies on microlocal methods and exhibits delicate behavior near glancing and in the presence of a boundary due to infinite speed of propagation, see Robbiano–Zuily \cite{RZ99} and Martinez–Nakamura–Sordoni \cite{MNS10}. It is uncertain if a global propagation of analytic regularity on manifolds, especially with boundary, can be obtained under the same assumptions using the analysis for the linear case.

So far, we have seen that for local unique continuation, there is an interplay in between the geometric restrictions and the regularity of the coefficients of the differential operator. In this direction, in a series of remarkable works by Robbiano \cite{Robbiano:91}, Tataru \cite{Tat95, Tat99}, Robbiano and Zuily \cite{RZ:98} and Hörmander \cite{Hormander:92, Hor:97}, a general unique continuation result that interpolates in between Holmgren and Hörmander's theorem was obtained. In the particular case of the Schrödinger operator $P=i\partial_t+\Delta_g+V$, it states that local unique continuation holds assuming regular enough metric $g$, that the potential depends analytically in the time variable $t$ only and that the hypersurface $S=\{\Psi=0\}$ is non-characteristic to $P$. That is, if $p_2(t, x, \xi_t, \xi_x)=-\sum_{j, k} g^{jk}(x)\xi_{x_j}\xi_{x_k}$ denotes the principal symbol of $P$, where $\xi_t$ is the dual variable to $t$ and $\xi_{x_j}$ is the dual variable to $x_j$, the non-characteristic assumption translates into
\begin{align*}
    p_2(t_0, x_0, d\Psi(t_0, \xi_0))\neq 0\ \Longleftrightarrow\ (\nabla_x\Psi)(t_0, \xi_0)\neq 0.
\end{align*}
This hypothesis is optimal: it only excludes the surfaces that are tangent to $\{t=t_0\}$, for which unique continuation is, in general, not verified. We refer to \cite[Section 5]{FLL24:X} for details on the counterexample. This leads to a global unique continuation result from $(0, T)\times\omega$ for any $T>0$ and any nonempty open $\omega\subset \M$, as long as the potential $V$ depends analytically in the time variable $t$; see \cref{thm:UCP-linearschrodinger} below. Regarding further results in this direction, T'joën \cite{Tjo00} has proved a quasi-homogeneous variant of the Tataru-Robbiano-Zuily-Hörmander theorem. Recently, the analyticity assumption was relaxed to the $2$-Gevrey class by Filippas, Laurent and Léautaud \cite{FLL24} by exploiting the anisotropy of the Schrödinger operator.

Thus, going back to our main motivation, to apply this result in our situation we would need to prove that $V=f'(u)$ is analytic in time, leading to prove the same for $u$. From the point of view of regularity, asking for analyticity is indeed a quite strong hypothesis and a solution $u$ to the nonlinear Schrödinger equation \eqref{eq:NLS}, a priori, has no reason to be analytic in time. This observation underpins the importance of our result \cref{thm:NLS-analytic}, which in particular allows us to obtain unique continuation for analytic nonlinearities and under the sole \ref{NLS:assumGCC} as geometric condition.

\subsection{Abstract frequency-based reconstruction}\label{s:abstrNLSIntro}

Let $T>0$. In this section, we consider the following nonlinear observability system
\begin{align}\label{NLS:eq:nleq-obs}
\left\{\begin{array}{cl}
\partial_t u=Au+\fk(u+\bh_1)+\bh_2&\ \text{ in } [0, T],\\
\bC u(t)=0&\ t\in [0, T],
\end{array}\right.
\end{align}
on a suitable real Hilbert space $X$, where $A$ is a skew-adjoint operator on $X$, $\fk$ is a mapping from $X$ into itself, $\bh_1$ and $\bh_2$ are some applications from $[0, T]$ into $X$, and $\bC$ is a linear bounded observation operator in $X$. 

In what follows, we will make several assumptions that will be enforced towards our main result. The first assumption dictates the class of PDEs we will be working with.

\begin{assump}{1}
\label{NLS:assumA}$A$ is a skew-adjoint operator with domain $D(A)$ on a real separable Hilbert space $X$, so that $A^*A=-A^2$ has a compact resolvent. 
\end{assump}
We will now list some consequences of such an assumption. That $A^*A=-A^2$ is non-negative self-adjoint, allows us to define the Hilbert space $X^{\sigma}=D((A^*A)^{\sigma/2})\hookrightarrow X$ for any $\sigma\in\R$. 

By the spectral theorem, and since $A^*A$ has a compact resolvent, the spectrum of $A^*A$ is real and discrete, allowing us to construct an orthonormal basis of eigenvectors of $A^*A$ in $X$, denoted by $(e_j)_{j\in\N}$ and associated to the nonnegative eigenvalues $(\lambda_j)_{j\in \N}$ (ranged increasingly) with $\lambda_j \longrightarrow +\infty$ as $j\to+\infty$. We introduce the high-frequency projectors $\Q_n$ on the space $\overline{\text{span}\{e_j\}_{j\geq n}}$ and then we set the low-frequency projection $\P_n=I-\Q_n$. Note that $A$ commutes with $\P_n$ and $A\P_n$ is a bounded operator of $X^{\sigma}$ into itself with norm less than $\inn{\ld_n}$.

The parameter $\sigma$ will be fixed from now on. We will use the notation $\P_n X^{\sigma}$ or $\Q_n X^{\sigma}$ for that related image of the Hilbert space endowed with the topology of $X^{\sigma}$.

 Let us introduce some notation. For a given real Banach space $Y$, we denote the ball centered at $0$ of radius $M$ by
\begin{align*}
    \mathbb{B}_M(Y):=\{y\in Y\ |\ \norm{y}_Y\leq M\}.
\end{align*}
 For a given non-empty interval $I\subset \R$, we denote the ball of radius $R$ of $C^0(I, Y)$ by
\begin{align*}
    \B_{R}^{I}(Y)=\{y\in C^0(I, Y)\ |\ \norm{y}_{C^0(I, Y)}\leq R\},
\end{align*}
equipped with the $L^\infty([0, T], Y)$-norm. Building upon this notation, given a non-empty set $\A\subset Y$ we introduce
\begin{align*}
    \B_{R}^{I}(Y, \A)=\{y\in \B_R^I(Y)\ |\ y(t)\in \A\text{ for every }t\in [0, T]\}.
\end{align*}
Furthermore, we introduce the canonical complexification $Y_\mathbb{C}$, defined as the set of elements $y_1+iy_2$, $y_j\in Y$, see \cite[Section 2]{BS:71-polynomials}. We then introduce following notation for the \emph{cylinder} on $Y_\mathbb{C}$,
\begin{align*}
    \mathbb{B}_{M, \delta}(Y)=\{y\in Y_\mathbb{C}\ |\ \norm{\Re(y)}_Y\leq M\ \text{and}\ \norm{\Im(y)}_Y\leq \delta\},
\end{align*}
and similarly on $C^0(I, Y_\mathbb{C})$,
\begin{align*}
    \B_{R, \delta}^{I}(Y):=\{z\in C^0(I, Y_\mathbb{C})\ |\ \forall t\in I,\ \norm{\Re(z(t))}_{Y}\leq R,\ \norm{\Im(z(t))}_{Y}\leq \delta\}.
\end{align*}
The latter space is equipped with the natural $L^\infty([0, T], Y_\mathbb{C})$-norm.

We will make the following assumption on the nonlinearity.

\begin{assump}{2}\label{NLS:assumFholom}
    The nonlinear map $\fk: X^\sigma\to X^\sigma$ belongs to $C^2(X^\sigma)$ and is bounded on the bounded sets of $X^\sigma$. Moreover, its differential map $D\fk: x\in X^\sigma\mapsto D\fk(x)\in \L(X^\sigma)$ is bounded and Lipschitz-continuous on the ball $\mathbb{B}_{4R_0}(X^\sigma)$, for some $R_0>0$.
    
    Furthermore, there exists $\delta>0$ so that $\fk$ can be extended holomorphically from $\mathbb{B}_{4R_0, 2\delta}(X_\mathbb{C}^\sigma)$ into $X_\mathbb{C}^\sigma$. This means that $\fk$ is holomorphic in the interior of this ball and continuous up to the boundary. Moreover, for some $L>0$, the following inequalities hold for any $z$, $z'\in \mathbb{B}_{4R_0, 2\delta}(X^\sigma)$:
    \begin{align}\label{ineq:assumFholom}
        \norm{\fk(z)}_{X_\mathbb{C}^\sigma}\leq L,\ \norm{\fk(z)-\fk(z')}_{X_\mathbb{C}^\sigma}\leq L\norm{z-z'}_{X_\mathbb{C}^\sigma}\ \text{ and }\ \norm{D\fk(z)-D\fk(z')}_{\mc{L}(X_\mathbb{C}^\sigma)}\leq L\norm{z-z'}_{X_\mathbb{C}^\sigma}.
    \end{align}
\end{assump}

We make an observability assumption for solutions of the linear system
\begin{align}\label{NLS:eq:ode-hf-intro}
    \left\{\begin{array}{cc}
        \partial_t w=(A+\Q_n D\fk(v))w, & t\in (0, T], \\
        w(0)=w_0\in \Q_n X^\sigma,& 
    \end{array}\right.
\end{align}
which is uniform with respect to the parameter $n$ and the input $v\in C^0([0, T], X^\sigma)$. Denote by $S_n(v)$ the evolution operator associated to \eqref{NLS:eq:ode-hf-intro} (see \cref{lem:sv-properties} below for a more precise definition).

\begin{assump}{3}
\label{NLS:assumC} Let $R_0>0$, $T>0$ and $n_0\in \N$ be given and let $\mathbb{V}$ be a non-empty subset of $\B_{3R_0}^{[0, T]}(X^\sigma)$. Let $\bC\in\L(X^\sigma)$ be an observation operator. We assume that $t\mapsto S_n(v)(t, 0)$ is observable on $[0, T]$ for any $v\in \mathbb{V}$ and $n\geq n_0$. Moreover, there exists a constant $\mathfrak{C}_{obs}>0$ such that, for any $n\geq n_0$, for all $v\in \mathbb{V}$
\begin{align}\label{assum:observabstract}
    \norm{w_0}_{X^\sigma}^2\leq \mathfrak{C}_{\text{obs}}^2\int_0^T \norm{\bC S_n(v)(t, 0)w_0}_{X^\sigma}^2dt,\quad \forall w_0\in \Q_nX^\sigma.
\end{align}
\end{assump}

The main result is the following.

\begin{theorem}\label{thm:mainabs-analytic}
    Let $T>0$, $R_0>0$, $R_1>0$ and $n_0\in \N$. Let $T^{*}>T$. Let $\A$ be a compact subset of $X^\sigma$ and let $\mc{K}$ be a compact subset of $\B_{R_0}^{[0, T^*]}(X^\sigma)$. Let us further assume that both $\A$ and $\mc{K}$ are stable under the projector $\P_n$. Let $\bh_1\in \B_{R_0}^{[0, T^*]}(X^\sigma, \A)\cap\K$ and $\bh_2\in \B_{R_1}^{[0, T^*]}(X^\sigma, \A)$ be such that they both admit some extension $C^0([0, T^*]+i[-\mu, \mu], X_\mathbb{C}^\sigma)$, respectively, with $\mu>0$, so that the application
    \begin{align*}
        \left\{\begin{array}{rcl}
    (0,T^{*})+i(-\mu,\mu) & \longrightarrow&X^{\sigma}_{\mathbb{C}}\\
      z&\longmapsto& \bh_1(z)
        \end{array}\right.
    \end{align*}
    is holomorphic. We assume the same for $\bh_2$. We assume moreover that $\Re \bh_1(z)\in \mathbb{B}_{R_0}(X^{\sigma})$ for any $z\in [0,T^{*}]+i[-\mu,\mu]$.
    
    Assume that Assumptions \ref{NLS:assumA}, \ref{NLS:assumFholom} (with $R_{0}$) and \ref{NLS:assumC} (with $R_0$, $T$, $n_0$ and $\mathbb{V}=\B_{3R_0}^{[0, T]}(X^\sigma, \A+\A)$) hold.
    Then, any solution $u\in \mc{K}\subset C^{0}([0,T^{*}],X^{\sigma})$ satisfying $u(t)\in \A$ for $t\in [0, T^*]$ and
    \begin{align}\label{NLS:eq:UCabstT*intro}
        \left\{\begin{array}{cr}
             \partial_t u=Au+\fk(u+\bh_1)+\bh_2& \textnormal{ on } [0,T^{*}], \\
             \bC u(t)=0 &\textnormal{ for }t\in [0,T^{*}],
        \end{array}\right.
    \end{align}
    is real analytic in $t$ in $(0,T^*)$ with value in $X^{\sigma}$.
\end{theorem}

This result is a next natural step of the technique introduced by Laurent and the author in \cite{LL24}, which can be seen as a sort of finite-time adaptation of an abstract result due to Hale and Raugel \cite{HR03} in the context of dynamical systems. Here, we relax the compactness assumption on $\fk$ by introducing the uniform high-frequency observability Assumption \ref{NLS:assumC}. This assumption permits us to face the lack of compactness of the nonlinearity by considering the linearization around low frequencies and translating a big part of the analysis into this linear problem. In turn, the compactness hypotheses are now put onto the solution itself. 

We thus follow the same strategy of proof as in \cite{LL24}, but we perform a finer analysis regarding the behaviour of the nonlinearity and the linearized systems involved. In practice, this allow us to consider more general analytic nonlinearities whereas the compactness properties will depend on the equation under study, mostly obtained through smoothing effects or propagation of regularity. 

\subsection{Outline of the article} \cref{sec:NLSabstract-construction} is devoted to the proof of the abstract \cref{thm:mainabs-analytic}. \cref{sec:nls_uc} contains the applications to the nonlinear Schrödinger equation. It contains the verification that the abstract \cref{thm:mainabs-analytic} can be applied in the settings aforementioned for a sufficiently high regularity index $\sigma$. It also contains some propagation of regularity arguments that allow to reach this regularity $\sigma$ starting from the energy space. In \cref{A:eveq}, we gathered some analysis results, including complex analysis in Banach spaces. In \cref{A:pseudodiff} we recall some basic facts about pseudodifferential operators.

\subsubsection*{Notation} Given Banach spaces $X$ and $Y$, we denote by $\mc{L}(X, Y)$ the Banach space of all bounded linear operators from $X$ to $Y$. Sometimes to clarify the difference in between real and complex structures, if $X$ and $Y$ are complex Banach spaces, we will denote by $\mc{L}_\mathbb{C}(X, Y)$ the Banach space of bounded linear operators which are $\mathbb{C}$-linear with the inherited complex structure. Given two quantities $A$ and $B$, we will sometimes write $A\lesssim B$ to say that there exists a constant $C>0$, independent of $A$ and $B$ but possibly depending on other parameters, such that $A\leq CB$.

\subsection{Acknowledgment} I would like to warmly thank Camille Laurent for carefully reading an earlier version of this article and suggesting countless improvements, as well as for the helpful discussions, encouragement, and patient guidance.

This project has received funding from the European Union's Horizon 2020 research and innovation programme under the Marie Sk\l{}odowska-Curie grant agreement No 945332.

\section{Abstract analytic reconstruction}\label{sec:NLSabstract-construction}

\subsection{Sketch of the strategy} The aim of this section is to prove \cref{thm:mainabs-analytic}. For the reader's convenience, we will briefly outline its proof, based on a generalized Galerkin decomposition introduced in Hale-Raugel \cite{HR03} and Laurent and the author \cite{LL24}. 

In this section we will employ the notations introduced in \cref{s:abstrNLSIntro}. Let $T>0$ and $\sigma\geq 0$ be fixed parameters from now onward, unless specified otherwise. Let $u=u(t)$ be a mild solution of \eqref{NLS:eq:UCabstT*intro} in $C^0([0, T], X^\sigma)$ and suppose $\bh_1=0$, $\bh_2=0$ for simplicity. Recall that, from Assumption \ref{NLS:assumA}, we have the low and high-frequency projections $\P_n$ and $\Q_n=I-\P_n$, which allows us to consider the splitting
\begin{align*}
    u(t)=\P_n u(t)+\Q_nu(t)=v(t)+w(t),
\end{align*}
where $\big(v, w\big)$ solves the following system
\begin{align*}
    \left\{\begin{array}{rl}
         \partial_t v(t)&=Av(t)+\P_n \fk(v+w),  \\
         \partial_t w(t)&=Aw(t)+\Q_n \fk(v+w), \\
         \bC w(t)&=-\bC v(t).
    \end{array}\right.
\end{align*}
By Duhamel's formula, the high-frequency component $w$ can be written as
\begin{align*}
    w(t)=e^{tA}w(0)+\int_0^t e^{A(t-s)}\Q_n \fk\big(v(s)+w(s)\big)ds.
\end{align*}
The observation condition $\bC w=-\bC v$ suggests that given $v$, we can reconstruct $w$ by solving the corresponding nonlinear observability system. Since here we aim to relax compactness assumptions on the nonlinearity $\fk$, we cannot directly setup a fixed point argument as in \cite{LL24}. To face this lack of compactness on $\fk$, we will consider a linearization of the component $w$ along the component $v$ of the solution. Let us introduce $\H: X^\sigma\times X^\sigma\to X^\sigma$ defined by
\begin{align*}
    \H(v, w):=\int_0^1 [D\fk(v+\tau w)-D\fk(v)]wd\tau.
\end{align*}
Formally, by writing $\fk(v+w)=D\fk(v)w+\fk(v)+\H(w, v)$, we are led to study the nonlinear observability problem at high-frequency
\begin{align*}
    \left\{\begin{array}{l}
         \partial_t w(t)=\big(A+\Q_nD\fk(v)\big)w+\Q_n\big(\fk(v)+\H(v, w)\big), \\
         \bC w(t)=-\bC v(t).
    \end{array}\right.
\end{align*}
Due to the observability condition $\bC w=-\bC v$, we can regard the \emph{low-frequency} part as an input for the \emph{high-frequency system}. In order to set up the high-frequency reconstruction, the first part of Assumption \ref{NLS:assumFholom} will allow us to justify that Duhamel's formula holds, including the potential $\Q_nD\fk(v)$ in the linear part of the equation. Then, further enforcing the uniform high-frequency observability Assumption \ref{NLS:assumC}, we will argue that a \emph{linear} reconstruction is possible. If we momentarily forget about the source term, this says that we can reconstruct $w(0)$ in terms of the observation of $w$, which is $\bC w=-\bC v$. This linearization will then allow us to tackle the nonlinear observability problem by means of an adequate fixed point argument at high-frequency for inputs belonging to a compact set. By using Duhamel's formula, this will yield a reconstruction operator $\mc{R}$, defined in some appropriate spaces, such that $\Q_n u=\mc{R}(\P_n u)$ for a large enough frequency threshold $n\in \N$.

Then, the solution $u$ can be represented as $u(t)=v(t)+\mc{R}(v)(t)$, where $v$ solves
\begin{align}\label{intro:eq-low-split}
    \partial_t v=Av+\P_n \fk(v+\mc{R}(v)).
\end{align}
To demonstrate that $t\in (0, T)\mapsto u(t)\in X^\sigma$ is analytic, the second part of Assumption \ref{NLS:assumFholom}, namely, that $\fk$ admits a holomorphic extension, and a compactness assumption on $u$ are essential. This will be achieved by establishing that $t\mapsto v(t)$ and $v\mapsto \mc{R}(v)$ are both analytic maps. If instead, we consider \eqref{intro:eq-low-split} as a differential equation on the space Banach space $C^0([0, T], \P_n X^\sigma)$, classical ODEs theory imply that $t\mapsto v(t)$ is as smooth as $\fk$, and therefore analytic. The uniform contraction principle further ensures that $\mc{R}$ depends analytically on $v$, from which the result follows.

In what follows, we develop these ideas towards the proof of the main theorem. Although most intermediate results will treat $v$ as a generic input, in the end we will consider a low-frequency input plus a parameter which is not necessarily low-frequency. For the sake of convenience, we keep this notation.

\subsection{Preliminaries} Let us make the following intermediary assumption on the nonlinearity.

\begin{assump}{2a}\label{assumF}
    The nonlinear map $\fk: X^\sigma\to X^\sigma$ belongs to $C^2(X^\sigma)$ and is bounded on the bounded sets of $X^\sigma$. Moreover, its differential map $D\fk: x\in X^\sigma\mapsto D\fk(x)\in \L(X^\sigma)$ is bounded and Lipschitz-continuous on the ball $\mathbb{B}_{4R_0}(X^\sigma)$, for some $R_0>0$.
\end{assump}

Let $s\in [0, T)$ and $n\in \N$. For $v\in C^0([s, T], X^\sigma)$ given, we consider the linear equation
\begin{align}\label{eq:ode-hf-1}
    \left\{\begin{array}{cc}
        \partial_t w=(A+\Q_n D\fk(v))w, & t\in (s, T], \\
        w(s)=w_s.& 
    \end{array}\right.
\end{align}
Our first task is to justify that, for any $w(s)\in \Q_n X^\sigma$, \eqref{eq:ode-hf-1} has a unique mild solution characterized by Duhamel's formula
\begin{align}\label{eq:duhamel-hf-1}
    w(t):=e^{(t-s)A}w_s+\int_{s}^t e^{(t-\tau)A}\Q_nD\fk(v(\tau))w(\tau)d\tau.
\end{align}
This will lead us to introduce a map $S_n: v\mapsto S_n(v)$, where $S_n(v)$ is the evolution operator, in the sense of \cite[Chapter 5, Definition 5.3]{Pazy82}, characterized by formula \eqref{eq:duhamel-hf-1} above as the unique mild solution of \eqref{eq:ode-hf-1} with $S_n(v)(\cdot, s)w_s=w(\cdot)$. We make this precise in the following lemma.

\begin{lemma}\label{lem:sv-properties}
    Let $R_0>0$ and $T>0$. Let $n\in \N$. Under Assumptions \ref{NLS:assumA} and \ref{assumF} (with $R_0$), for each $v\in C^0([s, T], X^\sigma)$, the map $(t, s)\mapsto S_n(v)(t, s)$ is a linear evolution operator for $0\leq s\leq t\leq T$ and Duhamel's formula \eqref{eq:duhamel-hf-1} holds. Moreover, for $s\in [0, T)$, the map
    \begin{align}\label{eq:sv-lip}
        S_n: v\in \B_{3R_0}^{[s, T]}(X^\sigma)\longmapsto S_n(v)(\cdot, s)\in \L(\Q_nX^\sigma, C^0([s, T], \Q_n X^\sigma)),
    \end{align}
    is uniformly bounded and Lipschitz-continuous of constant $C>0$ (independent of $n\in \N$): for any $v_1$, $v_2\in \B_{3R_0}^{[s, T]}(X^\sigma)$ it holds
    \begin{align}\label{ineq:sv-lip}
        \norm{S(v_1)(\cdot, s)-S(v_2)(\cdot, s)}_{\L(\Q_n X^\sigma, C^0([s, T], \Q_n X^\sigma))}\leq C\norm{v_1-v_2}_{C^0([s, T], X^\sigma)}.
    \end{align}
\end{lemma}
\begin{proof}
    Let $\F$ be the map given by
    \begin{align}\label{eq:Fduhamel}
        (\F w)(t)=e^{(t-s)A}w_s+\int_{s}^t e^{(t-\tau)A}\Q_nD\fk(v(\tau))w(\tau)d\tau,\ t\in [s, T]. 
    \end{align}
    From Assumption \ref{assumF} and $v\in C^0([s, T], X^\sigma)$, the map $\F$ is well defined from $C^0([s, T], \Q_nX^\sigma)$ into itself. Let $M$ be the uniform bound of the map $v\mapsto D\fk(v)$ on $\mathbb{B}_{4R_0}(X^\sigma)$. We now claim that for every $t\in [s, T]$ it holds
    \begin{align}\label{ineq:Fkbound}
        \norm{\F^k w-\F^k \widetilde{w}}_{C^0([s, t], \Q_nX^\sigma)}\leq \dfrac{(CM(t-s))^k}{k!}\norm{w-\widetilde{w}}_{C^0([s, t], \Q_nX^\sigma)},
    \end{align}
    for each $k\in \N$. From the Duhamel formula, we get that
    \begin{align*}
        \norm{\F w-\F \widetilde{w}}_{C^0([s, t], \Q_nX^\sigma)}\leq CM(t-s)\norm{w-\widetilde{w}}_{C^0([s, t], \Q_nX^\sigma)}.
    \end{align*}
    Assume that \eqref{ineq:Fkbound} is true for some $k\in\N$. Then, by writing $\F^{k+1} w-\F^{k+1} \widetilde{w}=\F(\F^k w)-\F(\F^k\widetilde{w})$, using Duhamel's formula \eqref{eq:Fduhamel} again
    \begin{align*}
        \norm{\F^{k+1} w-\F^{k+1} \widetilde{w}}_{C^0([s, t], \Q_nX^\sigma)}&\leq CM\int_s^t \norm{\F^k w-\F^k\widetilde{w}}_{C([s, \tau], \Q_n X^\sigma)}d\tau\\
        &\leq \dfrac{(CM)^{k+1}}{k!}\int_s^t (\tau-s)^k\norm{w-\widetilde{w}}_{C^0([s, t], \Q_nX^\sigma)}d\tau\\
        &\leq \dfrac{(CM(t-s))^{k+1}}{(k+1)!}\norm{w-\widetilde{w}}_{C^0([s, t], \Q_nX^\sigma)}.
    \end{align*}
    The claim follows by induction.
    
    Choosing $k\in \N$ such that $\tfrac{(CM(T-s))^k}{k!}<1$, by a generalization of the contraction principle, $\F$ admits a unique fixed point $w\in C^0([s, T], \Q_nX^\sigma)$ which satisfies \eqref{eq:duhamel-hf-1}. A classical Gronwall's lemma argument (used below) shows uniqueness of the solutions. This way, we introduce the map $S_n$ as precised in \eqref{eq:sv-lip} which maps each $v\in C^0([s, T], X^\sigma)$ into the evolution operator $S_n(v)$ characterized by \eqref{eq:duhamel-hf-1}.
    
    Let $w_s\in \Q_n X^\sigma$. By Duhamel's formula \eqref{eq:duhamel-hf-1}, for $s\leq t\leq T$ we have the estimate
    \begin{align*}
        \norm{w(t)}_{X^\sigma}\leq \norm{e^{(t-s)A}}_{\mc{L}(X^\sigma)}\norm{w_s}_{X^\sigma}+\int_{s}^t \norm{e^{(t-\tau)A}}_{\mc{L}(X^\sigma)}\norm{\Q_n D\fk(v(\tau))}_{\mc{L}(X^\sigma)}\norm{w(\tau)}_{X^\sigma}d\tau
    \end{align*}
    and so Gronwall's lemma implies
    \begin{align*}
        \norm{S_n(v)(t, s)w_s}_{X^\sigma}\leq C\norm{w_s}_{X^\sigma}\exp\left(C{\int_{s}^t \norm{D\fk(v(\tau))}_{\mc{L}(X^\sigma)}d\tau}\right).
    \end{align*}
    By Assumption \ref{assumF}, there exists $M=M(R_0)>0$ so that $\sup_{s\in [0, T]}\norm{D\fk(v(s))}_{\mc{L}(X^\sigma)}\leq M$, uniformly in $v\in \B_{3R_0}^{[0, T]}(X^\sigma)$. We then obtain
    \begin{align*}
        \norm{S_n(v)(\cdot, s)w_s}_{C^0([s, T], \Q_nX^\sigma)}\leq C\exp(CTM)\norm{w_s}_{X^\sigma},
    \end{align*}
    uniformly for $v\in \B_{3R_0}^{[s, T]}(X^\sigma)$. The linearity of the map $w_s\mapsto S_n(v)(\cdot, s)w_s$ implies that \eqref{eq:sv-lip} is bounded.
    
    Let $v_1$, $v_2\in \B_{3R_0}^{[s, T]}(X^\sigma)$ and set $w^i(t)=S(v_i, t, s)w_s$ for $w_s\in \Q_n X^\sigma$ for $i=1, 2$. In the same fashion as before, we employ Duhamel's formula to write
    \begin{align*}
        w^1(t)-w^2(t)&=(S(v_1)(t, s)-S(v_2)(t, s))w_s\\
        &\begin{multlined}=\int_{s}^t e^{(t-\tau)A}\Q_n[D\fk(v_1(\tau))-D\fk(v_2(\tau))]w^2(\tau)d\tau\\+\int_{s}^t e^{(t-\tau)A}\Q_nD\fk(v_1(\tau))(w^1(\tau)-w^2(\tau))d\tau.\end{multlined}
    \end{align*}
    Thus, by taking $X^\sigma$-norm and using the Lipschitz continuity of $D\fk$ given by Assumption \ref{assumF}, we get
    \begin{align*}
        \norm{w^1(t)-w^2(t)}_{X^\sigma}&\leq \begin{multlined}[t] C\int_{s}^t\big( \norm{D\fk(v_1(\tau))-D\fk(v_2(\tau))}_{\L(X^\sigma)}\norm{w^2(\tau)}_{X^\sigma}dt\\+\norm{D\fk(v_1(\tau))}_{\L(X^\sigma)}\norm{w^1(\tau)-w^2(\tau)}_{X^\sigma}\big)d\tau \end{multlined}\\
        &\leq \begin{multlined}[t]C(T-s)M\norm{w_s}_{X^\sigma}\norm{v_1-v_2}_{C^0([s, T], X^\sigma)}\\+C\int_{s}^t \norm{D\fk(v_1(\tau))}_{\L(X^\sigma)}\norm{w^1(\tau)-w^2(\tau)}_{X^\sigma}d\tau.\end{multlined}
    \end{align*}
    Gronwall's inequality implies that for any $w_s\in\Q_n X^\sigma$
    \begin{align*}
        \norm{(S(v_1)(t, s)-S(v_2)(t, s))w_s}_{X^\sigma}\leq C(T-s)M e^{C(T-s)M}\norm{w_s}_{X^\sigma}\norm{v_1-v_2}_{C^0([s, T], X^\sigma)}.
    \end{align*}
    Using that the previous bound is independent of $t\in [s, T]$, we deduce estimate \eqref{ineq:sv-lip} with $C$ depending on $M$, $R_0$ and $T-s$. 
\end{proof}

\begin{remark}
    The previous lemma shows that for any $T>s\geq 0$ and $v\in \B_{3R_0}^{[s, T]}(X^\sigma)$, \eqref{eq:ode-hf-1} admits a unique mild solution $w:=S_n(v)(\cdot, s)w_s\in C^0([s, T], \Q_n X^\sigma)$.
\end{remark}

Our task now is to justify that Duhamel's formula holds for the equation
\begin{align}\label{eq:ode-hf-2}
    \left\{\begin{array}{cc}
        \partial_t w(t)=(A+\Q_n D\fk(v(t)))w(t)+\Q_nh(t), & t\in (s, T], \\
        w(s)=w_s,& 
    \end{array}\right.
\end{align}
when $h\in L^1([s, T], X^\sigma)$ is an appropriate source term and $w_s\in \Q_n X^\sigma$.
\begin{remark}
    From now on, every time we refer to equations of the form \eqref{eq:ode-hf-2}, we will understand them in the sense of mild solutions (i.e. whose solutions are defined through Duhamel's formula). Indeed, since $v\in C^0([s, T], X^\sigma)$, even for regular data $w_s\in D(A)\cap \Q_n X^\sigma$ and right-hand side $h\in C^0([s, T], X^\sigma)$, we cannot ensure that \eqref{eq:ode-hf-2} admits a classical solution.
\end{remark}

\begin{lemma}\label{lem:duhamel-h}
    Let $n\in \N$. Let $v\in C^0([s, T], X^\sigma)$. If $w_s\in \Q_n X^\sigma$ and $h\in L^1([s, T], X^\sigma)$, then $w\in C^0([s, T], \Q_nX^\sigma)$ is a mild solution of \eqref{eq:ode-hf-2} if and only if
     \begin{align}\label{lem:eq:duhamel-rhs}
         w(t)=S_n(v)(t, s)w_s+\int_{s}^t S_n(v)(t, \tau)\Q_n h(\tau)d\tau,\ t\in [s, T].
     \end{align}
     Moreover, if $v\in \B_{3R_0}^{[0, T]}(X^\sigma)$, there exists $C>0$ (independent of $n$) such that
     \begin{align}\label{lem:eq:rhs-est}
         \norm{w(t)}_{C^0([s, T], X^\sigma)}\leq C\big(\norm{w_s}_{X^\sigma}+\norm{\Q_n h}_{L^1([s, T], X^\sigma)}\big).
     \end{align}
\end{lemma}
\begin{proof}
    Let $w$ be given by \eqref{lem:eq:duhamel-rhs}. We will develop the right-hand side of evolutionary formula \eqref{lem:eq:duhamel-rhs} to arrive to the classical Duhamel formulation of \eqref{eq:ode-hf-2}, proving that it is indeed the mild solution we are looking for.
    
    By definition, the evolution operator $S(v)(t, s)$ is a mild solution of the the homogeneous problem
    \begin{align*}
        \left\{\begin{array}{lc}
            \dfrac{d}{dt} \Psi(t)=\big(A+\Q_n D\fk(v(t))\big)\Psi(t),    &  \\
            \Psi(s)=I,     & 
        \end{array}\right.
    \end{align*}
    that is, it is characterized by
    \begin{align}\label{eq:operatorSv}
        S_n(v)(t, s)=e^{(t-s)A}+\int_{s}^t e^{(t-\tau)A}\Q_nD\fk(v(\tau))S_n(v)(\tau, s)d\tau
    \end{align}
    as a linear operator in $\mc{L}(X^\sigma)$. We can thus write
    \begin{align}\label{eq:wduhamelSv1}
        w(t)&=e^{(t-s)A}w_s+\int_{s}^t e^{(t-\tau)A}\Q_nD\fk(v(\tau))S_n(v)(\tau, s)w_sd\tau+\int_{s}^t S_n(v)(t, \tau)\Q_nh(\tau)d\tau.
    \end{align}    
    On the one hand, focusing on the forcing term above, by using formula \eqref{eq:operatorSv} we get
    \begin{multline*}
        \int_{s}^t S_n(v)(t, \tau)\Q_nh(\tau)d\tau=\int_s^t e^{(t-\tau)A}\Q_n h(\tau)d\tau\\+\int_s^t \int_{\tau}^t e^{(t-\eta)A}\Q_nD\fk(v(\eta))S_n(v)(\eta, \tau)\Q_n h(\tau)d\eta d\tau.
    \end{multline*}
    On the other hand, by using once again that $w$ is of the form \eqref{lem:eq:duhamel-rhs}, we see that
    \begin{multline*}
        \int_{s}^t e^{(t-\tau)A}\Q_nD\fk(v(\tau))w(\tau)d\tau=\int_s^t e^{(t-\tau)A}\Q_nD\fk(v(\tau))S_n(v)(\tau, s)w_sd\tau\\
        +\int_s^t e^{(t-\tau)A}\Q_nD\fk(v(\tau))\left(\int_s^\tau S_n(v)(\tau, \xi)\Q_nh(\xi)d\xi\right)d\tau.
    \end{multline*}
    Up to relabeling variables, by Fubini's theorem we see that the double integrals in the two last identities are equal. By putting them together we obtain
    \begin{multline*}
        \int_s^t e^{(t-\tau)A}\Q_nD\fk(v(\tau))S_n(v)(\tau, s)w_sds+\int_{s}^t S_n(v)(t, \tau)\Q_nh(\tau)d\tau\\
        =\int_{s}^t e^{(t-\tau)A}\Q_nD\fk(v(\tau))w(\tau)d\tau+\int_s^t e^{(t-\tau)A}\Q_n h(\tau)d\tau.
    \end{multline*}
    Plugging the above identity into \eqref{eq:wduhamelSv1}, we obtain that 
    \begin{align*}
        w(t)=e^{(t-s)A}w_s+\int_{s}^t e^{(t-\tau)A}\big(\Q_nD\fk(v(\tau))w(\tau)+\Q_nh(\tau)\big)d\tau,
    \end{align*}
    proving that $w$ is a mild solution of \eqref{eq:ode-hf-2}.

    Conversely, let $w$ be a mild solution of \eqref{eq:ode-hf-2} and let $\widetilde{w}$ be given by
    \begin{align*}
        \widetilde{w}(t):=S_n(v)(t, s)w_s+\int_{s}^t S_n(v)(t, \tau)\Q_n h(\tau)d\tau,\ t\in [s, T].
    \end{align*}
    By the above argument, we see that $\widetilde{w}$ is a mild solution of \eqref{eq:ode-hf-2} with initial data $w_s$ and source term $h$. Thus $\widetilde{w}=w$ by uniqueness of mild solutions, proving that both formulations are equivalent.
    
    Classically, estimate \eqref{lem:eq:rhs-est} follows from Duhamel's formula \eqref{lem:eq:duhamel-rhs} and the uniform bound of $S_n$ given by \cref{lem:sv-properties}.
\end{proof}

\begin{remark}
    After \cref{lem:sv-properties}, observe that the constant $C$ in \eqref{lem:eq:rhs-est} does not depend on $n\in \N$.
\end{remark}

Building up on Assumption \ref{assumF}, we will further assume that there exists $\delta>0$ so that $\fk$ can be extended holomorphically to $\mathbb{B}_{4R_0, 2\delta}(X_\mathbb{C}^\sigma)$ into $X_\mathbb{C}^\sigma$. This means that $\fk$ is holomorphic in the interior of this ball and continuous up to the boundary. Moreover, for some $L>0$, the following inequalities hold for any $z$, $z'\in \mathbb{B}_{4R_0, 2\delta}(X^\sigma)$:
    \begin{align*}
        \norm{\fk(z)}_{X_\mathbb{C}^\sigma}\leq L,\ \norm{\fk(z)-\fk(z')}_{X_\mathbb{C}^\sigma}\leq L\norm{z-z'}_{X_\mathbb{C}^\sigma}\ \text{ and }\ \norm{D\fk(z)-D\fk(z')}_{\mc{L}(X_\mathbb{C}^\sigma)}\leq L\norm{z-z'}_{X_\mathbb{C}^\sigma}.
    \end{align*}
This is precisely Assumption \ref{NLS:assumFholom}.

We have the following lemma for the validity of a holomorphic extension of the map $v\mapsto S_n(v)$.

\begin{lemma}\label{lem:sv-holomorphic}
    Let $T>0$ and $R_0>0$. Under Assumptions \ref{NLS:assumA} and \ref{NLS:assumFholom} (with $R_0$), there exists $\delta_0>0$ such that for every $n\in \N$ the map $S_n$ defined by \eqref{eq:sv-lip} can be holomorphically extended as
    \begin{align}\label{eq:Sv-extension}
        S_n^{\mathbb{C}}: v\in \B_{3R_0, \delta_0}^{[s, T]}(X^\sigma)\longmapsto S_n(v)\in \mc{L}_\mathbb{C}(\Q_nX_\mathbb{C}^\sigma, C^0([s, T], \Q_n X_\mathbb{C}^\sigma)).
    \end{align}
    Furthermore, such extension is bounded and Lipschitz-continuous of constant $C>0$ (independent of $n\in\N$): for any $v_1$, $v_2\in \B_{3R_0,\delta_0}^{[s, T]}(X^\sigma)$
    \begin{align}\label{ineq:Sv-extension-lip}
        \norm{S_n^{\mathbb{C}}(v_1)(\cdot, s)-S_n^{\mathbb{C}}(v_2)(\cdot, s)}_{\L(\Q_n X_\mathbb{C}^\sigma, C^0([s, T], \Q_n X_\mathbb{C}^\sigma))}\leq C\norm{v_1-v_2}_{C^0([s, T], X_\mathbb{C}^\sigma)}.
    \end{align}
\end{lemma}
\begin{proof}
    First, we show the existence of the complex extension as specified in \eqref{eq:Sv-extension} and then we will prove that this map is indeed holomorphic.
    \medskip
    \paragraph{\emph{Step 1. Complex extension.}}\ The first step is to prove that Duhamel's formula still makes sense in the complex setting, which will give us the desired complex extension. 
        
    First of all, for each $t\in [s, T]$ the semigroup map $e^{tA}$, as well as the projector $\Q_n$, are linear maps in $\mc{L}(X^\sigma)$, so they admit holomorphic extensions as linear maps in the complexification $X_\mathbb{C}^\sigma$, namely, these extensions belong to $\mc{L}_\mathbb{C}(X_\mathbb{C}^\sigma)$; see \cite[Theorem 3]{BS:71-polynomials}. Secondly, from Assumption \ref{NLS:assumFholom}, $\fk$ admits a holomorphic extension from $\mathbb{B}_{4R_0, 2\delta}(X_\mathbb{C}^\sigma)$ into $X_\mathbb{C}^\sigma$ and thus its complex differential $D\fk$, defined as 
    \begin{align*}
        D\fk(x)k:=\delta \fk(x; k)=\lim_{\mathbb{C}\ni s\to 0} \dfrac{1}{s}\big(\fk(x+sk)-\fk(x)\big).
    \end{align*}
    maps $\mathbb{B}_{4R_0, 2\delta}(X_\mathbb{C}^\sigma)$ into $\mc{L}_\mathbb{C}(X_\mathbb{C}^\sigma)$. We denote all of these extensions by the same letter as its non-complexified versions. 
    
    Observe that whenever $\delta_0\leq\delta$, for any $v\in \B_{3R_0, \delta_0}^{[s, T]}(X^\sigma)$, the linear map $\Q_n D\fk(v)$ is well-defined and belongs to $\mc{L}_{\mathbb{C}}(\Q_n X_\mathbb{C}^\sigma, C^0([s, T], X_\mathbb{C}^\sigma))$. We thus have a well-defined map
    \begin{align}\label{eq:Svextensionmap}
        S_n^{\mathbb{C}}: v\in \B_{3R_0, \delta_0}^{[s, T]}(X^\sigma)\mapsto S_n^{\mathbb{C}}(v)(\cdot, s)\in \mc{L}_{\mathbb{C}}(\Q_nX_\mathbb{C}^\sigma, C^0([s, T], \Q_n X_\mathbb{C}^\sigma))
    \end{align}
    characterized by    
    \begin{align}\label{eq:Svduhamel-hf-1}
        S_n^{\mathbb{C}}(v)(t, s)w_s:=e^{(t-s)A}w_s+\int_{s}^t e^{(t-\tau)A}\Q_nD\fk(v(\tau))S_n^{\mathbb{C}}(v)(\tau, s)w_sd\tau.
    \end{align}
    Indeed, by the discussion at the beginning, the right-hand side above makes sense and moreover, Assumption \ref{NLS:assumFholom} allows us to perform a similar fixed-point argument as in \cref{lem:sv-properties}, from which we get the Duhamel's formula characterization and the regularity claim. Furthermore, the same Gronwall-type argument allows us to establish the Lipschitz-continuity \eqref{ineq:Sv-extension-lip} and boundedness of $S_n^{\mathbb{C}}$, uniform in $n\in \N$. This map is characterized by the corresponding Duhamel's formula \eqref{eq:Svduhamel-hf-1} and it give us the mapping properties as stated in \eqref{eq:Svextensionmap}, that is, $S_n^{\mathbb{C}}(v)$ is $\mathbb{C}$-linear for each fixed $v$. Therefore it provides us with the desired complex extension, which we will denote by just $S_n$ from now on.
     
    \medskip 
    \paragraph{\emph{Step 2. Holomorphicity of the extension.}}\ We now prove that the complex extension $S_n$ defined above is holomorphic by showing that it is complex differentiable; see \cref{appendix:thm:equiv-holom}. Let us consider $t\in [s, T]\mapsto \xi^q(t):=\xi(t; q)\in \Q_n X_\mathbb{C}^\sigma$, mild solution of
    \begin{align*}
    \left\{\begin{array}{cc}
        \dfrac{d\xi^q(t)}{dt}=(A+\Q_nD\fk(q(t)))\xi^q(t), &~ t\in (s, T],\\
        \xi^q(s)=\xi_s,
        \end{array}\right.
    \end{align*}
    with $\xi_s\in  \Q_n X_\mathbb{C}^\sigma$ and $q\in C^0([s, T], X_\mathbb{C}^\sigma)$. By the previous discussion, we have $\xi^q\in C^0([s, T], \Q_n X_\mathbb{C}^\sigma)$ and it is characterized by
    \begin{align*}
        \xi^q(t)=S(q)(t, s)\xi_s=e^{tA}\xi_s+\int_{s}^t e^{(t-\tau)A}\Q_nD\fk(q(\tau))\xi^q(\tau)d\tau.
    \end{align*}
    We claim that the differential of $S$ at $v\in \Int{\B_{3R_0, \delta_0}^{[s, T]}(X^\sigma)}$ is 
    \begin{align*}
        dS_n(v): h\in C^0([s, T], X_\mathbb{C}^\sigma)\longmapsto dS_n(v) h\in \L(\Q_n X_\mathbb{C}^\sigma, C^0([s, T], \Q_n X_\mathbb{C}^\sigma)),
    \end{align*}
    where $(dS_n(v)h)(\cdot)\xi_s:=z(\cdot; h)\in C^0([0, T], \Q_n X_\mathbb{C}^\sigma)$ is characterized by
    \begin{align*}
        z(t; h)=\int_{s}^t S_n(v)(t, \tau)\Q_nD^2\fk(v(\tau))[h(\tau),\xi^v(\tau)]d\tau,
    \end{align*}
    with $\xi^v\in C^0([s, T], \Q_nX_\mathbb{C}^\sigma)$ as above with $\xi_s$ fixed and $D^2\fk(x): X_\mathbb{C}^\sigma\times X_\mathbb{C}^\sigma\to X_\mathbb{C}^\sigma$ being the bilinear map defined as
    \begin{align*}
        D^2\fk(x)[h,k]:=\lim_{\mathbb{C}\ni s\to 0} \dfrac{1}{s}\big(D\fk(x+sh)k-D\fk(x)k\big).
    \end{align*}
    To this end, let us consider the map
    \begin{align*}
        t\in [s, T]\mapsto \D(t)=\xi(t; v+h)-\xi(t; v)-z(t; h)\in \Q_nX_\mathbb{C}^\sigma.
    \end{align*}
    Using Duhamel's formula and some algebraic manipulation, we see that $\D$ is a mild solution of
    \begin{align*}
        \left\{\begin{array}{lc}
            \dfrac{d\D(t)}{dt}=\big(A+\Q_n D\fk(v)\big)\D(t)+\Q_n\big(D\fk(v+h)-D\fk(v)\big)\xi^{v+h}-\Q_nD^2\fk(v)[h,\xi^v],     &  \\
            \D(s)=0 & 
        \end{array}\right.
    \end{align*}
    that is, $\D$ is characterized by
    \begin{align*}
        \D(t)=\int_s^t S_n(v)(t, \tau)\Q_n[(D\fk(v+h)-D\fk(v)-D^2\fk(v))[h, \xi^{v+h}]+D^2\fk(v)[h,\xi^{v+h}-\xi^v]d\tau,
    \end{align*}
    for $t\in [s, T]$ and $\xi_s\in \Q_n X_\mathbb{C}^\sigma$. Thus, we have the bound
    \begin{multline*}
        \norm{\D}_{C^0([s, T], X_\mathbb{C}^\sigma)}\leq C\big(\norm{(D\fk(v+h)-D\fk(v)-D^2\fk(v))[h, \xi^{v+h}]}_{C^0([s, T], X_\mathbb{C}^\sigma)}\\+\norm{D^2\fk(v)[h,\xi^{v+h}-\xi^v]}_{C^0([s, T], X_\mathbb{C}^\sigma)}\big)
    \end{multline*}
    with $C=C(R_0, T-s, \delta)>0$. Fix $v\in \text{Int}(\B_{3R_0, \delta_0}^{[s, T]}(X^\sigma))$ and $\ell>0$ so that $B_{C^0([s,T],X_{\mathbb{C}}^\sigma))}(v,\ell)\subset \B_{3R_0, \delta_0}^{[s, T]}(X^\sigma))$. For all $h\in C^0([s, T], X_\mathbb{C}^\sigma)$ with $\norm{h}_{C^0([s, T], X_\mathbb{C}^\sigma)}\leq \ell/4$ and $t\in [s,T]$, we have $v(t)+h(t) \in \mathbb{B}_{R_{0},\delta_0}(X^{\sigma})$ (recall $\delta_0\leq \delta$). Being $\fk$ holomorphic, using a Taylor expansion (see \cref{appendix:prop:cauchy-form}) along with Cauchy estimates (see \cref{appendix:prop:cauchy-est}), we get the following bound, uniform in $t\in [s,T]$,
    \begin{multline}\label{prop:proof:frechet-1tfixed}
        \norm{\fk(v(t)+h(t))-\fk(v(t))-D\fk(v(t))h(t)}_{ X_\mathbb{C}^{\sigma}}\\
        \leq \dfrac{4L\norm{h(t)}_{X_\mathbb{C}^\sigma}^2}{\ell(\ell-2\norm{h(t)}_{X_\mathbb{C}^\sigma})}   \leq \dfrac{4L\norm{h}_{C^0([s, T], X_\mathbb{C}^\sigma)}^2}{\ell(\ell-2\norm{h}_{C^0([s, T], X_\mathbb{C}^\sigma)})}\leq \dfrac{8L\norm{h}_{C^0([s, T], X_\mathbb{C}^\sigma)}^2}{\ell^2}.
    \end{multline}
    The differential of $\fk$ can be easily extended from $C^0([s, T], X_\mathbb{C}^\sigma)$ into $C^0([s, T], X_\mathbb{C}^{\sigma})$ and remains $\mathbb{C}$-linear. Since, the previous estimate is uniform in $t\in [s,T]$, we can write
    \begin{align}\label{prop:proof:frechet-1}
        \norm{\fk(v+h)-\fk(v)-D\fk(v)h}_{C^0([s, T], X_\mathbb{C}^{\sigma})}\leq \dfrac{8L\norm{h}_{C^0([s, T], X_\mathbb{C}^\sigma)}^2}{\ell^2}.
    \end{align}
    This shows that $\fk$ is a holomorphic map from $C^0([s, T], X_\mathbb{C}^\sigma)$ into itself. In particular, $D\fk: C^0([s, T], X_\mathbb{C}^\sigma)\to \mc{L}_{\mathbb{C}}(C^0([s, T], X_\mathbb{C}^\sigma))$ is holomorphic as well and therefore, with a similar reasoning as above, we arrive to
    \begin{align*}
        \norm{D\fk(v+h)-D\fk(v)-D^2\fk(v)h}_{\mc{L}(C^0([s, T], X_\mathbb{C}^\sigma))}\leq \dfrac{8L}{\ell^2}\norm{h}_{C^0([s, T], X_\mathbb{C}^\sigma)}^2.
    \end{align*}
    On the other hand, note that $D^2\fk$ is well-defined and continuous as a map from $C^0([s, T], X_\mathbb{C}^\sigma)$ into\footnote{$\mc{L}^k (E, F)$ is defined as the space of $k$-linear forms from $E\times \ldots \times E$ into $F$.} $\mc{L}_{\mathbb{C}}^2(C^0([s, T], X_\mathbb{C}^\sigma), C^0([s, T], X_\mathbb{C}^\sigma))$. Therefore, by the Lispchitz-continuity of $v\mapsto S_n(v)$, we get
    \begin{align*}
        \norm{D^2\fk(v)[h,\xi^{v+h}-\xi^v]}_{C^0([s, T], X_\mathbb{C}^\sigma)}
        &\leq \norm{D^2\fk(v)}_{\mc{L}^2(C^0([s, T], X_\mathbb{C}^\sigma))}\norm{h}_{C^0([s, T], X_\mathbb{C}^\sigma)}\norm{\xi^{v+h}-\xi^v}_{C^0([s, T], X_\mathbb{C}^\sigma)}\\
        &\leq C\norm{h}_{C^0([s, T], X_\mathbb{C}^\sigma)}\norm{\xi^{v+h}-\xi^v}_{\mc{L}(\Q_n X_\mathbb{C}^\sigma, C^0([s, T], X_\mathbb{C}^\sigma))}\norm{\xi_s}_{X_\mathbb{C}^\sigma}\\
        &\leq C \norm{h}_{C^0([s, T], X_\mathbb{C}^\sigma)}^2\norm{\xi_s}_{X_\mathbb{C}^\sigma}.
    \end{align*}
    Here, $C>0$ is a constant given by the local boundedness of $D^2\fk$ around $v$ (see \cite[Proposition 8.6]{Muj86}). Gathering the above inequalities, we obtain, uniformly in a ball centered at $v$ of radius $\ell/4$,
    \begin{align*}
        \norm{\D}_{C^0([s, T], X_\mathbb{C}^\sigma)}\leq C\norm{h}_{C^0([s, T], X_\mathbb{C}^\sigma)}^2\norm{\xi_s}_{X_\mathbb{C}^\sigma}.
    \end{align*}
    In terms of the operator $S_n$, the previous inequality reads as
    \begin{align*}
        \norm{S_n(v+h)(\cdot, 0)\xi_s-S_n(v)(\cdot, 0)\xi_s-(dS_n(v)h)\xi_s}_{C^0([s, T], \Q_n X_\mathbb{C}^\sigma)}\leq C\norm{h}_{C^0([s, T], X_\mathbb{C}^\sigma)}^2\norm{\xi_s}_{X_\mathbb{C}^\sigma}.
    \end{align*}
    By linearity of $S_n(v)$ and $dS_n(v)h$, we finally get
    \begin{align*}
        \norm{S_n(v+h)(\cdot, 0)-S_n(v)(\cdot, 0)-dS_n(v)h}_{\L_{\mathbb{C}}(\Q_nX_\mathbb{C}^\sigma, C^0([s, T], \Q_n X_\mathbb{C}^\sigma))}\leq C\norm{h}_{C^0([s, T], X_\mathbb{C}^\sigma)}^2.
    \end{align*}
    The latter inequality shows that $S$ is complex differentiable, hence holomorphic. 
\end{proof}

\subsection{High-frequency linear observability problem} From now on, we assume Assumption \ref{NLS:assumC} holds: for $R_0>0$ and $n_0\in \N$ given, for an observation operator $\bC\in \mc{L}(X^\sigma)$ and for some non-empty set $\mathbb{V}\subseteq \B_{3R_0}^{[0, T]}(X^\sigma)$, there exists a constant $\mathfrak{C}_{obs}>0$ such that for any $n\geq n_0$ and all $v\in\mathbb{V}$, it holds
\begin{align}\label{assum:observabstract1}
    \norm{w_0}_{X^\sigma}^2\leq \mathfrak{C}_{\text{obs}}^2\int_0^T \norm{\bC S_n(v)(t, 0)w_0}_{X^\sigma}^2dt,\quad \forall w_0\in \Q_nX^\sigma.
\end{align}

For all $n\geq n_0$ and for any $v\in \B_{3R_0}^{[0, T]}(X^\sigma)$, let $\O_{n, v}\in \mc{L}(\Q_nX^\sigma, L^2([0, T], X^\sigma))$, defined by $\mc{O}_{n, v}:=\bC S_n(v)(\cdot, 0)$, be the observation operator of linear solutions at high-frequency with potential. Indeed, from \cref{lem:sv-properties}, we know that for $w_0\in \Q_n X^\sigma$, the map $t\in [0, T]\mapsto S_n(v)(t, 0)w_0\in \Q_n X^\sigma$ is continuous and in particular $\bC S_n(v)(\cdot, 0)w_0$ belongs to $L^2([0, T], X^\sigma)$.

From Assumption \ref{NLS:assumC}, observability inequality \eqref{assum:observabstract1} implies that for any $v\in \mathbb{V}$, the operator $\O_{n, v}$ is injective and it has closed range. This allows us to define $\Pi_{n, v}$ as the orthogonal projection\footnote{According to the natural scalar product in $L^2([0, T], X^\sigma)$.} onto its image $\Ima(\mc{O}_{n, v})\subset L^2([0, T], X^\sigma)$. From now on, we equip $\Y_{n, v}:=\Ima(\O_{n, v})$ with the induced topology from $L^2([0, T], X^\sigma)$ which makes it a Banach space. By the observability inequality \eqref{assum:observabstract1}, we know that $\O_{n, v}: \Q_n X^\sigma\to \Y_{n, v}$ is a bijection, ensuring that $\Y_{n, v}$ is closed and that a bounded reconstruction operator $\O_{n, v}^{-1}: \Y_{n, v}\to \Q_n X^\sigma$ exists.

By applying the observability inequality \eqref{assum:observabstract1}, we get for any $y\in \Y_{n, v} \subset L^2([0, T], X^\sigma)$,
\begin{align}\label{boundOn-1}
    \norm{\O_{n, v}^{-1}y}_{X^\sigma}&\leq \mathfrak{C}_{\text{obs}}\norm{\O_{n, v}\O_{n, v}^{-1}y}_{L^2([0, T], X^\sigma)}=\mathfrak{C}_{\text{obs}}\norm{y}_{L^2([0, T], X^\sigma)},
\end{align}
and this estimate is uniform in $v\in \mathbb{V}$ and $n\geq n_0$. In what follows, we give two main consequences of Assumption \ref{NLS:assumC}, which will be key to set up the reconstruction formula and study its regularity with respect to the parameter $v$. First, we introduce the Gramian-dependent operator
\begin{align*}
    \left\{\begin{array}{rcl}
        \G_n:  \B_{3R_0}^{[0, T]}(X^\sigma) &\longrightarrow & \mc{L}(\Q_n X^\sigma)  \\
            v &\longmapsto & \O_{n, v}^*\O_{n, v}. 
    \end{array}\right.
\end{align*}
Note that such operator is intrinsically related to the uniform observability Assumption \ref{NLS:assumC}, since
\begin{align*}
    \int_0^T \norm{\bC S_n(v)w_0}_{X^\sigma}^2dt=\inn{\O_{n, v}^*\O_{n, v}w_0, w_0}_{X^\sigma},\ \forall w_0\in \Q_n X^\sigma,
\end{align*}
where the adjoint is taken with respect to the real structure of $X^\sigma$. Hence, under Assumption \ref{NLS:assumC}, we can introduce the operator
\begin{align*}
    \left\{\begin{array}{rcl}
        \G_n^\dagger:  \mathbb{V}\subseteq\B_{3R_0}^{[0, T]}(X^\sigma) &\longrightarrow & \mc{L}(\Q_n X^\sigma)  \\
            v &\longmapsto & (\O_{n, v}^*\O_{n, v})^{-1}. 
    \end{array}\right.
\end{align*}
The next lemma establishes some regularity properties of this map.

\begin{lemma}\label{lem:inv-gramian-regularity}
    Let $T>0$, $R_0>0$, $n_0\in \N$ and $\mathbb{V}\subseteq \B_{3R_0}^{[0, T]}(X^\sigma)$. Under Assumptions \ref{NLS:assumA}, \ref{assumF} (with $R_0$) and \ref{NLS:assumC} (with $T$, $R_0$, $n_0$ and $\mathbb{V}$), for any $n\geq n_0$, the map $\G_n^\dagger: \mathbb{V}\subseteq \B_{3R_0}^{[0, T]}(X^\sigma)\to \mc{L}(\Q_n X^\sigma)$ is well-defined, bounded uniformly by $\mathfrak{C}_{obs}^2$ and Lipschitz-continuous with
    \begin{align*}
        \norm{\G_n^\dagger(v_1)-\G_n^\dagger(v_2)}_{\mc{L}(\Q_n X^\sigma)}\leq L\norm{v_1-v_2}_{C^0([0, T], X^\sigma)}
    \end{align*}
    where the Lipschitz constant $L$ is independent of $n\geq n_0$. Moreover, if Assumption \ref{NLS:assumFholom} is enforced and $\mc{V}_{K}\subset\mathbb{V}$ is a compact set in $C^0([0, T], X^\sigma)$, there exists $\eta>0$ such that for any $n\geq n_0$, the map $\G_n^\dagger$ restricted to $\mc{V}_{K}$ admits a holomorphic extension as
    \begin{align*}
        \G_n^\dagger: \mc{V}_{K}+\B_{\eta, \eta}^{[0, T]}(X^\sigma)\to \mc{L}_\mathbb{C}(\Q_n X_\mathbb{C}^\sigma)
    \end{align*}
    which is Lipschitz-continuous and uniformly bounded with respect to $n\geq n_0$.
\end{lemma}
\begin{proof}
    To slightly simplify notation, let us introduce
    \begin{align*}
        Y=\mc{L}(\Q_nX^\sigma, L^2([0, T], X^\sigma))\ \text{ and }\ \widetilde{Y}= \mc{L}(L^2([0, T], X^\sigma), \Q_n X^\sigma),
    \end{align*}
    where both of them are endowed with the natural operator norm. We also introduce the complexified versions $Y_\mathbb{C}$ and $\widetilde{Y}_\mathbb{C}$, obtained by replacing $X^\sigma$ by $X_\mathbb{C}^\sigma$ into the above spaces, respectively, and considering the corresponding space of $\mathbb{C}$-linear bounded operators.

    \medskip 
    \paragraph{\emph{Step 1. Boundedness.}}\  
    Let $v\in \mathbb{V}$. Since $\O_{n, v}$ is a bounded linear operator from $\Q_n X^\sigma$ into $L^2([0, T],  X^\sigma)$, so is $\O_{n, v}^*$ from $L^2([0, T], X^\sigma)$ into $\Q_n X^\sigma$; see \cite[Remark 16]{Brez11}. Under Assumption \ref{NLS:assumC}, observability inequality \eqref{assum:observabstract1} implies that $\O_{n, v}^*\O_{n, v}\in \mc{L}(\Q_nX^\sigma)$ is a bijection with bounded inverse, that is, $(\O_{n, v}^*\O_{n, v})^{-1}\in \mc{L}(\Q_nX^\sigma)$. We then have a well-defined map
    \begin{align*}
        \G_n^\dagger: v\in\mathbb{V}\subseteq\B_{3R_0}^{[0, T]}(X^\sigma)\mapsto \mathfrak{I}\G_n(v)=(\O_{n, v}^*\O_{n, v})^{-1}\in \mc{L}(\Q_n X^\sigma),
    \end{align*}
    where $\mathfrak{I}: \Isom(\Q_nX^\sigma)\to \Isom(\Q_nX^\sigma)$ is the map $\mathfrak{I}(L)=L^{-1}\in\Isom(\Q_nX^\sigma)$, with $\Isom(\Q_nX^\sigma)$ being the (open) subset of invertible operators in $\mc{L}(\Q_nX^\sigma)$.
    
    The boundedness of $\G_n^\dagger$ follows directly from Assumption \ref{NLS:assumC}. Indeed, as a consequence of the observability inequality \eqref{assum:observabstract1} we can write    
    \begin{align*}
        \dfrac{1}{\mathfrak{C}_{obs}^2}\norm{w_0}_{X^\sigma}^2\leq \int_0^T \norm{\O_{n, v}w_0}_{X^\sigma}^2dt=\inn{\O_{n, v}^*\O_{n, v}w_0, w_0}_{X^\sigma}\leq \norm{\O_{n, v}^*\O_{n, v}w_0}_{X^\sigma}\norm{w_0}_{X^\sigma},
    \end{align*}
    from which we obtain
    \begin{align}\label{ineq:bound-ovov}
        \norm{(\O_{n, v}^*\O_{n, v})^{-1}}_{\mc{L}(\Q_n X^\sigma)}=\sup\{\norm{\O_{n, v}^*\O_{n, v}w_0}_{\Q_n X^\sigma}^{-1}\ |\ w_0\in \Q_n X^\sigma,\ \norm{w_0}=1\}\leq \mathfrak{C}_{obs}^2,
    \end{align}
    and the latter constant is uniform on $v\in \mathbb{V}$ and $n\geq n_0$.
    
    \medskip 
    \paragraph{\emph{Step 2. Lipschitz-continuity.}}\
    To verify the Lipschitz continuity of $\G_n^\dagger$, we will show that it can be expressed as the composition of several Lipschitz-continuous maps. 
    
    In view of \cref{lem:sv-properties}, by composition with linear maps, we have that both maps
    \begin{align*}
    \begin{array}{l}
        v\in \B_{3R_0}^{[0, T]}(X^\sigma)\longmapsto \O_{n, v}=\bC S_n(v)\in Y,\\
        v\in \B_{3R_0}^{[0, T]}(X^\sigma)\longmapsto \O_{n, v}^*=(\bC S_n(v))^*\in \widetilde{Y}
    \end{array}
    \end{align*}
    are Lipschitz-continuous. Let us consider the following (nonlinear) operators:
    \begin{itemize}
        \item Let $\mathfrak{N}_n: v\mapsto (\O_{n, v}^*, \O_{n, v})$ be defined from $\B_{3R_0}^{[0, T]}(X^\sigma)$ into $\widetilde{Y}\times Y$. By \cref{lem:sv-properties} and composition of linear operators, we see that $\mathfrak{N}$ is bounded and Lipschitz-continuous map in the $C^0([0, T], X^\sigma)$-norm, uniformly in $n\in\N$. 
        \item The bilinear continuous form $\mathfrak{B}: (S, T)\mapsto ST$, as a map from $\widetilde{Y}\times Y$ into $\mc{L}(\Q_n X^\sigma)$. By direct computation, it is Lipschitz-continuous on bounded sets of $\widetilde{Y}\times Y$, with constant depending on the size of the set.
        \item Let $\mathfrak{I}$ be the map $\mathfrak{I}(L)=L^{-1}$ introduced above. Note that for any $L_1$, $L_2\in \Isom(\Q_n X^\sigma)$ we can write $L_2^{-1}-L_1^{-1}=L_1^{-1}(L_1-L_2)L_2^{-1}$. This means that $\mathfrak{I}$ is Lipschitz-continuous on subsets of $\Isom(\Q_n X^\sigma)$ on which linear maps have bounded inverse.

        Let $\mc{V}_n=\mathfrak{B}\mathfrak{N}_n(\mathbb{V})$, which is a non-empty subset of $\Isom(\Q_n X^\sigma)$. Note that any operator in $\mc{V}_n$ is of the form $\O_{n, v}^*\O_{n, v}$ with $v\in\mathbb{V}$ and whose inverse is uniformly bounded with respect to $v$ and $n\geq n_0$ by $\mathfrak{C}_{obs}^{2}$, see \eqref{ineq:bound-ovov}. Therefore, $\mathfrak{I}$ is Lipschitz-continuous in $\mc{V}_n$ with the $\mc{L}(\Q_nX^\sigma)$-topology with constant independent of $v\in \mathbb{V}$ and $n\geq n_0$.
    \end{itemize}
    Hence, we can write $\G_n^\dagger=\mathfrak{I}\mathfrak{B}\mathfrak{N}_n$ and it is well-defined from $\mathbb{V}\subset \B_{3R_0}^{[0, T]}(X^\sigma)$ into $\Isom(\Q_n X^\sigma)\subset\mc{L}(\Q_n X^\sigma)$ for $n\geq n_0$. Moreover, by composition, it is Lipschitz-continuous in the $C^0([0, T], X^\sigma)$ topology, uniformly in $n\geq n_0$.

    \medskip 

    \paragraph{\emph{Step 3. Holomorphic extension.}} From \cref{lem:sv-holomorphic}, we have a well-defined holomorphic extension
    \begin{align*}
        v\in \B_{3R_0, \delta_0}^{[0, T]}(X^\sigma)\longmapsto \O_{n, v}=\bC S_n(v)\in Y_\mathbb{C},
    \end{align*}
    where we dropped the upper script of $S_n^{\mathbb{C}}$ for simplicity and we keep denoting by $\bC$ the complexification of the linear observability map on $\mc{L}_\mathbb{C}(\Q_n X_\mathbb{C}^\sigma)$. Also, by \cref{app:lem:adjext} we have a well-defined holomorphic extension of the adjoint map
    \begin{align*}
        v\in \B_{3R_0, \delta_0}^{[0, T]}(X^\sigma)\longmapsto \widetilde{\adj}(\O_{n, v})\in \widetilde{Y}_\mathbb{C},
    \end{align*}
    where $\widetilde{\adj}$ denotes the $\mathbb{C}$-linear holomorphic extension of the map $\adj$ which sends a linear bounded operator into its adjoint. Let us then consider the following maps:
    \begin{itemize}
        \item Let $\mathfrak{N}_n: v\mapsto \big(\widetilde{\adj}(\O_{n, v}), \O_{n, v}\big)$ be defined from $\B_{3R_0, \delta_0}^{[0, T]}(X^\sigma)$ into $\widetilde{Y}_\mathbb{C}\times Y_\mathbb{C}$. By the previous arguments, we see that $\mathfrak{N}$ is bounded and holomorphic.
        \item The bilinear continuous form $\mathfrak{B}: (S, T)\mapsto ST$, as a map from $\widetilde{Y}_\mathbb{C}\times Y_\mathbb{C}$ into $\mc{L}_\mathbb{C}(\Q_n X_\mathbb{C}^\sigma)$. In particular, it is holomorphic, given that it is linear on each component \cref{appendix:thm:holom-sev-var}.
        \item Let $\mathfrak{I}: \Isom(\Q_nX_\mathbb{C}^\sigma)\to \Isom(\Q_nX_\mathbb{C}^\sigma)$ be the map $\mathfrak{I}(L)=L^{-1}\in\Isom(\Q_nX_\mathbb{C}^\sigma)$.
    \end{itemize}
    Since $\G_n:=\mathfrak{B}\mathfrak{N}_n$, by composition of holomorphic maps, we can consider the its holomorphic extension, which we keep denoting by the same letter,
    \begin{align*}
        \G_n: v\in\B_{3R_0, \delta_0}^{[0, T]}(X^\sigma)\mapsto \widetilde{\adj}(\O_{n, v})\O_{n, v}\in \mc{L}_\mathbb{C}(\Q_n X_\mathbb{C}^\sigma).
    \end{align*}
    Let us fix $\veps>0$ such that $\veps\mathfrak{C}_{obs}^2\leq \tfrac{1}{2}$. By continuity of the extension, for each $v_0\in\mc{V}_{K}$, there exists $\delta>0$ such that for $v\in v_0+\B_{\delta, \delta}^{[0, T]}(X^\sigma)$ we have $\norm{\G_n(v)-\G_n(v_0)}_{\mc{L}(X_\mathbb{C}^\sigma)}<\veps$ and, by a Neumann series argument, $\G_n(v)^{-1}$ is well-defined and belongs to $\mc{L}_\mathbb{C}(\Q_n X_\mathbb{C}^\sigma)$. Moreover, by \cref{app:lem:bij2} the extended map $\mathfrak{I}: \Isom(X_\mathbb{C}^\sigma)\to\Isom(X_\mathbb{C}^\sigma)$ is holomorphic on any of these $\veps$-neighborhoods around $\G(v_0)$ with $v_0\in\mc{V}_{K}$. Since these neighborhoods form an open cover of the compact set $\mc{V}_{K}+i0\subset \B_{3R_0, \delta_0}^{[0, T]}(X^\sigma)$, there exists $\eta>0$ such that $\G^\dagger$ admits a holomorphic extension as
    \begin{align*}
        \G_n^\dagger: v\in \mc{V}_{K}+\B_{\eta, \eta}^{[0, T]}(X^\sigma)\mapsto \mathfrak{I}\G_n(v)\in\mc{L}_\mathbb{C}(\Q_n X_\mathbb{C}^\sigma).
    \end{align*}
    Indeed, at each $v\in \mc{V}_{K}+\B_{\eta, \eta}^{[0, T]}(X^\sigma)$ we can find a suitable neighborhood on which the Neumann series expansion for $\G_n^\dagger$ is valid and thus the conclusion follows by composition of the holomorphic maps $\mathfrak{I}$ and $\G_n$. Moreover, by a Neumann series argument, for each $v\in \mc{V}_{K}+\B_{\eta, \eta}^{[0, T]}(X^\sigma)$,
    \begin{align*}
        \norm{\G_n^\dagger(v)}_{\mc{L}_{\mathbb{C}}(X_\mathbb{C}^\sigma)}&\leq \dfrac{1}{1-\norm{\G_n^\dagger(v_0)}_{\mc{L}(X^\sigma)}\norm{\G_n(v)-\G_n(v_0)}_{\mc{L}_\mathbb{C}(X_\mathbb{C}^\sigma)}}\norm{\G_n^\dagger(v_0)}_{\mc{L}(X^\sigma)}\\
        &\leq 2\mathfrak{C}_{obs}^2,
    \end{align*}
    where we used that $\norm{\G_n^\dagger(v_0)}_{\mc{L}(X^\sigma)}\leq \mathfrak{C}_{obs}^2$. This allows us to prove that $\G_n^\dagger$ is Lipschitz-continuous by following the same reasoning as in the real case.
\end{proof}

As a second consequence, we establish an explicit formula for the orthogonal projector $\Pi_{n, v}$.

\begin{lemma}\label{lem:proj}
    Let $T>0$, $R_0>0$, $n_0\in \N$ and $\mathbb{V}\subseteq \B_{3R_0}^{[0, T]}(X^\sigma)$. Under Assumptions \ref{NLS:assumA}, \ref{assumF} (with $R_0$) and \ref{NLS:assumC} (with $T$, $R_0$, $n_0$ and $\mathbb{V}$), for all $n\geq n_0$ and for any $v\in \mathbb{V}$, the orthogonal projector $\Pi_{n, v}: L^2([0, T], X^\sigma)\to L^2([0, T], X^\sigma)$ on $\Y_{n, v}=\Ima(\O_{n, v})$ is well-defined and $\Pi_{n, v}=\O_{n, v}(\O_{n, v}^*\O_{n, v})^{-1}\O_{n, v}^*$.
\end{lemma}

\begin{proof}    
    Let us consider $\widetilde{\Pi}_{n, v}=\O_{n, v}(\O_{n, v}^*\O_{n, v})^{-1}\O_{n, v}^*$. From \cref{lem:inv-gramian-regularity} this operator is well-defined as a map from $L^2([0, T], X^\sigma)$ into itself. We first note that it is the identity on $\mc{Y}_{n, v}$: if $y\in \mc{Y}_{n, v}$ then it can be written as $y=\O_{n, v}\zeta$ for some $\zeta\in L^2([0, T], X^\sigma)$ and therefore
    \begin{align*}
        \widetilde{\Pi}_{n, v}y=\Pi_{n, v}\O_{n, v}\zeta=\O_{n, v}(\O_{n, v}^*\O_{n, v})^{-1}(\O_{n, v}^*\O_{n, v})\zeta=\O_{n, v}\zeta=y.
    \end{align*}
    Additionally, by definition we note that $\mathrm{Im}(\widetilde{\Pi}_{n, v})\subset \Y_{n, v}$. Now, observe that $\widetilde{\Pi}_{n, v}$ is an orthogonal projection in $L^2([0, T], X^\sigma)$. Indeed, observe that
    \begin{align*}
        \widetilde{\Pi}_{n, v}^2=\O_{n, v}(\O_{n, v}^*\O_{n, v})^{-1}\O_{n, v}^*\circ \O_{n, v}(\O_{n, v}^*\O_{n, v})^{-1}\O_{n, v}^*=\O_{n, v}(\O_{n, v}^*\O_{n, v})^{-1}\O_{n, v}^*=\widetilde{\Pi}_{n, v}.
    \end{align*}
    By recalling that the operations taking adjoint and inverse commute, we arrive to
    \begin{align*}
        \inn{\widetilde{\Pi}_{n, v}\xi, \zeta}_{L^2([0, T], X^\sigma)}&=\inn{\O_{n, v}(\O_{n, v}^*\O_{n, v})^{-1}\O_{n, v}^*\xi, \zeta}_{L^2([0, T], X^\sigma)}\\
        &=\inn{\xi, \O_{n, v}(\O_{n, v}^*\O_{n, v})^{-1}\O_{n, v}^*\zeta}_{L^2([0, T], X^\sigma)}\\
        &=\inn{\xi, \widetilde{\Pi}_{n, v}\zeta}_{L^2([0, T], X^\sigma)}.
    \end{align*}
    By construction, $\widetilde{\Pi}_{n, v}$ is a map from $L^2([0, T], X^\sigma)$ into  itself whose image is $\Ima(\O_{n, v})=\Y_{n, v}$. Since $\Y_{n, v}$ is a linear subspace of $L^2([0, T], X^\sigma)$, by uniqueness of the orthogonal projection we conclude that $\Pi_{n, v}=\widetilde{\Pi}_{n, v}$.
\end{proof}

\begin{remark}
    For notational purposes, when needed we will also write $\Pi_{n, v}$ as the operator with value on $\Y_{n, v}$. That way, for instance, the operator $\O_{n, v}^{-1}\Pi_{n, v}$ is well-defined for any $v\in \mathbb{V}$.
\end{remark}

\subsubsection{Linear reconstruction} To ease notation, we consider the operator $\I_v(t): \xi\mapsto \int_0^t S_n(v)(t, s)\xi(s)ds$ and denote $\I_v(\cdot)$ when the operator is seen with value in $C^{0}([0,T],Y)$ for a suitable Banach space.

The above discussion will enable us to solve an \emph{observability Cauchy problem}, which is the content of the following Lemma.

\begin{lemma}\label{lem:CauchyObs}
    Let $R_0>0$, $T>0$, $n_0\in \N$ and $\mathbb{V}\subseteq\B_{3R_0}^{[0, T]}(X^\sigma)$. Under Assumptions \ref{NLS:assumA}, \ref{assumF} (with $R_0>0$) and \ref{NLS:assumC} (with $T$, $R_0$, $n_0$ and $\mathbb{V}$), there exists $C>0$ so that for any $n\geq n_0$, $v\in \mathbb{V}$, $g\in L^{2}([0,T],X^{\sigma})$ and $h\in L^{1}([0,T],X^{\sigma})$, there exists a unique $w\in C^{0}([0,T],\Q_{n}X^{\sigma})$ solution of 
    \begin{align}\label{eq:solprojectlin}
        \left\{\begin{array}{rl}
            \partial_t w&=(A+\Q_nD\fk(v))w+\Q_nh,  \\
       \Pi_{n, v}\bC w    &=\Pi_{n, v}g. 
        \end{array}\right.
    \end{align}
    It satisfies $w(0)=w_0:=\O_{n, v}^{-1}\Pi_{n, v}\left[g-\bC\I_v(\cdot)\Q_n h\right]$ and is given by $w(t)=S_n(v)(t,0)w_{0}+\I_v(t)\Q_nh$. We denote by $w:=\F_n(v)(g, h)$ the associated solution operator. Moreover, we have the estimate
    \begin{align}
    \label{ineq:estimFL}
    \norm{\F_{n}(v)(g, h)}_{C^{0}([0,T],\Q_{n}X^{\sigma})}\leq C \norm{\Pi_{n, v}g}_{L^{2}([0,T],X^{\sigma})}+C\norm{\Q_{n} h}_{L^{1}([0,T],X^{\sigma})}
    \end{align}
    uniformly in $v\in \mathbb{V}$ and $n\geq n_0$, and the map 
    \begin{align}\label{map:v-FL}
        v\in \mathbb{V}\mapsto \F_{n}(v)\in \mc{L}(L^2([0, T], X^\sigma)\times L^1([0, T], X^\sigma), C([0, T], \Q_n X^\sigma))
    \end{align}
    is Lipschitz-continuous and uniformly bounded with respect to $n\geq n_0$.
\end{lemma}

\begin{proof}
    In view of \cref{lem:duhamel-h}, to be solution of the first line of \eqref{eq:solprojectlin}, it is equivalent to be written as a Duhamel formula
    \begin{align}\label{eq:duhamel-lemma}
        w(t)=S_n(v)(t, 0)w_0+\int_{0}^t S_n(v)(t, s)\Q_n h(s)ds,
    \end{align}
    for some $w_0\in \Q_n X^\sigma$. So, we need to compute $w_0\in \Q_nX^\sigma$. Given the Duhamel's formula, we have $w\in C^0([0, T], \Q_n X^\sigma)\subset L^2([0, T], \Q_n X^\sigma)$ and we can compute
    \begin{align*}
        \Pi_{n, v}\bC w=\Pi_{n, v}\bC \left[ S(v,\cdot, 0)w_0\right]+\Pi_{n, v}\bC[\I_v(\cdot) \Q_n h].
    \end{align*}
    Observe that if $w_0\in \Q_n X^\sigma$, then $\bC[S(v,\cdot, 0)w_0]=\O_{n, v}w_0$, and therefore $\Pi_{n, v}\bC[S(v,\cdot, 0)w_0]=\Pi_{n, v}\O_{n, v}w_0=\O_{n, v}w_0$ by definition of $\Pi_{n, v}$.

    Since both belong to $\Y_{n, v}$ and we want $\Pi_{n, v}\bC w=\Pi_{n, v}g$, we should have
    \begin{align*}
        \O_{n, v}^{-1}\Pi_{n, v}g&=\O_{n, v}^{-1}\Pi_{n, v}\bC w\\
        &=\O_{n, v}^{-1}\bC S(v,\cdot, 0)w_0+\O_{n, v}^{-1}\Pi_{n, v}\bC[\I_v(\cdot) \Q_n h]\\
        &=w_0+\O_{n, v}^{-1}\Pi_{n, v}\bC[\I_v(\cdot) \Q_n h].
    \end{align*}
    With the initial data $w_0$ given by the above formula, \eqref{eq:duhamel-lemma} is satisfied as well as the observability condition $\Pi_{n, v}\bC w=\Pi_{n, v}g$. Indeed, it belongs to $\Q_n X^\sigma$ and by reproducing the same computation backwards, we have
    \begin{align*}
        \Pi_{n, v}\bC w&=\Pi_{n, v}\bC\left[S(v, \cdot, 0)w_0+\I_v(\cdot)\Q_nh\right] =\Pi_{n, v} \O_{n}w_0+\Pi_{n, v}\bC\I_v(\cdot)\Q_nh\\
        &=\Pi_{n, v} \O_{n}\O_n^{-1}\Pi_{n, v}\left[g-\bC\I_v(\cdot)\Q_n h\right]+\Pi_{n, v}\bC\I_v(\cdot)\Q_nh\\
        &=\Pi_{n, v}\left[g-\bC\I_v(\cdot)\Q_n h\right]+\Pi_{n, v}\bC\I_v(\cdot)\Q_nh=\Pi_{n, v}g.
    \end{align*}
    The uniqueness follows from the uniqueness of the definition of $w_0$.

    We now prove estimate \eqref{ineq:estimFL}. From \cref{lem:sv-properties}, we have the estimate
    \begin{align*}
        \norm{w}_{C([0, T], \Q_n X^\sigma)}\leq C\big(\norm{w_0}_{X^\sigma}+\norm{\Q_n h}_{L^1([0, T], X^\sigma)}\big),
    \end{align*}
    where $C$ is a constant uniform in $\mathbb{V}$ and $n\geq n_0\in \N$. Hence we are led to estimate $w_0$. As a consequence of Assumption \ref{NLS:assumC}, by using estimate \eqref{boundOn-1}, we get
    \begin{align*}
        \norm{w_{0}}_{X^{\sigma}}&= \norm{\O_{n, v}^{-1}\Pi_{n, v}\left[g-\bC\I_v(\cdot)\Q_n h\right]}_{X^{\sigma}}\\ &\leq \mathfrak{C}_{\text{obs}}\norm{\Pi_{n, v}g}_{L^2([0, T], X^\sigma)}+\mathfrak{C}_{\text{obs}}\norm{\Pi_{n, v}\bC\I_v(\cdot)\Q_n h}_{L^2([0, T], X^\sigma)}.
    \end{align*}
    We now use the unitarity of $\Pi_{n, v}$, H\"older's inequality and \cref{lem:sv-properties} to obtain
    \begin{align*}
        \norm{\Pi_{n, v}\bC\I_v(\cdot)\Q_n h}_{L^2([0, T], X^\sigma)}&\leq \norm{\bC\I_v(\cdot)\Q_n h}_{L^2([0, T], X^\sigma)}\\
        &\leq \norm{\bC}_{\mc{L}(X^\sigma)}\norm{\I_v(\cdot)\Q_n h}_{L^2([0, T], X^\sigma)}\\
        &\leq T^{1/2}\norm{\bC}_{\mc{L}(X^\sigma)}\norm{\I_v(\cdot)\Q_n h}_{L^{\infty}([0, T], X^\sigma)}\\
        &\leq C_1T^{1/2}\norm{\bC}_{\mc{L}(X^\sigma)}\norm{\Q_n h}_{L^{1}([0, T], X^\sigma)},
    \end{align*}
    uniformly in $v\in \mathbb{V}$ and $n\geq n_0\in \N$. Gathering the previous estimates, we arrive to
    \begin{align*}
    \norm{w}_{C^{0}([0,T],\Q_{n}X^{\sigma})}\leq C\mathfrak{C}_{\text{obs}} \norm{\Pi_{n, v}g}_{L^{2}([0,T], X^{\sigma})}+C\left(1+C_1T^{1/2}\mathfrak{C}_{\text{obs}}\norm{\bC}_{\mc{L}(X^\sigma)}\right)\norm{\Q_{n} h}_{L^{1}([0,T],X^{\sigma})}.
    \end{align*}    
    In particular, the nonlinear map
    \begin{align}\label{map:v-FL2}
        v\in \mathbb{V}\longmapsto \F_{n}(v)\in \mc{L}(L^2([0, T], X^\sigma)\times L^1([0, T], X^\sigma), C^0([0, T], \Q_n X^\sigma))
    \end{align}
    is well-defined and bounded, uniformly with respect to $n\geq n_0$. To prove that it is Lipschitz-continuous, we first use the explicit expression for the projector $\Pi_{n, v}$ given by \cref{lem:proj} to write
    \begin{align*}
        \F_n(v)(g, h)&=S_n(v)(\cdot, 0)\O_{n, v}^{-1}\Pi_{n, v}\left[g-\bC\I_v(\cdot)\Q_n h\right]+\I_v(\cdot)\Q_n h\\
        &=S_n(v)(\cdot, 0)\big(\O_{n, v}^*\O_{n, v}\big)^{-1}\O_{n, v}^*\left[g-\bC\I_v(\cdot)\Q_n h\right]+\I_v(\cdot)\Q_n h.
    \end{align*}
    From \cref{lem:sv-properties} and composition with linear maps, the map 
    \begin{align*}
        v\in \mathbb{V}\longmapsto \I_v\Q_n\in \mc{L}(L^1([0, T], X^\sigma), C^0([0, T], Q_nX^\sigma))
    \end{align*}
    is Lipschitz-continuous with constant uniform in $n\in \N$. Furthermore, combining \cref{lem:sv-properties} and \cref{lem:inv-gramian-regularity}, we obtain that the map
    \begin{align*}
        v\in \mathbb{V}\longmapsto S_n(v)(\cdot, 0)\big(\O_{n, v}^*\O_{n, v}\big)^{-1}\O_{n, v}^*\in \mc{L}(L^2([0, T], X^\sigma), C^0([0, T], \Q_n X^\sigma)
    \end{align*}
    is Lipschitz-continuous with constant uniform in $n\geq n_0$. Therefore, by composition of Lipschitz-continuous maps, the map \eqref{map:v-FL2} is Lipschitz-continuous as well
    \begin{align*}
        \norm{\F_n(v_1)-\F_n(v_2)}_{\mc{L}(L^2([0, T],X^\sigma)\times L^1([0, T],X^\sigma), C^0([0, T],X^\sigma))}\leq C\norm{v_1-v_2}_{C^0([0, T], X^\sigma)}
    \end{align*}
    where $C>0$ is a constant uniform in $n\geq n_0$.
\end{proof}

\begin{remark}\label{rk:cauchyObslipschitz}
    After \cref{lem:CauchyObs}, the fact that the operator $\F_n$ is bounded, implies that the map $(v, g, h)\mapsto \F_n(v)(g, v, h)$ is Lipschitz-continuous in the $C^0([0, T], X^\sigma)\times L^2([0, T],X^\sigma)\times L^1([0, T],X^\sigma)$-topology whenever the pair $(g, h)$ belongs to a bounded set of its corresponding space.
\end{remark}

\begin{lemma}\label{lem:Cauchyobsholom}
    Under the notation and Assumptions of \cref{lem:CauchyObs}, if we further enforce Assumption \ref{NLS:assumFholom}, and $\mc{V}_{K}\subset \mathbb{V}$ is a compact set in $C^0([0, T], X^\sigma)$, there exists $\eta>0$ such that the map $\F_n$ restricted to $\mc{V}_{K}$ admits a holomorphic extension as
    \begin{align}\label{def:FnC-holom}
        \F_n^{\mathbb{C}}: v\in \mc{V}_{K}+\B_{\eta, \eta}^{[0, T]}(X^\sigma)\longmapsto \F_n^{\mathbb{C}}(v)\in \mc{L}_{\mathbb{C}}(L^2([0, T],X_\mathbb{C}^\sigma)\times L^1([0, T],X_\mathbb{C}^\sigma), C^0([0, T], \Q_nX_\mathbb{C}^\sigma)).
    \end{align}
    This extension is Lipschitz-continuous and bounded uniformly with respect to $n\geq n_0$. Furthermore, there exists a constant $C>0$ such that for any $n\geq n_0$ and $v\in \mc{V}_{K}+\B_{\eta,\eta}(X^\sigma)$ the following estimate holds
    \begin{align}
    \label{estimFLholom}
    \norm{\F_n^{\mathbb{C}}(v)(g, h)}_{C^{0}([0,T],\Q_{n}X_\mathbb{C}^{\sigma})}\leq C \norm{\Pi_{n, v}g}_{L^{2}([0,T],X_\mathbb{C}^{\sigma})}+C\norm{\Q_{n} h}_{L^{1}([0,T],X_\mathbb{C}^{\sigma})},
    \end{align}
    for any $(g, h)\in L^2([0, T],X_\mathbb{C}^\sigma)\times L^1([0, T],X_\mathbb{C}^\sigma)$.
\end{lemma}
\begin{proof}
    After \cref{lem:CauchyObs}, we have an explicit formula for $\F_n$, that is, if $v\in \mathbb{V}$,
    \begin{align*}
        \F_n(v)(g, h)=S_n(v)(\cdot, 0)\big(\O_{n, v}^*\O_{n, v}\big)^{-1}\O_{n, v}^{*}\big[g-\bC\I_v\Q_nh\big]+\I_v(\cdot)\Q_nh,
    \end{align*}
    for any $(g, h)\in L^2([0, T],X^\sigma)\times L^1([0, T],X^\sigma)$, where we used the explicit formula for the projector $\Pi_{n, v}=\O_{n, v}(\O_{n, v}^*\O_{n, v})^{-1}\O_{n, v}^*$ given by \cref{lem:proj}.
    
    From \cref{lem:sv-holomorphic}, \cref{lem:inv-gramian-regularity} and \cref{app:lem:adjext}, there exists $\eta>0$ such that we have well-defined holomorphic maps
    \begin{align*}
    \begin{array}{rccl}
        S_n^\mathbb{C}: & v\in \mc{V}_K+\B_{\eta, \eta}^{[0, T]}(X^\sigma) & \longmapsto & S_n^\mathbb{C}(v)\in \mc{L}_{\mathbb{C}}(\Q_n X_\mathbb{C}^\sigma, C^0([0, T], \Q_n X_\mathbb{C}^\sigma)),\\
        \G_n^\dagger: & v\in \mc{V}_K+\B_{\eta, \eta}^{[0, T]}(X^\sigma) &\longmapsto & \G_n^\dagger(v)\in \mc{L}_{\mathbb{C}}(\Q_n X_\mathbb{C}^\sigma),\\
        \widetilde{\adj}: &v\in \mc{V}_K+\B_{\eta, \eta}^{[0, T]}(X^\sigma) & \longmapsto & \widetilde{\adj}(\O_{n, v})\in \mc{L}_{\mathbb{C}}(L^2([0, T], X_\mathbb{C}^\sigma), \Q_n X_\mathbb{C}^\sigma),
    \end{array}
    \end{align*}
    and on $\mc{V}_K$ they coincide with their respective real counterparts restricted to $\mc{V}_K$. Similarly, by linearity, we have the holomorphic map
    \begin{align*}
        v\in \mc{V}_K+\B_{\eta, \eta}^{[0, T]}(X^\sigma)\longmapsto \I_v^{\mathbb{C}}(\cdot)\Q_n\in \mc{L}_{\mathbb{C}}(L^1([0, T], X_\mathbb{C}^\sigma), C^0([0, T], \Q_n X_\mathbb{C}^\sigma)).
    \end{align*}
    characterized by $\I_v^{\mathbb{C}}(t): \xi\mapsto \int_0^t S_n^{\mathbb{C}}(v)(t, s)\xi(s)ds$. Also, $\bC$ can be extended by linearity into $\mc{L}_{\mathbb{C}}(\Q_n X_\mathbb{C}^\sigma)$. Following a similar argument as in the proof of \cref{lem:inv-gramian-regularity}, by composition, we have a well-defined holomorphic map $\F_n^{\mathbb{C}}$ as specified in \eqref{def:FnC-holom}, characterized by the formula
    \begin{align*}
        \F_n^{\mathbb{C}}(v)(g, h)=S_n^{\mathbb{C}}(v)(\cdot, 0)\G_n^\dagger(v)\widetilde{\adj}(\O_{n, v})\big[g-\bC\I_v^{\mathbb{C}}\Q_n h\big]+\I_v^{\mathbb{C}}(\cdot)\Q_n h.
    \end{align*}
    Given the above explicit formula and that all the maps involved in the definition are Lipschitz continuous and bounded, by composition, so is $\F_n^{\mathbb{C}}$ uniformly with respect to $n\geq n_0$. Moreover, by construction, on $\mc{V}_K$ the map $\F_n^\mathbb{C}$ coincides with $\F_n$ restricted to $\mc{V}_K$, which finishes the proof.
\end{proof}

\begin{remark}\label{rk:Fnholom}
    Note that we find a holomorphic extension of $\F_n$ restricted to $\mc{V}_{K}$ rather than on $\mathbb{V}$, where it is initially defined. Moreover, this extension does not necessarily solve an observability problem such as \eqref{eq:solprojectlin}.
\end{remark}

\subsubsection{Finite determining modes} As a first direct consequence of the previous result, we can get a finite determining mode result: two solutions of a nonlinear equation with the same observation and the same low frequency modes are the same. This result will not be used directly later, but can be considered as an easier version of what will follow where we will actually construct the reconstruction operator and study its regularity.

\begin{proposition}\label{prop:finite-det-modes}
Let $R_{0}>0$ and $\A\subset X^\sigma$ be a nonempty compact set. Under Assumptions \ref{NLS:assumA}, \ref{assumF} (with $R_0$) and \ref{NLS:assumC} (with $T$, $R_0$, $n_0$ and $\mathbb{V}=\B_{3R_0}^{[0, T]}(X^\sigma, \A)$), there exists $n\geq n_0$ such that the following holds. Let $\bh\in L^1([0,T],X^{\sigma})$ and $g\in L^2([0, T], X^\sigma)$.  Let $u(t)$ and $\widetilde{u}(t)$ be two solutions on $(0,T)$ of
\begin{align*}
    \left\{\begin{array}{cl}
        \partial_t u=Au+\fk(u)+\bh, &\ \text{on } (0, T), \\
        \bC u(t)=g(t), & \text{for } t\in (0, T),
    \end{array}\right.
\end{align*}
such that $u$, $\widetilde{u}\in \B_{R_0}^{[0, T]}(X^\sigma, \A)$. If $\P_n u(t)=\P_n \widetilde{u}(t)$ for all times $t\in [0, T]$, then $u(t)\equiv \widetilde{u}(t)$ for all $t\in [0, T]$.
\end{proposition}
\begin{proof}
    By assumption $\P_n u=\P_n\widetilde{u}$ as applications in $\B_{R_0}^{[0, T]}(X^\sigma)$. Let $z=u-\widetilde{u}$ and note that it solves
    \begin{align*}
        \left\{\begin{array}{cc}
            \partial_t z=Az+\fk(u)-\fk(\widetilde{u})&  \\
            \bC z=0.     & 
        \end{array}\right.
    \end{align*}
    If we linearize $\fk$ around $\widetilde{u}$, we can write
    \begin{align*}
        \fk(u)-\fk(\widetilde{u})=\fk(z+\widetilde{u})-\fk(\widetilde{u})=D\fk(\widetilde{u})z+\H(\widetilde{u}, z).
    \end{align*}
    By assumption $z=\Q_n(\widetilde{u}-u)$ and thus by applying $\Q_n$ to the equation satisfied by $z$, we get
    \begin{align*}
        \left\{\begin{array}{cc}
            \partial_t z=\big(A+\Q_nD\fk(\widetilde{u})\big)z+\Q_n\H(\widetilde{u}, z) &  \\
            \bC z=0     & 
        \end{array}\right.
    \end{align*}
    We are in the framework of \cref{lem:CauchyObs}, from which it follows $z=\F_n(\widetilde{u})\big(0, \H(\widetilde{u}, z)\big)$ along with the estimate
    \begin{align*}
        \norm{z}_{C^0([0, T], X^\sigma)}\leq C\norm{\Q_n\H(\widetilde{u}, z)}_{L^1([0, T], X^\sigma)}
    \end{align*}
    where $C>0$ is uniform in $n\geq n_0$. Moreover, using that $D\fk$ is Lipschitz-continuous we obtain
    \begin{align*}
        \norm{z}_{C^0([0, T], X^\sigma)}\leq C\norm{z}_{C^0([0, T], X^\sigma)}^2
    \end{align*}
    and thus whether $\norm{z}_{C^0([0, T], X^\sigma)}=0$ or $\norm{z}_{C^0([0, T], X^\sigma)}\geq 1/C$. In the former case, we are done. In the latter case, since $z=\Q_n(u-\widetilde{u})$ and both $u$, $\widetilde{u}$ belong pointwise in time to the same compact $\A\subset X^\sigma$, we can find $n\geq n_0$ such that $\norm{z}_{C^0([0, T], X^\sigma)}<1/C$, given that $C$ only depends on $n_0$. This yields $z=0$ and consequently $u=\widetilde{u}$.
\end{proof}

\subsection{High-frequency nonlinear reconstruction} Let $v\in \B_{3R_0}^{[0, T]}(X^\sigma)$. We are now interested in solving the nonlinear observability problem at high-frequency
\begin{align}
    \left\{\begin{array}{cc}
        \partial_t w(t)=Aw+\Q_n\fk(v+w)+\Q_n\bh &~ \text{ on } [0, T],  \\
        \Pi_{n, v}\bC w=\Pi_{n, v}g. &
\end{array}\right.
\end{align}
To this end, as explained at the beginning of this section, we will consider the linear variation of $v$ along $\Q_n D\fk(v)$ by writing
\begin{align*}
    \fk(w+v)=\fk(v)+\int_0^1 D\fk(v+\tau w)wd\tau=D\fk(v)w+\fk(v)+\H(v, w).
\end{align*}
This lead us to study the system
\begin{align}\label{eq:hf-nlw-obs}
    \left\{\begin{array}{cc}
        \partial_t w(t)=\big(A+\Q_nD\fk(v)\big)w+\Q_n\big(\fk(v)+\H(v, w)\big)+\Q_n\bh &  \\
        \Pi_{n, v}\bC w=\Pi_{n, v}g. & 
    \end{array}\right.
\end{align}
Under the assumptions of \cref{lem:CauchyObs}, it suggests that the initial condition of the above system must be given by
\begin{align*}
    w_0=\O_{n, v}^{-1}\Pi_{n, v}\big[g-\bC\I_v\Q_n(\fk(v)+\H(v, w)+\bh)\big].
\end{align*}
This yields the nonlinear operator $\Phi_{n, v, \bh, g}: C^0([0, T], \Q_n X^\sigma)\to C^0([0, T], \Q_n X^\sigma)$ defined by
\begin{align}\label{eq:hf-nlw-duhamel}
    \Phi_{n, v, \bh, g}(w)=\F_{n}(v)\big(g, \fk(v)+\H(v, w)+\bh\big).
\end{align}
We are then led to seek for fixed points of the operator $\Phi_{n, v, \bh, g}$. Indeed, as a consequence of the above formula, a fixed point of $\Phi_{n, v, \bh, g}$ is a solution of \eqref{eq:hf-nlw-obs}.

The aim of this section is to prove the following reconstruction result.

\begin{proposition}\label{prop:hf-unique-sol}(High-frequency reconstruction map) 
    Let $T>0$, $R_0>0$, $R_1>0$ and $n_0\in \N$. Let $\A_1$ and $\A_2$ be two nonempty compact subsets of $X^\sigma$ which are stable under the projection $\P_n$. Assume that Assumptions \ref{NLS:assumA}, \ref{assumF} (with $R_0$) and \ref{NLS:assumC} (with $T$, $R_0$, $n_0$ and $\mathbb{V}=\B_{3R_0}^{[0, T]}(X^\sigma, \A_1)$) hold. There exist $n^*\geq n_0$, $\eta>0$ and $0<R<R_0$ such that, for any $n\geq n^*$, for any $v\in \B_{3R_0}^{[0, T]}(X^\sigma, \A_1)$, $\bh\in \B_{R_1}^{[0, T]}(X^\sigma, \A_2)$ and $g\in \mathbb{B}_{\eta}(L^2([0, T], X^\sigma))$, there exists a unique solution $w\in \B_{R}^{[0, T]}(\Q_n X^\sigma)$ of
    \begin{align}\label{eq:w-hf-fixed-point}
        \left\{\begin{array}{cc}
            \partial_t w(t)=Aw+\Q_n\fk(v+w)+\Q_n\bh, &  \\
            \Pi_{n, v}\bC w=\Pi_{n, v}g. & 
        \end{array}\right.
    \end{align}
    This defines a nonlinear Lipschitz reconstruction operator
    \begin{align}\label{eq:mapRreal}
    \boldsymbol{\mathrm{R}}: \left\{\begin{array}{rcl}
        \B_{3R_0}^{[0, T]}(X^\sigma, \A_1)\times\B_{R_1}^{[0, T]}(X^\sigma, \A_2) \times  \mathbb{B}_{\eta}(L^{2}([0,T],X^{\sigma})) & \longrightarrow& \B_{R}^{[0, T]}(\Q_n X^\sigma)\\
        (v,\bh,g)&\longmapsto& w:=\mathcal{R}(v,\bh,g).
    \end{array}\right.
    \end{align}
    Furthermore, if additionally, $\fk$ satisfies Assumption \ref{NLS:assumFholom} and $\mc{V}_{K}\subset\B_{3R_0}^{[0, T]}(X^\sigma, \A_1)$ is nonempty and compact in $C^0([0, T], X^\sigma)$, then there exist $\eta>0$ and $\eta_1>0$, so that the map $\boldsymbol{\mathrm{R}}$ restricted to $\mc{V}_{K}$ in the first variable extends holomorphically as
    \begin{multline}\label{eq:mapRcomplex}
    \boldsymbol{\mathrm{R}}: \big(\mc{V}_{K}+\B_{\eta, \eta}^{[0, T]}(X^\sigma)\big)\times \big(\B_{R_1}^{[0, T]}(X^\sigma, \A_2)+\B_{\eta, \eta}^{[0, T]}(X^\sigma)\big)\\ \times  \mathbb{B}_{\eta, \eta}(L^{2}([0,T],X_\mathbb{C}^{\sigma}))
    \longrightarrow \B_{R, \eta_1}^{[0, T]}(\Q_n X^\sigma).
    \end{multline}
\end{proposition}

\begin{remark}
    The compactness in space is related to the existence of the reconstruction operator, whereas the compactness in time-space allows us to find a holomorphic extension uniformly with respect to the input acting trough the nonlinearity.
\end{remark}

\begin{remark}
    Similar to \cref{rk:Fnholom}, we find a holomorphic extension of the reconstruction operator $\boldsymbol{\mathrm{R}}$ when its first variable is restricted to a compact subset $\mc{V}_{K}$ of $\B_{3R_0}^{[0, T]}(X^\sigma, \A)$, where it was initially defined. This suggests that such extension is not necessarily a reconstruction operator anymore. However, to not over complicate the notation we keep the same letter to denote this extension.
\end{remark}

\begin{proof}[Proof of \cref{prop:hf-unique-sol}] In view of \cref{lem:CauchyObs}, we are looking for $w$ solution of
    \begin{align*}
        w=\F_{n}(v)\big(g, \fk(v)+\H(v, w)+\bh\big).
    \end{align*}
    We thus consider the nonlinear operator $\Phi$ defined by
    \begin{align*}
        \Phi_{n, v, \bh, g}(w)=\F_{n}(v)\big(g, \fk(v)+\H(v, w)+\bh\big).
    \end{align*}
    We will prove that it is well-defined as a map from $C^0([0, T], \Q_n X^\sigma)$ into itself and that it has a fixed point in a small ball, which also satisfies \eqref{eq:w-hf-fixed-point}. Without keeping track of the dependence on the parameters, we simplify the notation by writing $\Phi_v:=\Phi_{n, v, \bh, g}$. Also, for simplicity we will assume that $\A_1$ and $\A_2$ are the same compact set, which we simply denote by $\A$. This simplification is harmless in the proof and only serves to simplify the notation.
    \medskip 
    \paragraph{\emph{Step 1. Fixed point: real-valued case.}}\ We will prove $\Phi$ is a contraction in the ball $\B_R^{[0, T]}(\Q_n X^\sigma)$, for some $R>0$ to be specified later.
    
    Let $w\in \B_R^{[0, T]}(\Q_n X^\sigma)$. By using estimate \eqref{ineq:estimFL} of \cref{lem:CauchyObs}, we get
    \begin{align*}
        \norm{\Phi_v(w)}_{C^0([0, T], \Q_n X^\sigma)}&=\norm{\F_{n}(v)\big(g, \fk(v)+\H(v, w)+\bh\big)}_{C^0([0, T], \Q_n X^\sigma)}\\
        &\leq C\norm{\Pi_{n, v}g}_{L^2([0, T], X^\sigma)}+C\norm{\Q_n(\fk(v)+\H(v, w)+\bh)}_{L^1([0, T], X^\sigma)},
    \end{align*}
    for a constant $C$ uniform in $n\geq n_0\in \N$ and $v\in \B_{3R_0}^{[0, T]}(X^\sigma, \A)$.
    
    To estimate the first term we note that, for any $v\in \B_{3R_0}^{[0, T]}(X^\sigma, \A)$, $\Pi_{n, v}$ is a projection on $L^2([0, T], X^\sigma)$, and thus for $g\in \mathbb{B}_\eta(L^2([0, T], X^\sigma))$ we have
    \begin{align*}
        \norm{\Pi_{n, v}g}_{L^2([0, T], X^\sigma)}\leq \norm{g}_{L^2([0, T], X^\sigma)}\leq \eta.
    \end{align*}
    For the second term, under Assumption \ref{assumF}, the differential $D\fk$ is Lipschitz of constant $L$, thus
    \begin{align*}
        \norm{\Q_n\H(v, w)}_{L^1([0, T], X^\sigma)}&\leq \int_0^T\left(\int_0^1 \norm{D\fk(v+\tau w)-D\fk(v)}_{\L(X^\sigma)}d\tau\right) \norm{w}_{X^\sigma} dt\\
        &\leq TL\norm{w}_{C^0([0, T], X^\sigma)}^2.
    \end{align*}
    Since $\fk: X^\sigma\to X^\sigma$ is continuous and $\A$ is compact in $X^\sigma$, we have that $\fk(\A)$ is compact in $X^\sigma$. In particular, we observe that $\{\fk(v(t))\ |\ v\in \B_{3R_0}^{[0, T]}(X^\sigma, \A),\ t\in [0, T]\}\subset \fk(\A)$. Since the sequence of high-frequency projectors $(\Q_n)_n$ converge pointwise to $0$, the compactness of $\fk(\A)$ gives us that, for any $\eta_0>0$, we can find $n^*\geq n_0$ so that for any $v\in \B_{3R_0}^{[0, T]}(X^\sigma, \A)$ we have $\norm{\Q_n \fk(v)}_{L^1([0, T], X^\sigma)}\leq \eta_0/2$ for all $n\geq n^*$. Similarly, we get that $\norm{\Q_n \bh}_{L^1([0, T], X^\sigma)}\leq \eta_0/2$ for all $n\geq n^*$. Gathering the previous estimates, we obtain 
    \begin{align}\label{ineq:hf-fp-1}
        \norm{\Phi_v(w)}_{C^0([0, T], \Q_n X^\sigma)}&\leq C(\eta+\eta_0)+CTLR^2.
    \end{align}
    Concerning the difference, for $w_1$, $w_2\in \B_R^{[0, T]}(\Q_nX^\sigma)$, we observe that
    \begin{align*}
        \norm{\Phi_v(w_1)-\Phi_v(w_2)}_{C^0([0, T], \Q_nX^\sigma)}&=\norm{\F_{n}(v)(0, \H(v, w_1)-\H(v, w_2))}_{C^0([0, T], \Q_nX^\sigma)}\\
        &\leq C\norm{\H(v, w_1)-\H(v, w_2)}_{L^1([0, T], X^\sigma)}.
    \end{align*}
    A simple computation gives us
    \begin{multline}\label{eq:hf-Hvw}
        \H(v, w_1)-\H(v, w_2)\\=\int_0^1 \big([D\fk(v+\tau w_1)-D\fk(v)](w_1-w_2)-[D\fk(v+\tau w_1)-D\fk(v+\tau w_2)]w_2\big)d\tau.
    \end{multline}
    Thus
    \begin{align*}
        \norm{\H(v, w_1)-\H(v, w_2)}_{L^1([0, T], X^\sigma)}\leq CL\big(\norm{w_1}_{C^0([0, T], X^\sigma)}+\norm{w_2}_{C^0([0, T], X^\sigma)}\big)\norm{w_1-w_2}_{C^0([0, T], X^\sigma)}.
    \end{align*}
    We deduce
    \begin{align}\label{ineq:hf-fp-2}
        \norm{\Phi_v(w_1)-\Phi_v(w_2)}_{C^0([0, T], \Q_nX^\sigma)}\leq 2R CL\norm{w_1-w_2}_{C^0([0, T], X^\sigma)}.
    \end{align}
    
    In view of inequalities \eqref{ineq:hf-fp-1} and \eqref{ineq:hf-fp-2}, if we choose $R=\min\{\tfrac{1}{4CL}, \tfrac{1}{2TL}, R_0\}$ and $\eta_0$ small enough, then there exist $\eta>0$ and $n\geq n^*$ so that $\Phi_v$ reproduces the ball $\B_R^{[0, T]}(\Q_n X^\sigma)$ and is contracting in such set.
    
    We have then a well-defined reconstruction map
    \begin{align*}
        \boldsymbol{\mathrm{R}}: (v, \bh, g)\in \B_{3R_0}^{[0, T]}(X^\sigma,\A)\times\B_{R_1}^{[0, T]}(X^\sigma,\A)\times \mathbb{B}_\eta(L^2([0, T], X^\sigma))\longmapsto w\in \B_{R}^{[0, T]}(\Q_n X^\sigma).
    \end{align*}
    The fact that $\boldsymbol{\mathrm{R}}$ is Lipschitz, follows from composition of Lipschitz maps. As we did to get inequality \eqref{ineq:hf-fp-2}, under Assumption \ref{assumF}, we get that
    \begin{align*}
        \norm{\H(v_1, w_1)-\H(v_2, w_2)}_{L^1([0, T], X^\sigma)}\leq 2R CL\norm{w_1-w_2}_{C^0([0, T], X^\sigma)}+2R_0CL\norm{v_1-v_2}_{C^0([0, T], X^\sigma)}.
    \end{align*}
    Thus, $(v, w)\mapsto \fk(v)+\H(v, w)$ is a Lipschitz-continuous map. Using \cref{lem:CauchyObs} and \cref{rk:cauchyObslipschitz}, by linearity in the variables $\bh$ and $g$ and by composition of Lipschitz maps, we get
    \begin{multline*}
        \norm{\Phi_{n, v_1, \bh_1, g_1}(w_1)-\Phi_{n, v_2, \bh_2, g_2}(w_2)}_{C^0([0, T], X^\sigma)}\leq CR \norm{w_1-w_2}_{C^0([0, T], X^\sigma)}\\
        +C\big(\norm{v_1-v_2}_{C^0([0, T], X^\sigma)}+\norm{\bh_1-\bh_2}_{C^0([0, T], X^\sigma)}+\norm{g_1-g_2}_{L^2([0, T], X^\sigma)}\big).
    \end{multline*}
    Since the fixed points are given by $\Phi_{n, v_i, \bh_i, g_i}(w_i)=w_i$ for $i=1, 2$, up to making $R$ smaller if necessary, we get that $\boldsymbol{\mathrm{R}}$ is Lipschitz-continuous.
    \medskip 
    \paragraph{\emph{Step 2. Complex extension.}} Under Assumption \ref{NLS:assumFholom}, by \cref{lem:Cauchyobsholom}, there exists $\eta>0$ such that we have a well-defined holomorphic map $\F_n^{\mathbb{C}}$, extension of the map $\F_n$ restricted to $\mc{V}_{K}$, that is,
    \begin{align}
        v\in \mc{V}_{K}+\B_{\eta, \eta}^{[0, T]}(X^\sigma) \longmapsto \F_{n}^{\mathbb{C}}(v)\in \mc{L}_{\mathbb{C}}(L^2([0, T], X_\mathbb{C}^\sigma)\times L^1([0, T], X_\mathbb{C}^\sigma), C^0([0, T], \Q_n X_\mathbb{C}^\sigma)).
    \end{align}
    Let $\eta_1>0$ to be fixed later. Let us consider $\Phi_n^{\mathbb{C}}$ given by
    \begin{align*}
        \Phi_n^{\mathbb{C}}(\bh, g, v, w):=\F_n^{\mathbb{C}}(v)\big(g, \fk(v)+\H(v, w)+\bh\big)
    \end{align*}
    for
    \begin{align*}
        v\in\mc{V}_{K}+\B_{\eta, \eta}^{[0, T]}(X^\sigma),\bh\in \B_{R_1}^{[0, T]}(X^\sigma, \A)+\B_{\eta, \eta}^{[0, T]}(X^\sigma),\  g\in \mathbb{B}_{\eta, \eta}(L^2([0, T], X^\sigma)),\ w\in \B_{R, \eta_1}^{[0, T]}(\Q_n X^\sigma).
    \end{align*}
    We need to verify that this map is well-defined. By linearity of $\F_n^\mathbb{C}(v)$ with respect to $g$ and $\bh$, as well as by the linearity of $\Q_n$, it only remains to check that the map $(v, w)\mapsto \fk(v)+\H(v, w)$ is well-defined. By Assumption \ref{NLS:assumFholom}, up to making constants smaller, as long as $\eta+\eta_1\leq 3\delta/2$ and $R+\eta\leq 3R_0/4$, for $v\in\mc{V}_{K}+\B_{\eta, \eta}^{[0, T]}(X^\sigma)$ and $w\in \B_{R, \eta_1}^{[0, T]}(\Q_n X^\sigma)$, we have $v+\tau w\in \B_{3R_0+\eta+R, \eta+\eta_1}^{[0, T]}(X^\sigma)\subset \B_{4R_0, 2\delta}^{[0, T]}(X^\sigma)$ for any $\tau\in [0, 1]$. Therefore, the map $(v, w)\mapsto \fk(v)+\H(v, w)$ is well-defined from $\mc{V}_{K}+\B_{\eta, \eta}^{[0, T]}(X^\sigma)\times\B_{R, \eta_1}^{[0, T]}(\Q_n X^\sigma)$ into $C^0([0, T], X_\mathbb{C}^\sigma)$, and thus into $L^1([0, T], X_\mathbb{C}^\sigma)$ as well.

    Let $v=v_1+z$ with $v_1\in\B_{3R_0}^{[0, T]}(X^\sigma, \A)$ and $z\in \B_{\eta, \eta}^{[0, T]}(X^\sigma)$. By employing \cref{lem:inv-gramian-regularity} and \cref{lem:proj}, there exists $C>0$ independent of $v$ and $n\geq n_0$ such that
    \begin{align*}
        \norm{\Pi_{n, v}g}_{L^2([0, T], X_\mathbb{C}^\sigma)}\leq C\eta.
    \end{align*}
    For $\bh\in \B_{R_1}^{[0, T]}(X^\sigma, \A)+\B_{\eta, \eta}^{[0, T]}(X^\sigma)$, using the compactness of $\A$ to control the real part, we have for $n\geq n^*$
    \begin{align*}
        \norm{\Q_n \bh}_{L^1([0, T], X_\mathbb{C}^\sigma)}\leq \dfrac{\eta_0}{2}+T\eta
    \end{align*}
    A Taylor development around $v_1$ allows us to write $\fk(v)=\fk(v_1)+\int_0^1 D\fk(v_1+\tau z)zd\tau$ and thus
    \begin{align*}
        \norm{\Q_n \fk(v)}_{L^1([0, T], X_\mathbb{C}^\sigma)}&\leq \norm{\Q_n \fk(v_1)}_{L^1([0, T], X_\mathbb{C}^\sigma)}+M\norm{z}_{L^1([0, T], X_\mathbb{C}^\sigma)}\\
        &\leq \dfrac{\eta_0}{2}+2TM\eta.
    \end{align*}
    For the last term, we have
    \begin{align*}
        \norm{\Q_n\H(v, w)}_{L^1([0, T], X_\mathbb{C}^\sigma)}&\leq \int_0^T\left(\int_0^1 \norm{D\fk(v+\tau w)-D\fk(v)}_{\L(X_\mathbb{C}^\sigma)}d\tau\right) \norm{w}_{X_\mathbb{C}^\sigma} dt\\
        &\leq TL\norm{w}_{C^0([0, T], X_\mathbb{C}^\sigma)}^2\leq TL(R+\eta_1)^2.
    \end{align*}
    Let us remark that the above estimates are all uniform with respect to $n\geq n^*$.

    As we did in the real-valued case, but instead under Assumption \ref{NLS:assumFholom}, by \cref{lem:Cauchyobsholom}, we have
    \begin{align*}
        \norm{\Phi_n^{\mathbb{C}}(w)}_{C^0([0, T], X_\mathbb{C}^\sigma)}\leq C\big(\norm{\Pi_{n, v}g}_{L^2([0, T], X_\mathbb{C}^\sigma)}+\norm{\Q_n(\fk(v)+\H(v, w)+\bh)}_{L^1([0, T], X_\mathbb{C}^\sigma)}\big)
    \end{align*}
    and hence, for a possibly different constant $C>0$ depending on the previous bounds, we have
    \begin{align*}
        \norm{\Phi_n^{\mathbb{C}}(w)}_{C^0([0, T], X_\mathbb{C}^\sigma)}\leq C(\eta+\eta_0)+C(R+\eta_1)^2.
    \end{align*}
    Similarly, we have
    \begin{align*}
        \norm{\Phi_n^{\mathbb{C}}(w_1)-\Phi_n^{\mathbb{C}}(w_2)}_{C^0([0, T], X_\mathbb{C}^\sigma)}\leq 2CL(R+\eta_1)\norm{w_1-w_2}_{C^0([0, T], X_\mathbb{C}^\sigma)}.
    \end{align*}
    Up to shrinking both $R$ and $\eta$ if necessary, for $\eta_0$ small enough we can choose $\eta_1>0$ such that $\Phi_n^{\mathbb{C}}$ is contracting and reproduces the cylinder $\B_{R, \eta_1}^{[0, T]}(\Q_n X^\sigma)$. Since the map $\Phi_n^{\mathbb{C}}$ is a complex extension of the real map $\Phi_n$ from \emph{Step 1.} when its first variable is restricted to $\mc{V}_{K}$, it follows by uniqueness of the fixed point that $\boldsymbol{\mathrm{R}}$ as specified in \eqref{eq:mapRcomplex} coincides with the complex extension of the reconstruction map \eqref{eq:mapRreal} when restricted to $\mc{V}_{K}$ in the first variable.
    \medskip 
    
    \paragraph{\emph{Step 3. Regularity of the fixed point.}} By \cref{app:thm:uniffixedpoint}, that the map $(v, \bh, g)\mapsto w=\boldsymbol{\mathrm{R}}(v, \bh, g)$ as specified in \eqref{eq:mapRcomplex} is holomorphic follows by proving that the map $(v, w, \bh, g)\mapsto \Phi_{n, \bh, g, v}^{\mathbb{C}}(w)$ is holomorphic. Moreover, in view of \cref{appendix:thm:holom-sev-var}, it is enough to prove that the map $\Phi_n^{\mathbb{C}}$ is holomorphic in each one of its variables separately whilst the remaining ones are held fixed. Since for each $v\in \mc{V}_{K}+\B_{\eta,\eta}^{[0, T]}(X^\sigma)$ the map $\F_n^{\mathbb{C}}(v)$ is a linear continuous operator, the holomorphicity with respect to $\bh$ and $g$ is clear. It remains to check that $\Phi_n^{\mathbb{C}}$ is holomorphic in $v$ and $w$, separately

    First, to check that $\Phi_n^{\mathbb{C}}$ is holomorphic with respect to $v$, by linearity of $\F_n^{\mathbb{C}}$ it is enough to check that the map
    \begin{align}\label{eq:NLF-ext}
        \mc{J}_n: v\in \mc{V}_{K}+\B_{\eta,\eta}^{[0, T]}(X^\sigma)\longmapsto \F_n^{\mathbb{C}}(v)\big(g, \fk(v)+\H(v, w)\big)\in C^0([0, T], \Q_n X_\mathbb{C}^\sigma)
    \end{align}
    is complex differentiable, with $g\in \mathbb{B}_{\eta, \eta}(L^2([0, T], X^\sigma))$ and $w\in \B_{R, \eta_1}^{[0, T]}(\Q_n X^\sigma)$ being held fixed. First, note that the map
    \begin{align*}
        v\in \mc{V}_{K}+\B_{\eta,\eta}^{[0, T]}(X^\sigma)\longmapsto \fk(v)+\H(v, w)\in C^0([0, T], X_\mathbb{C}^\sigma)
    \end{align*}
    is holomorphic. Indeed, on the one hand, under Assumption \ref{NLS:assumFholom} we can easily extend $\fk$ as a holomorphic map $v\in \mc{V}_{K}+\B_{\eta,\eta}^{[0, T]}(X^\sigma)\mapsto \fk(v)\in C^0([0, T], X_\mathbb{C}^\sigma)$. On the other hand, as we did in the proof of \cref{lem:sv-holomorphic}, we can check that $\H_w: v\mapsto \H(v, w)$ is complex differentiable and its differential $d\H_w: \mc{V}_{K}+\B_{\eta,\eta}^{[0, T]}(X^\sigma)\longmapsto \mc{L}_{\mathbb{C}}(C^0([0, T], X_\mathbb{C}^\sigma))$ is characterized by
    \begin{align*}
        d\H_w(v)h=\int_0^1\big(D^2\fk(v+\tau w)-D^2\fk(v)\big)[h,w]d\tau,
    \end{align*}
    and therefore the claim follows. Let us now consider the following (nonlinear) operators.
    \begin{itemize}
        \item Let $\mc{S}_n: v\longmapsto \big(\F_n^{\mathbb{C}}(v)(g, \cdot), \fk(v)+\H_w(v))\big)$ be defined as a map from $\mc{V}_{K}+\B_{\eta,\eta}^{[0, T]}(X^\sigma)$ into $\mc{L}_{\mathbb{C}}(L^1([0, T], X_\mathbb{C}^\sigma), C^0([0, T], \Q_n X_\mathbb{C}^\sigma))\times C^0([0, T], X_\mathbb{C}^\sigma)$. From \cref{lem:CauchyObs} and the linear continuous embedding $C^0([0, T], X_\mathbb{C}^\sigma)\hookrightarrow L^1([0, T], X_\mathbb{C}^\sigma)$, we see that the first coordinate defining $\mc{S}_n$ is holomorphic. Moreover, by the previous discussion, its second coordinate is holomorphic as well, and therefore $\mc{S}_n$ is a holomorphic map.
        \item The evaluation map $\mathcal{E}: (T, \psi)\mapsto T\psi$ defined from $\mc{L}_{\mathbb{C}}(L^1([0, T], X_\mathbb{C}^\sigma), C^0([0, T], \Q_n X_\mathbb{C}^\sigma))\times L^1([0, T], X_\mathbb{C}^\sigma)$ into $C^0([0, T], \Q_n X_\mathbb{C}^\sigma)$ is linear on each coordinate and thus holomorphic.
    \end{itemize}
    By noticing that \eqref{eq:NLF-ext} can be written as $\mc{J}_n=\mc{E}\mc{S}_n$, by composition of holomorphic maps, we get that $\mc{J}_n$ is holomorphic as well on $\mc{V}_{K}+\B_{\eta,\eta}^{[0, T]}(X^\sigma)$.
    A similar argument, but simpler, proves that $\Phi_n^{\mathbb{C}}$ is holomorphic with respect to $w$ when the other variables are held fixed. This finishes the proof.
\end{proof}

\subsection{Analyticity in time of the observed solution} In this section we prove \cref{thm:mainabs-analytic}. The main building block is the following reconstruction result, where we show the existence of a holomorphic operator that allows us to reconstruct the high-frequency component for the solutions of our nonlinear system by means of the low-frequency component as an input.

\begin{theorem}\label{thm:holom-split}
    Let $T>0$, $R_0>0$, $R_1>0$ and $n_0\in \N$. Let $\A$ be a nonempty compact subset of $X^\sigma$ which is stable under the projection $\P_n$. Assume that Assumptions \ref{NLS:assumA}, \ref{assumF} (with $R_0$) and \ref{NLS:assumC} (with $T$, $R_0$, $n_0$ and $\mathbb{V}=\B_{3R_0}^{[0, T]}(X^\sigma, \A+\A)$) are enforced. Then, there exist $n\geq n_0$, $0<R<R_0$ and a nonlinear Lipschitz reconstruction operator $\mc{R}$
    \begin{align}\label{thm:hf-op}
        \mc{R}: \B_{R_0}^{[0, T]}(\P_n X^\sigma, \A)\times \B_{R_0}^{[0, T]}(X^\sigma, \A)\times\B_{R_1}^{[0, T]}(X^\sigma, \A)\longrightarrow \B_{R}^{[0, T]}(\Q_n X^\sigma)
    \end{align}
    so that, for any $u\in \B_{R_0}^{[0, T]}(X^\sigma, \A)$, $\bh_1\in \B_{R_0}^{[0, T]}(X^\sigma, \A)$ and $\bh_2\in \B_{R_1}^{[0, T]}(X^\sigma, \A)$ satisfying
    \begin{align}\label{thm:eq:nleq-obs}
        \left\{\begin{array}{cl}
        \partial_t u=Au+\fk(u+\bh_1)+\bh_2&\ \text{ in } [0, T],\\
        \bC u(t)=0&\ t\in [0, T],
        \end{array}\right.
    \end{align}
    then $\Q_n u=\mc{R}(\P_n u, \bh_1, \bh_2)$. Moreover, if additionally, Assumption \ref{NLS:assumFholom} is enforced and $\mathcal{K}$ is a nonempty compact subset of $C^0([0, T], X^\sigma)$ which is stable under the projection $\P_n$, there exist $\eta, \eta_1>0$ so that if $\bh_1\in \B_{R_0}^{[0, T]}(X^\sigma, \A)\cap\K$, for any $u\in \B_{R_0}^{[0, T]}(X^\sigma, \A)\cap \mc{K}$ solution of \eqref{thm:eq:nleq-obs}, then $\mc{R}$ with its first variable restricted to an $\eta$-neighborhood of $\P_n u$ extends holomorphically as
     \begin{multline}\label{thm:hf-op-ext}
        \mc{R}: \big(\P_n u+\B_{\eta, \eta}^{[0, T]}(\P_n X^\sigma)\big)\times \big(\B_{R_0}^{[0, T]}(X^\sigma, \A)\cap \K+\B_{\eta, \eta}^{[0, T]}(X^\sigma)\big)\\\times\big(\B_{R_1}^{[0, T]}(X^\sigma, \A)+\B_{\eta, \eta}^{[0, T]}(X^\sigma)\big)\longrightarrow \B_{R, \eta_1}^{[0, T]}(\Q_n X^\sigma).
    \end{multline}
\end{theorem}

\begin{remark}\label{rk:Acpctstable}
    In practical applications, most of the time $\A$ will be a bounded ball in $X^{\sigma+\veps}$ for some $\veps>0$, which is stable under $\P_n$ and compact in $X^\sigma$ as a consequence of Assumption \ref{NLS:assumA}.
\end{remark}

\begin{proof}
    Let us consider a cutoff function $\chi\in C_c^\infty(\R, [0, 1])$, whose support is contained in $[-1, 1]$ and satisfies $\chi(s)=1$ for $s\in [-1/2, 1/2].$ We define $\mc{R}$ as
    \begin{align}\label{thm:pf:R-map}
        \mc{R}(v, \bh_1, \bh_2)=\boldsymbol{\mathrm{R}}(v+\bh_1, \bh_2, -\chi(\eta^{-1}\norm{\bC v}_{L^2([0, T], X^\sigma)})\bC v),
    \end{align}
    where $n^*\geq n_0$, $\eta>0$, $0<R<R_0$ and the operator $\boldsymbol{\mathrm{R}}$ are given from \cref{prop:hf-unique-sol}, with $\A_1:=\A+\A$ and $\A_2:=\A$.

    Observe that $\chi(\eta^{-1}\norm{\bC v}_{L^2([0, T], X^\sigma)})=0$ if $\norm{\bC v}_{L^2([0, T], X^\sigma)}\geq \eta$, so that we always have 
    \begin{align*}
        \norm{\chi(\eta^{-1}\norm{\bC v}_{L^2([0, T], X^\sigma)})\bC v}_{L^2([0, T], X^\sigma)}\leq \eta.
    \end{align*}
    Also, $v+\bh_1\in\B_{2R_0}^{[0, T]}(X^\sigma, \A+\A)$ for every $n\in \N$. In particular, $\mc{R}$ is a well-defined map as precised in \eqref{thm:hf-op} with $n\geq n^*$.

    Now, we want to check the requirements to ensure the reconstruction property. Let $u$ be as in the theorem and solution of \eqref{thm:eq:nleq-obs}. Let us recall that $u=\P_n u+\Q_n u$ is a mild solution of \eqref{thm:eq:nleq-obs} in $\B_{R_0}^{[0, T]}(X^\sigma, \A)$. For every $n\in \N$, since $\P_nX^\sigma$ is finite dimensional and the operators $A$ and $\P_n$ commute, $v:=\P_n u$ is a mild solution of
    \begin{align*}
        \partial_t v(t)=Av(t)+\P_n \fk(u(t)+\bh_1(t))+\P_n\bh_2(t),
    \end{align*}
    and it belongs to $\B_{R_0}^{[0, T]}(\P_n X^\sigma)$. By hypothesis $u(t)\in \A$ for all $t\in [0, T]$, and since $\A$ is stable under $\P_n$, it follows that $\P_n u\in \B_{R_0}^{[0, T]}(\P_n X^\sigma, \A)$. Additionally, we need to check that, for $n\geq n^*$ large enough, we can impose $\bC \P_n u$ and $\Q_n u$ to be small enough in $X^\sigma$. Given that $\A$ is compact, utilizing the equation $\bC \P_n u=-\bC \Q_n u$, the continuity of the operator $\bC$ and that $(\Q_n)_n$ is pointwise convergent to $0$, for any $\eta>0$ we can find $n_1\geq n^*$ independent of $u$, so that for any $n\geq n_1$,
    \begin{align*}
        \norm{\bC \P_n u}_{L^2([0, T], X^\sigma)}=\norm{\bC \Q_n u}_{L^2([0, T], X^\sigma)}\leq \eta/4.
    \end{align*}
    Similarly, we can find $n_2\geq n_1$ such that $\Q_n u\in \B_{R}^{[0, T]}(\Q_n X^\sigma)$. Now, let us fix $n\geq n_2$. Then, if $w:=\mc{R}(v, \bh_1, \bh_2)$, by definition of $\boldsymbol{\mathrm{R}}$ we have that
    \begin{align*}
        w=\boldsymbol{\mathrm{R}}(\P_n u+\bh_1, \bh_2, -\bC \P_n u)
    \end{align*}
    is the unique solution in $\B_{R}^{[0, T]}(\Q_n X^\sigma)$ to
    \begin{align*}
        \left\{\begin{array}{cc}
            \partial_t w=Aw+\Q_n \fk(\P_nu+w+\bh_1)+\Q_n\bh_2, &  \\
            \Pi_{n, \P_n u+\bh_1}\bC w=-\Pi_{n, \P_n u+\bh_1}\bC \P_n u. & 
        \end{array}\right.
    \end{align*}
    Further, notice that $\Q_n u\in C^0([0, T], \Q_n X^\sigma)$ solves
    \begin{align*}
        \left\{\begin{array}{cc}
            \partial_t \Q_n u=A\Q_n u+\Q_n \fk(\P_nu+\Q_n u+\bh_1)+\Q_n\bh_2, &  \\
            \Pi_{n, \P_n u+\bh_1}\bC \Q_n w=-\Pi_{n, \P_n u+\bh_1}\bC \P_n u & 
        \end{array}\right.
    \end{align*}
    and it belongs to $\B_{R}^{[0, T]}(\Q_n X^\sigma)$. Note that $\A+\A$ is compact and stable under $\P_n$. Since $\P_n u+\bh_1\in\B_{2R_0}^{[0, T]}(\P_n X^\sigma,\A+\A)$, $\bh_2\in \B_{R_1}^{[0, T]}(X^\sigma, \A)$ and $-\bC \P_n u\in \mathbb{B}_{\eta}(L^2([0, T], X^\sigma))$, by definition of $\boldsymbol{\mathrm{R}}$, we have $\Q_n u=\boldsymbol{\mathrm{R}}(\P_n u+\bh_1, \bh_2, -\bC \P_n u)$ and therefore $w=\Q_n u=\mc{R}(v, \bh_1, \bh_2)$.
    
    In regards to the holomorphic extension, we check the required properties. If $v$ is complex-valued, in \eqref{thm:pf:R-map} instead we consider $\chi(\eta^{-1}\norm{\bC v}_{L^2([0, T], X_\mathbb{C}^\sigma)})$. Observe that $\K+\K$ is stable under $\P_n$. Let us consider $\mc{V}_{K}:=\B_{3R_0}^{[0, T]}(X^\sigma,\A+\A)\cap(\K+\K)$, which by hypothesis is compact in $C^0([0, T], X^\sigma)$. Further, by adjusting $n$ if necessary, the compactness of $\A$ provide us with the bound
    \begin{align*}
        \norm{\bC \P_n u}_{L^2([0, T], X^\sigma)}\leq \eta/4.
    \end{align*}
    Given the explicit form of the complex extension of linear operators, the map $v\mapsto \chi(\eta^{-1}\norm{\bC v}_{L^2([0, T], X_\mathbb{C}^\sigma)})$ is constant equal to $1$ around $\P_n u\in \mc{V}_{K}$ and consequently, holomorphic in some neighborhood $\P_n u+\B_{\eta_1, \eta_1}^{[0, T]}(\P_n X^\sigma)$ for $\eta_1>0$ small enough. Moreover, up to making $\eta_1$ smaller if needed, say $\eta_1\leq \eta/3$, we see that
    \begin{align*}
        \big(\P_n u+\B_{\eta_1, \eta_1}^{[0, T]}(\P_n X^\sigma)\big)+\big(\B_{R_0}^{[0, T]}(X^\sigma, \A)\cap\K+\B_{\eta_1, \eta_1}^{[0, T]}(X^\sigma)) \subset \mc{V}_{K}+\B_{\eta, \eta}^{[0, T]}(X^\sigma).
    \end{align*}
    Thus, the holomorphic extension of \eqref{thm:pf:R-map}, as specified in \eqref{thm:hf-op-ext}, is a consequence of \cref{prop:hf-unique-sol}, up to renaming $\eta_1$ as $\eta$.
\end{proof}

\subsubsection{Analyticity in time} Here we prove the main theorem of this section, that is the abstract \cref{thm:mainabs-analytic} of analytic regularity stated in the introduction. 

\begin{proof}[Proof of \cref{thm:mainabs-analytic}]
    First, by hypothesis $u\in \B_{R_0}^{[0, T^*]}(X^\sigma, \A)\cap\K$. Since $\tau\mapsto \max_{\tau\in [0,T^{*}]}\norm{\Im \bh_1(t+i\tau)}_
    {X^{\sigma}}$ is a continuous function on $[-\mu,\mu]$ that is equal to zero at $\tau=0$, there exists $0<\mu'<\mu$ so that $\max_{t\in [0,T]}\left|\Im \bh_1(t+i\tau)\right| \leq \eta$ for $\tau\in [-\mu',\mu']$, where $\eta$ is given by \cref{thm:holom-split}. We can do the same for $\bh_2$. We denote $R_{1}=\max_{z\in [0,T^{*}]+i[-\mu,\mu]}\norm{\bh_2(z)}_{X^{\sigma}}$.

    Let $n\geq n^*$, $\eta>0$, $0<R<R_0$ and $\mc{R}$ be given by \cref{thm:holom-split} with the chosen $R_1$.

    Let $0<\nu<T^*-T$. With these two choices, for each $s\in [-\nu, T^*-T-\nu]$ let us denote by $\bh_1^{s}$ the application $t\mapsto \bh_1(t+\nu+s)$, which belongs to $\B_{R_{0}}^{[0,T]}(X^{\sigma}, \A)$. Moreover, the application $s\mapsto \bh_1^s$ is continuous. Indeed, since $t\in [0, T^*]\mapsto \bh_1(t)\in X^\sigma$ is a continuous map defined on a compact set, it is uniformly continuous. Hence, for every $\veps>0$ and any $t\in [0, T]$, there exists $\delta>0$ so that for any $s$, $s_0\in [-\nu, T^*-T-\nu]$ with $|(s+t)-(s_0+t)|\leq \delta$, then $\norm{\bh_1(t+\nu+s)-\bh_1(t+\nu+s_0)}_{X^\sigma}\leq \veps$. Since the later property does not depend on $t$, we can take $\sup_{t\in [0, T]}$ in the last inequality to obtain that $\norm{\bh_1^s-\bh_1^{s_0}}_{C^0([0, T], X^\sigma)}\leq \veps$. We do the same for $\bh_2$, so we have $\bh_2^s\in \B_{R_1}^{[0,T]}(X^{\sigma}, \A)$ for any $s\in [-\nu, T^*-T-\nu]$ and moreover the map $s\mapsto \bh_2^s$ is uniformly continuous.
    \medskip 
    \paragraph{\emph{Step 1. Time-shifted reconstruction operator}} Let us define for $s\in [-\nu, T^*-T-\nu]$, $u^s$ as the application $t\mapsto u(t+s+\nu)$, which belongs to $\B_{R_0}^{[0, T]}(X^\sigma, \A)$. Let us decompose $u=\P_n u+\Q_n u=:v+w$ and similarly $u^s=v^s+w^s$. For any fixed $s\in [-\nu, T^*-T-\nu]$, that $u$ satisfies \eqref{NLS:eq:UCabstT*intro} implies
    \begin{align*}
        \left\{\begin{array}{cr}
             \partial_t u^s=Au^s+\fk(u^s+\bh_1^s)+\bh_2^s& \textnormal{ on } [0,T], \\
             \bC u^s(t)=0 &\textnormal{ for }t\in [0,T].
        \end{array}\right.
    \end{align*}
    Since $\A$ is stable under $\P_n$, we obtain that for any fixed $s\in [-\nu, T^*-T-\nu]$ we have $v^s\in \B_{R_0}^{[0, T]}(\P_n X^\sigma, \A)$. Then \cref{thm:holom-split} gives
    \begin{align}\label{eq:ws-reconst}
        w^s=\Q_n u^s=\mc{R}(\P_n u^s, \bh_1^s, \bh_2^s)=\mc{R}(v^s, \bh_1^s, \bh_2^s),
    \end{align}
    with equality meant in $C^0([0, T], \Q_n X^\sigma)$. If we denote 
    \begin{align*}
        g:=\P_n \fk(u+\bh_1)+\P_n\bh_2=\P_n \fk(v+w+\bh_1)+\P_n\bh_2\in C^0([0, T^*], \P_nX^\sigma),
    \end{align*}
    then $v\in C^1([0, T^*], \P_n X^\sigma)$ solves
    \begin{align}\label{eq:lf-vt}
        \partial_t v=\P_nAv+g,\ \text{ on }\ [0, T^*].
    \end{align}
    With the related notations (see \cref{A:eveq}), \cref{lmtranslat} implies
    \begin{align}\label{eq:lf-vs}
        \partial_s v^s=\P_nAv^s+g^s,\ \text{ on }\ [-\nu, T^*-T-\nu].
    \end{align}
    
    Furthermore, since $R_0+R\leq 2R_0$, for any $s\in [-\nu, T^*-T-\nu]$, , by using equality \eqref{eq:ws-reconst},
    \begin{align*}
        g^s=\P_n \fk(v^s+w^s+\bh_1^s)+\P_n\bh_2^s=\P_n \fk\big(v^s+\mc{R}(v^s, \bh_1^s, \bh_2^s)+\bh_1^s\big)+\P_n\bh_2^s,
    \end{align*}
    with equality in $C^0([0, T], X^\sigma)$. This motivates us to introduce the map
    \begin{align*}
        \left\{\begin{array}{rcl}
            \widetilde{\fk}: [-\nu, T^*-T-\nu]\times \B_{R_0}^{[0, T]}(\P_n X^\sigma, \A) & \longrightarrow & C^0([0, T], \P_n X^\sigma) \\
            (s, \widetilde{v}) & \longmapsto & \P_n\fk\big(\widetilde{v}+\mc{R}(\widetilde{v}, \bh_1^s, \bh_2^s)+\bh_1^s\big)+\P_n\bh_2^s. 
        \end{array}\right.
    \end{align*}
    This application is well-defined, continuous (by algebra of continuous functions) and Lipschitz-continuous with respect to the second variable (by Assumption \ref{NLS:assumFholom}, that $\mc{R}$ is Lipschitz and composition of Lipschitz maps).

    Observe that $\bh_1^s\in \B_{R_0}^{[0, T]}(X^\sigma, \A)$, $\bh_2^s\in \B_{R_1}^{[0, T]}(X^{\sigma}, \A)$ and $\Tilde{v}\in \B_{R_{0}}^{[0,T]}(\P_n X^{\sigma}, \A)$ imply, by construction, that $\Tilde{v}+\mathcal{R}(\Tilde{v},\bh_1^s,\bh_2^s)+\bh_1^s\in \B_{3R_0}^{[0, T]}(X^\sigma)$ for all $s\in [-\nu, T^*-T-\nu]$. Therefore $\fk(\Tilde{v}+\mathcal{R}(\Tilde{v},\bh_1^s,\bh_2^s)+\bh_1^s)$ is well defined and so is $\Tilde{\fk}$. Recall that $\mc{R}$ is Lipschitz on its variables, that is, there exists $C>0$ such that
    \begin{multline*}
        \norm{\mc{R}(\Tilde{v}, \bh_1^{s}, \bh_2^{s})-\mc{R}(\Tilde{v}^*, \bh_1^{s^*}, \bh_2^{s^*})}_{C^0([0, T], X^\sigma)}\leq C\big( \norm{\Tilde{v}-\Tilde{v}^*}_{C^0([0, T], X^\sigma)}\\+\norm{\bh_1^{s}-\bh_1^{s^*}}_{C^0([0, T], X^\sigma)}+\norm{\bh_2^{s}-\bh_2^{s^*}}_{C^0([0, T], X^{\sigma})} \big),
    \end{multline*}
    for any $\Tilde{v}$, $\Tilde{v}^*\in \B_{R_{0}}^{[0,T]}(\P_n X^{\sigma})$ and $s$, $s^*\in [-\nu, T^*-T-\nu]$. The continuity of $\Tilde{\fk}$ then follows by the continuity of $s\mapsto (\bh_1^s, \bh_2^s)$, Assumption \ref{NLS:assumFholom} and by algebra of continuous maps. The same estimate along with Assumption \ref{NLS:assumFholom} shows that
    \begin{align*}
        \norm{\Tilde{\fk}(s, \Tilde{v})-\Tilde{\fk}(s, \Tilde{v}^*)}_{C^0([0, T], \P_n X^\sigma)}\leq C\norm{\Tilde{v}-\Tilde{v}^0}_{C^0([0, T], \P_n X^\sigma)},
    \end{align*}
    for any $\Tilde{v}$, $\Tilde{v}^*\in \B_{R_{0}}^{[0,T]}(\P_n X^{\sigma})$ and $s\in [-\nu, T^*-T-\nu]$.
    \medskip
    \paragraph{\emph{Step 2. Holomorphic extensions.}} First, let us consider the following application 
    \begin{align*}
        \Psi: (s, \phi)\in [-\nu, T^*-T-\nu]\times\K\longmapsto \Psi(s, \phi)=\phi^s\in C^0([0, T], X^\sigma),
    \end{align*}
    which is continuous since translation and restriction in time are both continuous applications. We thus consider the compact set $\widetilde{\K}$ defined as the image of the compact set $[-\nu, T^*-T-\nu]\times\K$ under $\Psi$, which has the property that $\phi^s\in \widetilde{\K}$ for any $(s, \phi)\in [-\nu, T^*-T-\nu]\times\K$. Additionally, $\widetilde{\K}$ inherits the stability under $\P_n$ from that of $\K$. Therefore, since both $u$ and $\bh_1$ belong to $\B_{R_0}^{[0, T^*]}(X^\sigma, \A)\cap\mc{K}$, and by hypothesis both $\A$ and $\K$ are stable under $\P_n$, it follows that $v^{s_0}\in \B_{R_0}^{[0, T]}(\P_n X^\sigma, \A)\cap \widetilde{\K}$ and $\bh_1^{s_0}\in \B_{R_0}^{[0, T]}(X^\sigma, \A)\cap \widetilde{\K}$ for any $s_0\in [-\nu, T^*-T-\nu]$.
    
    We claim that for any $s_0\in (-\nu, T^*-T-\nu)$, there exists $\rho>0$ such that, if we restrict $\widetilde{\fk}$ to $[-\nu, T^*-T-\nu]\times \B_{R_0}^{[0, T]}(\P_nX^\sigma, \A)\cap \widetilde{\mc{K}}$, it admits a holomorphic extension around $(s_0, v^{s_0})$, given by
    \begin{align}\label{NLS:fHolomext}
        \left\{\begin{array}{rcl}
            \widetilde{\fk}_{\mathbb{C}}: B_{\mathbb{C}}(s_0, \rho)\times\big(v^{s_0}+\B_{\eta, \eta}^{[0, T]}(\P_n X^\sigma)\big) & \longrightarrow & C^0([0, T], \P_n X_\mathbb{C}^\sigma) \\
            (z, \widetilde{v}) & \longmapsto & \P_n\fk(\widetilde{v}+\mc{R}(\widetilde{v}, \bh_1^z, \bh_2^z)+\bh_1^z\big)+\P_n\bh_2^z.
        \end{array}\right.
    \end{align}
    Since $\widetilde{\mc{K}}$ is compact and stable under $\P_n$, there exists $\eta_1>0$ such that, if the first variable of $\mc{R}$ is restricted to $\B_{R_0}^{[0, T]}(\P_nX^\sigma)\cap\widetilde{\mc{K}}$, it has a holomorphic extension around $v^{s_0}$ given by \cref{thm:holom-split},
    \begin{multline*}
        \mathcal{R}: \left[ v^{s_0}+\B_{\eta,\eta}^{[0,T]}(\P_n X^{\sigma})\right] \times  \left[\B_{R_0}^{[0,T]}(X^{\sigma}, \A)\cap \widetilde{\K}+\B_{\eta, \eta}^{[0, T]}(X^\sigma)\right]\\ \times \big(\B_{R_1}^{[0,T]}(X^{\sigma}, \A)+\B_{\eta, \eta}^{[0, T]}(X^\sigma)\big) \longrightarrow \B_{R, \eta_1}^{[0, T]}(\Q_n X^{\sigma}).
    \end{multline*}
    We need to argue that, the application $z\mapsto \bh_1^z$ is holomorphic from $B_{\mathbb{C}}(s_0, \rho)$ with value in $C^0([0, T], X_\mathbb{C}^\sigma)$ for $\rho>0$ small enough, and that the same holds for $z\mapsto \bh_2^z$ from $B_{\mathbb{C}}(s_0, \rho)$ with value in $C^0([0, T], X_\mathbb{C}^{\sigma})$. Indeed, since
    \begin{align*}
        z\in [0, T^*]+i[-\mu', \mu']\mapsto \bh_1(z)\in X_\mathbb{C}^\sigma
    \end{align*}
    is holomorphic in the open strip and continuous up to the boundary, for each $s_0\in [-\nu, T^*-T-\nu]$ we can find $\rho>0$ such that, for any $t\in [0, T]$, the application
    \begin{align*}
        z\in B_\mathbb{C}(s_0, \rho)\longmapsto \bh_1(z+t+\nu)\in X_\mathbb{C}^\sigma
    \end{align*}
    is holomorphic and continuous up to the boundary. By using Cauchy estimates and the bound on $\bh_1(z)$, for each $t\in [0, T]$ and any $z_0\in B_\mathbb{C}(s_0, \rho)$, we can find $\ell>0$ such that
    \begin{align*}
        \norm{\bh_1(z+h+t+\nu)-\bh_1(z+t+\nu)-\delta \bh_1(z+t+\nu, h)}_{X_\mathbb{C}^\sigma}\leq \dfrac{4(R_0+\eta)|h|^2}{\ell(\ell-2|h|)}\leq \dfrac{8(R_0+\eta)|h|^2}{\ell^2} 
    \end{align*}
    holds uniformly for all $z\in B_\mathbb{C}(z_0, \ell/4)$ and $|h|\leq \ell/4$. Moreover, the last estimate is independent of $t\in [0, T]$, so we actually have
    \begin{align*}
        \norm{\bh_1^{z+h}-\bh_1^z-\delta \bh_1(z+\cdot+\nu, h)}_{C^0([0, T], X_\mathbb{C}^\sigma)}\leq \dfrac{8(R_0+\eta)|h|^2}{\ell^2},
    \end{align*}
    uniformly for all $z\in B_\mathbb{C}(z_0, \ell/4)$ and $|h|\leq \ell/4$. Further, by uniform continuity, up to shrinking $\rho>0$, we have that $\norm{\bh_1^z-\bh_1^{s_0}}_{C^0([0, T], X_\mathbb{C}^\sigma)}\leq \eta$ whenever $|z-s_0|\leq \rho$, that is, $\bh_1^z$ lies in a $\eta$-neighborhood of $\B_{R_0}^{[0, T]}(X^\sigma, \A)\cap\widetilde{\K}$. In summary, we showed that the application
    \begin{align*}
        z\in B_\mathbb{C}(s_0, \rho)\longmapsto \bh_1^z\in \B_{R_0}^{[0, T]}(X^\sigma, \A)\cap\widetilde{\K}+\B_{\eta, \eta}^{[0, T]}(X^\sigma)
    \end{align*}
    is well-defined and holomorphic. The argument works similarly to show that $z\in B_\mathbb{C}(s_0, \rho)\mapsto \bh_2^z\in\B_{R_0}^{[0, T]}(X^\sigma, \A)+\B_{\eta, \eta}^{[0, T]}(X^\sigma)$ is holomorphic.
    
    By composition of holomorphic functions and \cref{appendix:thm:holom-sev-var} we get that
    \begin{align*}
        \left\{\begin{array}{rcl}
        B_{\mathbb{C}}(s_0, \rho)\times \left[v^{s_0}+\B_{\eta,\eta}^{[0,T]}(\P_n X^{\sigma})\right]&\to& \B_{R, \eta_1}^{[0, T]}(\Q_n X^\sigma)\\
        (z,\Tilde{v})&\mapsto& \mathcal{R}(\Tilde{v},\bh_1^z,\bh_2^z)
        \end{array}\right.
    \end{align*}
    is a well-defined and holomorphic map. As long as $2\eta+\eta_1\leq 3\delta_0/2$, that the extension $\widetilde{\fk}_{\mathbb{C}}$ defined in \eqref{NLS:fHolomext} is holomorphic follows from Assumption \ref{NLS:assumFholom}, algebra of holomorphic maps, and that $\P_n$ is a linear bounded operator.
    \medskip 
    
    \paragraph{\emph{Step 3. ODE regularity argument}} Observe that for any $s\in [-\nu, T^*-T-\nu]$, we proved that $g^s=\widetilde{\fk}(s, v^s)$. Therefore, equation \eqref{eq:lf-vs}, verified by $v^s$, can be rewritten as
    \begin{align}\label{eq:lf-vs-2}
        \partial_s v^s=\P_nAv^s+\widetilde{\fk}(s, v^s),\ \text{ on }\ [-\nu, T^*-T-\nu].
    \end{align}
    We now consider the following ODE, locally defined in the Banach space $C^0([0, T], \P_n X^\sigma)$,
    \begin{align}\label{eq:lf-vs-3}
        \left\{\begin{array}{l}
            \partial_s\xi(s)=\P_n A\xi(s)+\widetilde{\fk}_{\mathbb{C}}(s, \xi(s)),  \\
            \xi(s_0)=v^{s_0},
        \end{array}\right.
    \end{align}
    where $\widetilde{\fk}_{\mathbb{C}}$ is the map \eqref{NLS:fHolomext}. Observe that $C^0([0, T],\P_n X^\sigma)$ is of infinite dimension. However, it can be checked that $\P_n A$ is a linear bounded operator on this space. \cref{lmDuhamelCauchy} shows that $v^s$, which solves \eqref{eq:lf-vs-3} as a time-dependent Banach valued map, also solves the ODE in the usual sense. Since $\widetilde{\fk}_{\mathbb{C}}$ is holomorphic and given its mapping properties \eqref{NLS:fHolomext}, $\P_n A+\widetilde{\fk}_{\mathbb{C}}$ is holomorphic from $B_{\mathbb{C}}(s_0, \rho)\times\big[v^{s_0}+\B_{\eta,\eta}^{[0, T]}(\P_n X^\sigma)\big]$ into $C^0([0, T], X_\mathbb{C}^\sigma)$. By classical theory of ODEs in Banach spaces \cite[Theorem 10.4.5]{Dieu69}, we obtain a unique classical solution of \eqref{eq:lf-vs-3},
    \begin{align*}
        \xi: B_\mathbb{C}(s_0, \rho')\mapsto \widetilde{v}(z)\in C^0([0, T], \P_n X_\mathbb{C}^\sigma),
    \end{align*}
    for some $0<\rho'<\rho$. Moreover, such a solution map inherits the regularity of the right-hand side of the ODE it satisfies, therefore it is a holomorphic map. Since the map $s\in [-\nu, T^*-T-\nu]\mapsto v^s\in \B_{R_0}^{[0, T]}(\P_n X^\sigma, \A)$ is uniformly continuous, up to shrinking $\rho'$, we have that $v^s\in v^{s_0}+\B_{\eta}^{[0, T]}(\P_n X^\sigma, \A)$ for every $s\in (s_0-\rho', s_0+\rho')$ and then by construction $\widetilde{\fk}_{\mathbb{C}}(s, v^s)=\widetilde{\fk}(s, v^s)$. In view of \eqref{eq:lf-vs-2} and by uniqueness of solutions, we have $\xi(s)=v^s$ for $s\in (s_0-\rho', s_0+\rho')$.

    Furthermore, the maps $z\mapsto \xi(z)$ and
    \begin{align*}
        \left\{\begin{array}{rcl}
            B_{\mathbb{C}}(s_0, \rho)\times\big[v^{s_0}+\B_{\eta, \eta}^{[0, T]}(\P_n X^\sigma)\big] &  \longrightarrow  & C^0([0, T], \P_n X_\mathbb{C}^\sigma)  \\
            (z, \widetilde{v}) &  \longmapsto  & \widetilde{v}+\mc{R}(\widetilde{v},\bh_1^z, \bh_2^z) 
        \end{array}\right.
    \end{align*}
    are both holomorphic in their respective spaces, and so is its composition $z\mapsto \xi(z)+\mc{R}(\xi(z),\bh_1^z,\bh_2^z)$. We further notice that for $s\in (s_0-\rho', s_0+\rho')$, $\xi(s)=v^s$ and therefore
    \begin{align*}
        \xi(s)+\mc{R}(\xi(s),\bh_1^s,\bh_2^s)=v^s+\mc{R}(v^s, \bh_1^s,\bh_2^s)=v^s+w^s=u^s,
    \end{align*}
    where we used \eqref{eq:ws-reconst} and that $(s_0-\rho', s_0+\rho')\subset (-\nu, T^*-T-\nu)$ for $\rho'>0$ small enough. In particular, the map $s\in (s_0-\rho', s_0+\rho')\mapsto u^s\in C^0([0, T], X^\sigma)$ is the restriction to a real interval of a holomorphic map, hence it is real analytic.

    Since for any $t_0\in [0, T]$ the trace application $\zeta\in C^0([0, T], X^\sigma)\mapsto \zeta(t_0)\in X^\sigma$ is linear continuous, we obtain by composition that the application
    \begin{align*}
        \left\{\begin{array}{rcl}
            (s_0-\rho', s_0+\rho') &  \longrightarrow  & X^\sigma  \\
            s &  \longmapsto  & u^s(t_0)=u(t_0+s+\nu)
        \end{array}\right.
    \end{align*}
    is real analytic. Observe that $\rho'$ depends on all the other parameters, while $\nu$ is an arbitrary number satisfying $0<\nu<T^*-T$, $s_0$ is any number satisfying $-\nu<s_0<T^*-T-\nu$ and $t_0$ is arbitrary in $[0, T$]. This means that $t\mapsto u(t)$ (which is well-defined for $t\in [0, T^*]$) is real analytic in a neighborhood of any $t_1$ of the form $t_1=t_0+\nu+s_0$, with $t_0$, $\nu$ and $s_0$ as before. Looking carefully at the constraints, we see that this implies that $t\mapsto u(t)$ is real analytic from $(0, T^*)$ into $X^\sigma$, as expected.
\end{proof}

\section{Applications to the nonlinear Schrödinger equation}\label{sec:nls_uc}

In this section we focus on the nonlinear Schrödinger equation, aiming to prove the main results about propagation of analyticity and unique continuation announced in \cref{sec:introNLS}.

\subsection{Preliminaries}\label{s3:preliminaries}
Let $L^2(\M):=L^2(\M; \mathbb{C})$ be equipped with the usual Hermitian inner product. By the spectral theorem, $-\Delta_g$ has a compact resolvent and thus, we can construct a complete orthonormal basis $(e_j)_{j\in \N}$ of eigenfunctions of $-\Delta_g$, associated to the eigenvalues $(\ld_j)_{j\in \N}$. In particular, we have $e_j\in C^\infty(\M)$, $-\Delta e_j=\ld_j e_j$ with $\ld_j\geq 0$ and $\inn{e_j, e_k}_{L^2(\M)}=\delta_{jk}$. We introduce the high-frequency projector $\Q_n$ on the space $\overline{\mathrm{span}(e_{j})_{j\geq n}}$ and then we set the low-frequency projector $\P_n=I-\Q_n$.

For $s\in \R$, we introduce the operator $\Ld_s: (1-\Delta)^{s/2}: C^\infty(\M)\to C^\infty(\M)$ defined spectrally by
\begin{align*}
    \Ld_s \psi=\sum_{j\in \N} (1+\ld_j)^{s/2}\inn{\psi, e_j}_{L^2(\M)}e_j.
\end{align*}
By duality, we can extend it as an operator $\Ld_s: \D'(\M)\to \D'(\M)$. We define the Sobolev spaces
\begin{align*}
    H^s(\M)=\{\psi\in \D'(\M)\ |\ \Ld_s\psi\in L^2(\M)\},
\end{align*}
equipped with
\begin{align*}
    \inn{\psi, \phi}_{H^s(\M)}=\inn{\Ld_s \psi, \Ld_s\phi}_{L^2(\M)}\ \text{ and }\ \norm{\psi}_{H^s(\M)}^2=\norm{\Ld_s\psi}_{L^2(\M)}^2.
\end{align*}
We can similarly define the real Sobolev space $H^s(\M; \mathbb{R})$ just by considering all the underlying spaces above to be $\mathbb{R}$ instead of $\mathbb{C}$.

On any $H^r(\M)$ with $r\in \R$, the operator $\Ld_s$ is an unbounded selfadjoint operator with domain $H^{r+s}(\M)$ and, is an isomorphism from $H^{r+s}(\M)$ onto $H^r(\M)$. According to \cite[Section 11]{Shu01}, we have
\begin{align*}
    \Ld_s\in \Psi_{\phg}^s(\M)\ \text{ with }\ \sigma_{\Ld_s}(x, \xi):=|\xi|_x^s,\ (x, \xi)\in T_0^*\M.
\end{align*}
We refer to \cref{A:pseudodiff} for some basics about pseudodifferential operators and notation that will be used throughout this section.

Set $X=L^2(\M)$ and introduce the operator
\begin{align*}
    A=i\Delta_g\ \text{ with }\ D(A)=H^2(\M).
\end{align*}
With the standard scalar product on $X$ we have
\begin{align*}
    A^*=-A\ \text{ and }\ A^*A=-A^2=\Delta_g^2.
\end{align*}
For $\sigma\geq 0$ one has
\begin{align*}
    X^\sigma=D((A^*A)^{\sigma/2})=D((-\Delta_g)^{\sigma})=H^{2\sigma}(\M, \mathbb{C}).
\end{align*}
To fit in the abstract framework given in \cref{sec:NLSabstract-construction}, we shall view $\mathbb{C}$ as $\R\times\R$. Indeed, to complexify the complex Hilbert space $H^s(\M; \mathbb{C})$, we note that
\begin{align*}
    H^s(\M; \mathbb{C})\simeq H^s(\M; \mathbb{R})\oplus iH^s(\M; \R)
\end{align*}
and thus $H^s(\M; \mathbb{C})$ can be identified with the real Hilbert space $H^s(\M; \R^2)$. As a consequence, we can consider the canonical complexification of $H^s(\M; \R^2)$, which is nothing but $H^s(\M; \mathbb{C}^2)$. This identification will be used freely and whenever the context is clear and there is no risk of confusion, we will just write $H^s(\M)$ for the corresponding real or complex Sobolev space. Observe then that $A$ satisfies Assumption \ref{NLS:assumA} on the real Hilbert space $H^s(\M, \R^2)$, and it will be used throughout the whole section, with no mention unless necessary.

\begin{remark}
    If $\psi=\psi_1+i\psi_2\in D(A)$ then $i\Delta_g\psi=-\Delta_g\psi_2+i\Delta_g\psi_1$. Therefore, under the above identification, $A$ acts on the space $H^s(\M; \mathbb{R}^2)$ as $\begin{psmallmatrix}
            0 & -\Delta_g \\ \Delta_g & 0
        \end{psmallmatrix}$.
\end{remark}

Let $\chi\in C^\infty(\M; \R)$ and let us now introduce $\fk(u):=-i\chi f(u)$ where $f: \mathbb{C}\to \mathbb{C}$ will be either real $C^\infty$ or real analytic. We then consider the Cauchy problem
\begin{align}\label{eq:NLSabst}
    \left\{\begin{array}{cc}
        \partial_tu=Au+\fk(u),     &  \\
        u(0)=u_0.     & 
    \end{array}\right.
\end{align}
Observe that if $\chi=1$, the latter corresponds to the abstract formulation of \eqref{eq:NLS}. Further, under the identification $\mathbb{C}\simeq \R^2$ we canonically identify $\fk$ with a function from $\R^2$ into $\R^2$ and 
\eqref{eq:NLSabst} is actually a system of two equations.

Since $A$ is skew-adjoint, by Stone's theorem it generates a unitary group $t\mapsto e^{tA}$ on $X$ and $D(A)$. In particular,
\begin{align}
    \forall t\in [0, T],\ \norm{e^{tA}}_{\mc{L}(X)}=1\ \text{ and }\ \norm{e^{tA}}_{\mc{L}(D(A))}=1.
\end{align}
By linear interpolation, the same holds on $X^\sigma$ for any $\sigma\geq 0$.

We now verify that the nonlinearity $f$ verifies Assumption \ref{NLS:assumFholom}. First of all, we recall a version of a result that can be found in Alinhac-G\'erard \cite[Proposition 2.2]{AG91}, in relation to the regularity of a composition.
\begin{lemma}\label{lem:comp-reg}
    Let $N$, $L\in \N$ and let $g: \R^N\to \R^L$ be a $C^\infty(\R^N, \R^L)$ function, with $g(0)=0$. If $u\in L^\infty(\M, \R^N)\cap H^s(\M, \R^N)$, with $s>0$, then $g(u)\in L^\infty(\M, \R^L)\cap H^s(\M, \R^L)$ and $\norm{g(u)}_{H^s}\leq C\norm{u}_{H^s}$, where $C$ only depends on $g$ and $\norm{u}_{L^{\infty}}$.
\end{lemma}

\begin{remark}
    The previous Lemma is actually written in \cite{AG91} for function in $H^s(\R^d)$ and $f\in C^{\infty}$ of real variable. On the one hand, such a result extends verbatim to the multivariable case, using for instance Faà di Bruno's formula to obtain the required estimates to verify the Meyer's multiplier condition. On the other hand, the same result still holds for functions in $H^s(\M)$ when $\M$ is a compact manifold with or without boundary using the definition of the norm of $H^s(\M)$ by partition of unity and sum of the norm in $H^s(\R^d)$ of the functions in local coordinates and with extension.
\end{remark}

\begin{proposition}\label{prop:NLS-assumF}
    Let $d\in \N$, $s>d/2$ and $\chi\in C^\infty(\M)$. Then if $f$ is real $C^\infty(\mathbb{C}, \mathbb{C})$ then $\fk=-i\chi f$ satisfies Assumption \ref{assumF} with $\sigma=s/2$ and for every $R_0$. Additionally, if $f$ is real analytic, then for every $R_0>0$ there exists $\delta_0>0$ such that $\fk$ satisfies Assumption \ref{NLS:assumFholom} with $R_0$ and $\delta_0$. 
\end{proposition}
\begin{proof}
    We begin by identifying $H^{s}(\M, \mathbb{C})$ with $H^{s}(\M, \R^2)=H^{s}(\M)\times H^{s}(\M)$. Observe that with our choice of $s$ we have that $H^s$ is an algebra and $H^{s}(\M, \R^2)\hookrightarrow L^\infty(\M, \R^2)$ with embedding constant $\kappa$. By working on each component of $f$, we reduce the analysis to the case where $f: \R^2\to \R$, which we assume from now on. As the multiplication by a smooth function is a bounded linear operator from $H^s$ into itself, without any loss of generality, we assume $\chi\equiv -1$ and consider $\fk=f$.

    Given that $f$ is real $C^\infty$, by \cref{lem:comp-reg} we have that both $f(v)$ and $Df(v)-Df(0)$ are well-defined in $H^{s}(\M, \R^2)$, for any $v\in H^s(\M, \R^2)$ along with the bound $\norm{Df(v)-Df(0)}_{H^s}\leq C \norm{v}_{H^s}$ with $C$ depending only on $Df$ and $\norm{v}_{L^\infty}$. Therefore, using that $H^{s}$ is an algebra, we get, for any $v$, $v'\in H^{s}$,
    \begin{align}
        \nonumber \norm{f(v)-f(v')}_{H^{s}}&=\Norm{\int_0^1 Df(v'+\tau(v-v'))(v-v')d\tau}_{H^{s}}\\ 
        \nonumber &\leq C\left(1+\Norm{\int_0^1 \big(Df(v'+\tau(v-v'))-Df(0)\big)d\tau}_{H^{s}}\right)\norm{v-v'}_{H^{s}}\\
        \label{fineq1} &\leq C\left(1+\norm{v}_{H^{s}}+\norm{v'}_{H^{s}}\right)\norm{v-v'}_{H^{s}},
    \end{align}
    where $C$ is a constant depending only on $Df$ and the $L^\infty$-norm of both $v$ and $v'$. We apply a similar reasoning to the second derivative of $f$, which lead us to
    \begin{align}\label{Dfineq1}
        \norm{Df(v)-Df(v')}_{\mc{L}(H^{s})}\leq C\left(1+\norm{v}_{H^{s}}+\norm{v'}_{H^{s}}\right)\norm{v-v'}_{H^{s}}.
    \end{align}    
    Once we assume that $v$, $v'\in \mathbb{B}_{R_0}(H^s(\M, \R^2))$ we observe that each component of $v$ and $v'$ always stays smaller than $4\kappa R_0$ in $L^{\infty}(\M, \R)$, thus $f$ is well-defined as a map on $\mathbb{B}_{R_0}(H^s(\M, \R^2))$. Further, estimate \eqref{fineq1} and \eqref{Dfineq1} directly implies that $f$ satisfies Assumption \ref{assumF} with $\sigma=s/2$.

    Now, we will show that $f$ satisfies Assumption \ref{NLS:assumFholom} if it is assumed to be real analytic. By compactness, there exists $\delta>0$ small such that $f$ extends holomorphically into the interior of the complex region
    \begin{align*}
        \mathbb{S}_{R_0, \delta}:=\{(z_1+iz_2, w_1+iw_2)\in \mathbb{C}^2\ |\ |z_1|, |w_1|\leq 4\kappa R_0\ \text{and}\ |z_2|, |w_2|\leq 2\kappa \delta\}.
    \end{align*}
    Moreover, this extension is continuous up to the boundary and there exists a constant $M>0$ so that $|f(z, w)|\leq M$ for all $(z, w)\in \mathbb{S}_{R_0, \delta}$. We still denote by $f$ such an extension. By identifying $\mathbb{C}^2\simeq \mathbb{R}^4$, Lemma \ref{lem:comp-reg} allows us to consider the composition of smooth functions defined on domains of $\mathbb{C}^2$ by functions in $H^{s}$, assuming that the composition makes sense. Since $\kappa$ is the constant of the embedding $H^{s}\hookrightarrow L^\infty$, we see that $f(v)$ is well defined in $\mathbb{B}_{4R_0, 2\delta}(H^{s}(\M; \mathbb{C}^2))$. In particular, it satisfies the same estimates as \eqref{fineq1} and \eqref{Dfineq1}. 

    To prove that $f$ is holomorphic, by \cref{appendix:thm:holom-sev-var}, we are led to prove that $f$ is continuous and holomorphic on each variable while the remaining one is being held fixed. The continuity is direct as we already showed that $f$ is Lipschitz-continuous in bounded sets of $H^s(\M; \mathbb{C}^2)$. We first perform a Taylor development of order $2$, say $f(z+r_1, w)=f(z, w)+\partial_1f(z, w)r_1+r_1^2\int_0^1 \partial_1^2f(z+\tau r_1, w)(1-\tau)d\tau$ for $(z, w)\in \Int(\mathbb{S}_{R_0, \delta})$ and $r_1\in \mathbb{C}$ small. Since $f$, $\partial_1 f$ and $\partial_1^2 f$ are smooth functions, by \cref{lem:comp-reg}, they are well-defined as maps on $H^s(\M; \mathbb{C}^2)$ and moreover, for any $(v_1, v_2)\in \Int(\mathbb{B}_{4R_0, 2\delta}(H^s(\M; \mathbb{C}^2))$ and $h\in \mathbb{B}_{\eta, \eta}(H^s(\M, \mathbb{C}))$ with $\eta>0$ small enough, depending on $(v_1, v_2)$ and $f$, we can write
    \begin{align}\label{eq:f-exp}
        f(v_1+h, v_2)=f(v_1, v_2)+\partial_1f(v_1, v_2)h+h^2\int_0^1\partial_1^{2}f(v_1+\tau h, v_2)(1-t)dt,
    \end{align}
    with equality meant as functions in $H^s(\M, \mathbb{C})$ (as $v_2$ is being held fixed). Therefore $f$ is holomorphic in the sense of \cref{def-Kdifferentiable} with (partial) derivative $h\mapsto \partial_1 f(v_1, v_2)h$, which is continuous and $\mathbb{C}$-linear as a map from $H^s(\M, \mathbb{C})$ into itself. The same argument works when $v_1$ is being held fixed. This finishes the proof.
\end{proof}

\subsubsection{Well-posedness}\label{s3:wp} Here we briefly recall some results related to the well-posedness of \eqref{eq:NLS} in the subcritical case. It is worth mention that in the case where $s>d/2$, the well-posedness can be classically achieved by Picard iteration in $C^0([0, T], H^s(\M))$.

In dimension $d=2$, a crucial tool to handle the subcritical case is the use of Strichartz estimates, which were obtained by Burq, Gérard, and Tzvetkov \cite{BGT04} by means of semiclassical analysis.

\begin{theorem}\cite[Theorem 1]{BGT04}\label{thm:strichartz}
    Let $(\M, g)$ be a compact Riemannian manifold of dimension $d\geq 2$. Let $(p, q)\in [2, \infty)\times[2,\infty)$ satisfying the admissibility condition
    \begin{align*}
        \dfrac{2}{p}+\dfrac{d}{q}=\dfrac{d}{2},\ (p, q)\neq (2, \infty).
    \end{align*}
    For any finite time interval $I$, there exists $C=C(I)>0$ such that for every $v_0\in L^2(\M)$ solution of $i\partial_t v+\Delta_g v=0$ with $v(0,\cdot)=v_0$, it holds
    \begin{align*}
        \norm{v}_{L^p(I, L^q(\M))}\leq C(I)\norm{v_0}_{H^{1/p}(\M)}.
    \end{align*}
\end{theorem}

To treat the nonlinearity and ensure the availability of the Strichartz estimates, the solution is sought in the following Banach space
\begin{align*}
    Y_T=C^0([0, T], H^1(\M))\cap L^p([0, T], W^{1-1/p, q}(\M)),
\end{align*}
equipped with its natural norm, where $p>\max(\varrho-1, 2)$, $1/p+1/q=1/2$ with $p<\infty$, with the notation $\varrho=2\text{deg}(P)-1\geq 3$. We then have the following well-posedness result for the defocusing NLS in dimension $d=2$.

\begin{theorem}\cite[Theorem 2]{BGT04}\label{thm:wellposedH1}
    Let $(\M, g)$ be a compact Riemannian surface and let $P$ be a polynomial with real coefficients such that $P'(r)\to+\infty$ as $r\to+\infty$. For every $u_0\in H^1(\M)$, there exists a unique maximal solution $u\in C(\R, H^1(\M))$ of equation 
    \begin{align*}
        i\partial_t u+\Delta_g u=P'(|u|^2)u,\ \ u(0)=u_0,
    \end{align*}
    where, for any finite $p$, $u\in L_{\loc}^p(\R, L^\infty(\M))$.  Moreover, for any $T>0$, the flow map $u_0\mapsto u$ is Lipschitz continuous from bounded sets of $H^1(\M)$ into $Y_T$.
\end{theorem}

In dimension $d=3$ the analysis requires the use of multilinear estimates and Bourgain spaces, which we briefly recall. The Bourgain space $X^{s, b}$ is equipped with the norm whose square is given by
\begin{align*}
    \norm{u}^2_{X^{s,b}}= \sum_{k\in \N}\left\langle\lambda_k\right\rangle^s \Norm{\left\langle\tau+\lambda_k\right\rangle^b \widehat{\mathbb{P}_k}(\tau) u}^2_{L^2(\R_{\tau}\times \M)}=\norm{u^{\#}}^2_{H^{b}(\R,H^s(\M))}
\end{align*}
where $u=u(t,x)$ with $(t, x)\in\R\times\M$, $u^{\#}(t)=e^{-it\Delta}u(t)$ and $ \widehat{\mathbb{P}_k u}(\tau)$ denotes the Fourier transform of $\mathbb{P}_k u$ with respect to the time variable, the latter being the projector onto $e_k$. The associated restriction space $X_T^{s, b}$ is the corresponding restriction space with the norm
\begin{align*}
    \norm{u}_{X_T^{s, b}}=\inf\big\{\norm{\tilde{u}}_{X^{s, b}}\ |\ \tilde{u}=u \text{ on } (0,T)\times \M \big\}.
\end{align*}
The following technical assumption ensures that the Cauchy problem is wellposed in $H^1(\M)$. It yields a bilinear loss of $s_0<1$.
\begin{assump}{WP}\label{assumBilWP}
There exists $C>0$ and $0\leq s_0<1$ such that for any $f_1$, $f_2\in L^2(\M)$ satisfying
\begin{align*}
f_j=\mathbbm{1}_{\sqrt{1-\Delta}\in [N_j,2N_j[}(f_j), \quad j=1,2,3,4 
\end{align*}
one has the following bilinear estimates: if $u_j(t)=e^{it\Delta}f_j$, $j=1,2$ then
\begin{align}\label{inegbilin}
\norm{u_1 u_2}_{L^2([0,T]\times \M)} \leq C \min (N,L)^{s_0} \norm{f_1}_{L^2(\M)}\norm{f_2}_{L^2(\M)}.
\end{align}
\end{assump}
The above assumption is satisfied in the following cases (here $s_0+$ means any $s>s_0$):
\begin{itemize}
    \item $\mathbb{T}^3$ with $s_0=1/2+$, see \cite{Bou93I}.
    \item The irrational torus $\R^3/(\theta_1 \Z \times \theta_2 \Z \times \theta_3\Z)$ with $\theta_i \in \R$, for which an estimate with $s_0=2/3+$ has been obtained in \cite{Bou07}.
    \item $S^3$ with $s_0=1/2+$, see \cite{BGT05}.
    \item $S^2\times S^1$ with $s_0=3/4+$, see \cite{BGT05}.
\end{itemize}

The well-posedness in each one of the aforementioned cases was studied in the corresponding cited article. We summarize them in the following result, which states the existence for a defocusing nonlinearity of degree $3$ of the form $f(u)=\alpha u+ \beta |u|^2u$ with $\alpha>0$, $\beta \geq 0$.
\begin{proposition}\cite[Proposition 2.1]{Lau10:NLS3d}\label{prop:wellposedXsb}
Let $T>0$ and $s\geq 1$. Assume that $\M$ satisfies Assumption \ref{assumBilWP}. Then, for every $g\in L^2([0,T], H^s(\M))$ and $u_0 \in H^s(\M)$, there exists a unique solution $u$ on $[0,T]$ in $X_T^{s,b}$ to the Cauchy problem 
\begin{align}\label{dampedeqn}
\left\{
\begin{array}{rcl}
i\partial_t u+\Delta u-\alpha u-\beta|u|^2u=g & \textnormal{ on } [0,T]\times M,\\
u(0)=u_{0} \in H^s.
\end{array}
\right.
\end{align}
Moreover the flow map
\begin{align*}
\left\{
\begin{array}{rcl}
F : H^s(\M) \times L^2([0,T],H^s(\M))&\longrightarrow & X^{s,b}_{T}\\
            (u_0,g) &\longmapsto   &  u
\end{array}\right.
\end{align*}
is Lipschitz on every bounded subset.
\end{proposition}

\subsubsection{Observability inequality} Let us first introduce some notation. Given $\omega$ which satisfies the \ref{NLS:assumGCC}, by a compactness argument \cite[Lemma A.3]{LL24} we can always find $\omega_0\Subset\omega$ which satisfies the \ref{NLS:assumGCC} as well and a smooth function $b_{\omega}$ supported in $\omega$ with $b_{\omega}=1$ on $\omega_0$. To this function, we will associate the multiplication operator $\bC$ defined as
\begin{align}\label{def:operatorC}
    \bC: z\in H^s(\M)\longmapsto b_{\omega}z\in H^s(\M),
\end{align}
which is linear and continuous for any $s\in\R$. We will use this operator from now onward.

The observability inequality for the linear Schrödinger equation under the \ref{NLS:assumGCC} is due to Lebeau \cite{Leb92}. Although it was obtained in the more complicated case of boundary observability, the same result holds for internal observation.

\begin{theorem}\cite{Leb92}\label{thm:schr-obs} 
    Let $s\in [0, 2]$. Let $\omega$ satisfy the \ref{NLS:assumGCC} and let $\bC$ be as in \eqref{def:operatorC}. Then, for every $T>0$, there exists $C=C(\omega, T, s)>0$ such that for any $v_0\in H^s(\M)$ it holds
    \begin{align*}
        \norm{v_0}_{H^s(\M)}^2\leq C\int_0^T \norm{\bC e^{it\Delta_g}v_0}_{H^s(\M)}^2dt.
    \end{align*}
\end{theorem}
\begin{proof}
    The result for $s=0$ follows from Lebeau \cite{Leb92}. For $s=2$, we solve the linear Schrödinger equation with initial data $z_0=(1-\Delta_g)v_0$ and by noticing that $e^{i\cdot\Delta_g}$ and $(1-\Delta_g)$ commute, by applying the known observability in $L^2(\M)$ and the definition of Sobolev norm we get
    \begin{align*}
        \norm{v_0}_{H^2(\M)}^2\leq C\int_0^T \norm{\bC e^{it\Delta_g}v_0}_{H^2(\M)}^2dt+C\int_0^T\norm{[\bC, (1-\Delta_g)]e^{it\Delta_g}v_0}_{L^2(\M)}^2dt.
    \end{align*}
    Since the commutator $[\bC, (1-\Delta_g)]$ is a pseudodifferential operator of order $-1$, the second term in the right-hand side is of lower order and hence it can be removed by a classical compactness-uniqueness argument (see \cite[Section 4]{LL24}, for instance), proving the case $s=2$. The intermediate case $s\in (0, 2)$ follows by linear interpolation.
\end{proof}

\subsection{Propagation of regularity} In this section we recall some propagation of regularity results for the NLS which are key to verify the compactness-related hypotheses of \cref{thm:mainabs-analytic}. The following propagation result has been essentially proved by Dehman-Gérard-Lebeau \cite{DGL06}. Although it is stated in dimension $d=2$ and including the more involved subcritical regularity, it is clear from their proof that it can be adapted to higher dimensions. Below we state the result for any dimension in an algebra and we give the proof for the convenience of the reader.

\begin{proposition}\label{thm:propregularity-highreg} Let $T>0$, $d\in \N$ and $s>d/2$. Let $u\in C^0([0, T], H^s(\M))$ be a solution of the NLS 
\begin{align}
    \begin{array}{c}
        i\partial_t u+\Delta_g u=f(u).
    \end{array}
\end{align}
Assume that $\omega$ satisfies the \ref{NLS:assumGCC} and let $\bC$ be defined as in \eqref{def:operatorC}. If $\bC u\in L^2([0, T], H^{s+\nu}(\M))$ for some $\nu>0$, then $u\in C^0([0, T], H^{s+\nu}(\M))$ and there exists a constant $C>0$, which depends in all the given parameters and $\norm{u}_{L^\infty([0, T]\times\M)}$, such that
\begin{align}\label{thm:ineq:propregularity}
    \norm{u}_{C^0([0, T], H^{s+\nu}(\M))}\leq C\big(\norm{\bC u}_{L^2([0, T], H^{s+\nu}(\M))}+\norm{u}_{C^0([0, T], H^s(\M))}\big).
\end{align}
\end{proposition}

Regarding propagation results in the subcritical case, in dimension $d=2$ we have the following result.

\begin{proposition}\cite[Theorem 3]{DGL06}\label{prop:propregularityYT}
    Let $d=2$ and let $\omega$ satisfy the \ref{NLS:assumGCC}. Let $u\in C^0([0, T], H^1(\M))$ be a solution of \eqref{eq:NLS} with finite Strichartz norms such that $\partial_t u=0$ on $(0, T)\times \omega$. Then $u\in C^\infty\big((0, T)\times\M\big)$.
\end{proposition}

\begin{remark}
Although it is not written this way, it follows directly from their theorem as $\partial_t u=0$ on $(0, T)\times\omega$ implies that $u$ satisfies the elliptic equation $\Delta_g u=P'(|u|^2)u$ on $(0, T)\times\omega$ and thus the hypotheses are satisfied by elliptic regularity.
\end{remark}

In dimension $3$, the propagation results have been adapted in the low-regularity framework by means of Bourgain spaces.

\begin{proposition}{\cite[Corollary 5.3]{Lau10:NLS3d}}\label{prop:propregularityXsb}
    Let $d=3$ and suppose that $\M$ satisfies \cref{assumBilWP}. Let $1/2<b\leq 1$. Let $\omega$ satisfy the \ref{NLS:assumGCC} and let $u\in X_T^{1,  b}$ be a solution of \eqref{eq:NLS} such that $\partial_t u=0$ on $(0, T)\times \omega$. Then $u\in C^\infty\big((0, T)\times\M\big)$.
\end{proposition}

\subsubsection{Linear propagation results} The main tool in the proof of \cref{thm:propregularity-highreg} is the linear propagation result \cite[Proposition 13]{DGL06}. It states that, for solutions to the linear Schrödinger equation, we can microlocally propagate higher regularity along the bicharacteristic flow associated to the symbol $p(x, \xi)=|\xi|_x^2$, see \cref{A:pseudodiff}. We point out that, as we work on the full time interval $[0, T]$, we replace the $L_{\loc}^2$ hypothesis by an $L^2$ assumption with respect to the aforementioned result.

\begin{proposition}\label{prop:regularity-flow}
Consider $T>0$ and let $u\in C^0([0, T], H^s(\M))$, with $s\in \R$, be a solution of
\begin{align*}
    i\partial_t u+\Delta_g u=h\in L^2([0, T], H^s(\M)).
\end{align*}
Given $\rho_0=(x_0, \xi_0)\in T_0^*\M$, we assume that there exists a $0$-order pseudodifferential operator $\psi(x, D_x)$, elliptic in $\rho_0$, such that $\psi(x, D_x)u\in L^2([0, T], H^{s+\nu}(\M))$ for some $\nu\leq 1/2$. Then for every $\rho_1\in \Phi_{\rho_0}(t)$, the bicharacteristic ray starting at $\rho_0$, there exists a $0$-order pseudodifferential operator $\eta(x, D_x)$, elliptic in $\rho_1$, such that $\eta(x, D_x)u\in L^2([0, T], H^{s+\nu}(\M))$. Moreover, there exists $C>0$ such that
\begin{multline}\label{ineq:PSIdopropag}
    \norm{\eta(x, D_x)u}_{L^2([0, T], H^{s+\nu}(\M))}^2\leq C\left(\norm{\psi(x, D_x)u}_{L^2([0, T], H^{s+\nu}(\M))}^2\right.\\+\left.\norm{u}_{C^0([0, T], H^s(\M))}^2+\norm{h}_{L^2([0, T], H^s(\M))}^2\right).
\end{multline}
\end{proposition}
\begin{proof}
First of all, we regularize $u$ by introducing $\J_n:=\big(1-\tfrac{1}{n^2}\Delta_g)^{-1}$ which belongs to $\Psi^{-2}(\M)$ for each $n\in \N$, and then we set $u_n:=\J_n u$ and $h_n:=\J_n h$. Note that $u_n\in C^0([0, T], H^{s+2}(\M))$ and $u_n\to u$ in $L^\infty([0, T], H^s(\M))$. We divide the proof in three steps.
\medskip
\paragraph{\emph{Step 1. Commutator estimate.}} We will carefully choose a time-independent pseudodifferential operator $A=A(x, D_x)$ of order $(2r-1)$ where $r=s+\nu$.

Let us denote $L=i\partial_t+\Delta$. By integration by parts, we have the following commutator identity
\begin{multline}\label{eq:commLA}
    \inn{Lu_n, A^*u_n}_{L^2([0, T]\times \M)}-\inn{Au_n, L u_n}_{L^2([0, T]\times \M)}\\=\inn{[A, \Delta]u_n, u_n}_{L^2([0, T]\times \M)}+\left.i\inn{Au_n, u_n}_{L^2(\M)}\right|_0^T.
\end{multline}
By construction, $A$ is of order $2r-1=2s+2\nu-1\leq 2s$, and so $[A, \Delta]$ is of order $2r\leq 2s+1$. First of all, we observe that the right-hand side of \eqref{eq:commLA} is uniformly bounded with respect to $n\in \N$. Indeed, as $(u_n)$ and $(h_n)$ are both uniformly bounded in $C([0, T], H^s(\M)$ and $L^2([0, T], H^s(\M))$, respectively, we get
\begin{align*}
    |\inn{Au_n, Lu_n}_{L^2{([0, T]\times\M)}}|=|\inn{\Ld_{-r+1/2}Au_n, \Ld_{r-1/2}h_n}_{L^2}|\lesssim \norm{u_n}_{L^\infty(H^s)}\norm{h_n}_{L^2(H^s)}.
\end{align*}
We used that $\Ld_{-r+1/2}A$ is of order $2r-1-r+1/2=r-1/2\leq s$ and so, in particular, it maps $H^s$ into $L^2$. The term $\inn{Lu_n, A^*u_n}$ can be bounded in the same way as before. To estimate the terms at $t=T$ and $t=0$ on the left-hand side of \eqref{eq:commLA}, we note that
\begin{align*}
    |\inn{Au_n(T), u_n(T)}_{L^2(\M)}|=|\inn{\Ld_{-s}Au_n(T), \Ld_{s}u_n(T)}_{L^2(\M)}|&\lesssim\norm{u_n(T)}_{H^s}^2\\
    &\lesssim\norm{u_n}_{L^\infty(H^s)}^2
\end{align*}
and the term $\inn{Au_n(0), u_n(0)}_{L^2(\M)}$ is bounded similarly. Gathering the above estimates, we get the following estimate independent of $n\in \N$,
\begin{align}\label{ineq:commutatorB}
    \left|\int_0^T \inn{[A, \Delta]u_n, u_n}_{L^2(\M)}dt\right|\lesssim\norm{u_n}_{L^\infty(H^s)}^2+\norm{h_n}_{L^2(H^s)}^2.
\end{align}
\medskip
\paragraph{\emph{Step 2. Microlocal propagation.}} Take $\rho_1=\Phi_t(\rho_0)$. We want to transport a symbol supported near $\rho_1$ along the flow generated by $p(x, \xi)=|\xi|_x^2$, up to a remainder localized near $\rho_0$, which carries information from $\rho_0$. Let $V_1$ be a conical neighborhood of $\rho_1$. Using that the bicharacteristic flow is time-reversible, by \cref{lem:propagation-symbol} backwards, there exists a conical neighborhood $V_0$ of $\rho_0$, such that for any symbol $\mathfrak{c}=\mathfrak{c}(x, \xi)$ of order $r$ supported in $V_1$, which we further choose to be elliptic at $\rho_1$, we can find another symbol $\mathfrak{a}=\mathfrak{a}(x, \xi)$ of order $2r-1$ such that
\begin{align*}
    \tfrac{1}{i}H_{p} \mathfrak{a}(x, \xi)=|\mathfrak{c}(x, \xi)|^2+\mathfrak{r}(x, \xi),
\end{align*}
where $\mathfrak{r}=\mathfrak{r}(x, \xi)$ is a symbol of order $2r$ supported in $V_0$. Now, we choose $A(x, D_x)$ to be a pseudodifferential operator of principal symbol $\mathfrak{a}$. If $C$ and $R$ denote some pseudodifferential operators whose principal symbols are $\mathfrak{c}$ and $\mathfrak{r}$, respectively, for some $K\in \Psi^{2s-1}(\M)$ it holds
\begin{align}\label{eq:commutator-identity}
    [A,\Delta]=C^*C+R+K.
\end{align}
\medskip

\paragraph{\emph{Step 3. Estimates.}} By local parametrix \cref{lem:localparametrix}, there exist a pseudodifferential operator $\psi^\dagger$ of order $0$ elliptic at $\rho_0$ and a cutoff $\chi$ equal to $1$ in a conic neighborhood of $\rho_0$ such that $\psi^\dagger(x, D_x)\psi(x, D_x)=\Op(\chi)+K_1$ with $K_1$ smoothing. As $r$ is supported in $V_0$, by symbolic calculus \cref{app:prop:psido-algebra}, we can write
\begin{align*}
    \inn{Ru_n, u_n}_{L^2}=\inn{\widetilde{R}\Ld_{s+\nu}\psi(x, D_x)u_n, \Ld_{s+\nu}\psi(x, D_x)u_n}_{L^2}+\inn{K_2u_n, u_n}_{L^2}
\end{align*}
where $\widetilde{R}:=\Ld_{-(s+\nu)}(\psi^{\dagger})^*R\psi^\dagger \Ld_{-(s+\nu)}\in\Psi^0(\M)$ and $K_2$ is smoothing. We thus obtain
\begin{align}\label{ineq:r-elliptic}
    |\inn{R(x, D_x)u_n, u_n}_{L^2(\M)}|\lesssim\norm{\psi(x, D_x)u_n}_{H^{s+\nu}(\M)}^2+\norm{u_n}_{H^s(\M)}^2.
\end{align}
Using identity \eqref{eq:commutator-identity} followed by estimates \eqref{ineq:commutatorB} and \eqref{ineq:r-elliptic}, we have
\begin{multline*}
    \int_0^T \norm{C(x, D_x)u_n}_{L^2(\M)}^2dt\\\lesssim\norm{\psi(x, D_x)u_n}_{L^2([0, T], H^{s+\nu}(\M))}^2+\norm{u_n}_{L^\infty([0, T], H^s(\M))}^2+\norm{h_n}_{L^2([0, T], H^s(\M))}^2,
\end{multline*}
uniformly in $n\in \N$. Since $u_n(t)\to u(t)$ in $H^s(\M)$ for each $t\in [0, T]$ and the convergence takes place in a Hilbert space, we get that $C(x, D_x)u\in L^2([0, T], L^2(\M))$ and
\begin{align*}
    \int_0^T \norm{C(x, D_x)u}_{L^2(\M)}^2dt\leq \liminf_{n\to\infty} \int_0^T \norm{C(x, D_x)u_n}_{L^2(\M)}^2dt.
\end{align*}
The proof concludes by taking $\eta(x, D_x):=\Ld_{-(s+\nu)}C(x, D_x)$.
\end{proof}

By a partition of unity argument, we can use the previous result to propagate any gain on regularity to the whole manifold from any region that satisfies the \ref{NLS:assumGCC}.

\begin{corollary}\label{cor:regularity-estimate}
    With the notations and assumptions of \cref{prop:regularity-flow}, let $\omega$ satisfy the \ref{NLS:assumGCC} and let $\bC$ be defined as in \eqref{def:operatorC}. If $u\in L^2([0, T], H^{s+\nu}(\omega))$, then $u\in L^2([0, T], H^{s+\nu}(\M))$ and there exists a constant $C>0$ such that
    \begin{align}\label{cor:ineq:regularity-estimate}
        \norm{u}_{L^2([0, T], H^{s+\nu}(\M))}^2\leq C\big(\norm{\bC u}_{L^2([0, T], H^{s+\nu}(\M))}^2+\norm{u}_{C^0([0, T], H^s(\M))}^2+\norm{h}_{L^2([0, T], H^s(\M))}^2\big).
    \end{align}
\end{corollary}

\begin{proof}
    For every $\rho\in S^*\M$, the \cref{NLS:assumGCC} gives $\tau\in [0, T_0)$ such that $\rho_0=\Phi_{-\tau}(\rho)\in S^*\omega_0$. Since $\bC=b_\omega$ is elliptic at $\rho_0$ and $\bC u\in L^2([0, T], H^{s+\nu}(\M))$ by hypothesis, \cref{prop:regularity-flow} yields a $0$-order pseudodifferential operator $\eta_\rho$ elliptic at $\rho$ such that
    \begin{multline}\label{ineq:regularityest-1}
        \norm{\eta_\rho(x, D_x)u}_{L^2([0, T], H^{s+\nu}(\M))}^2\\\leq C_\rho\left(\norm{\bC u}_{L^2([0, T], H^{s+\nu}(\M))}^2+\norm{u}_{C^0([0, T], H^s(\M))}^2+\norm{h}_{L^2([0, T], H^s(\M))}^2\right).
    \end{multline}
    By compactness, we can choose finitely many $\rho_1,\ldots, \rho_N$ such that $S^*\M$ can be covered by neighborhoods $V_{1}^S,\ldots, V_{N}^S$, where $\eta_j(x, D_x):=\eta_{\rho_j}(x, D_x)$ is elliptic on $V_j^S$. For each $j=1,\ldots, N$, let us define $V_j$ to be a conic lift of $V_j^S$ to $T^*\M$, that is,
    \begin{align*}
        V_j:=\{(x, \xi)\in T_0^*\M\ |\ (x, \xi/|\xi|)\in V_j^S,\ |\xi|\geq 2\}.
    \end{align*}

    Let $(\chi_j)_{j=1}^N\subset C^\infty(S^*\M)$ be a microlocal partition of unity with $\sum_{j=1}^N \chi_j^2=1$ and $\supp\chi_j\subset V_{j}^S$. By taking $\varrho\in C^\infty([0, +\infty)$ with $\varrho(\tau)=0$ for $\tau\leq 1$ and $\varrho(\tau)=1$ for $\tau\geq 2$, we can extend $\chi_j$ to act on $T_0^*\M$ by setting $\widetilde{\chi}_j(x, \xi)=\varrho(|\xi|)\chi_j\left(x, \frac{\xi}{|\xi|}\right)$. It is indeed a symbol $\widetilde{\chi}_j\in S_{\phg}^0(T^*\M)$ with $\supp\widetilde{\chi}_j\subset V_j$ and
    \begin{align*}
        \sum_{j=1}^N \big(\widetilde{\chi}_j(x, \xi)\big)^2=1\ \text{ for } |\xi|\geq 2,
    \end{align*}
    and it vanishes for $|\xi|\leq 1$.
    Let $\Theta_j\in \Psi^0(\M)$ be a quantization of $\widetilde{\chi}_j\mathfrak{c}_j$, where $\mathfrak{c}_j$ is the principal symbol of $\eta_j(x, D_x)$. Let us define
    \begin{align*}
        \Upsilon:=\Theta_1^*\Theta_1+\ldots+\Theta_N^*\Theta_N\in \Psi^0(\M),
    \end{align*}
    which is elliptic and positive. Indeed, given that each $\eta_j(x, D_x)$ is elliptic, the principal symbol of $\Upsilon$ satisfies
    \begin{align*}
        \nu_\Upsilon=\sum_{j=1}^N \widetilde{\chi}_j^2|c_j|^2\geq \kappa\sum_{j=1}^N \widetilde{\chi}_j^2=\kappa>0\ \text{ for } |\xi|\geq 2,
    \end{align*}
    where $\kappa:=\min_{j=1,\ldots, N}\min_{V_j^S} |\mathfrak{c}_j|^2$. Applying the sharp G{\aa}rding's inequality \cref{app:thm:garding} to $A:=\Ld_{s+\nu}\Upsilon\Ld_{s+\nu}\in \Psi^{2(s+\nu)}(\M)$ and using symbolic calculus, we get
    \begin{align}\label{ineq:regularityest-2}
        \norm{u}_{L^2([0, T], H^{s+\nu}(\M))}^2\lesssim \sum_{j=1}^N\norm{\Theta_j(x, D_x)u}_{L^2([0, T], H^{s+\nu}(\M))}^2+\norm{u}_{L^2([0, T], H^{s+\nu-\frac{1}{2}}(\M))}^2.
    \end{align}
    By writting, $\Theta_j=\Op(\widetilde{\chi}_j)\eta_j+K_j$ with $K_j\in \Psi^{-1}(\M)$, we have 
    \begin{align*}
        \norm{\Theta_j(x, D_x)u}_{L^2([0, T], H^{s+\nu}(\M))}^2\lesssim \norm{\eta_j(x, D_x)u}_{L^2([0, T], H^{s+\nu}(\M))}^2+\norm{u}_{L^2([0, T], H^{s+\nu-1}(\M))}^2.
    \end{align*} Since $\nu\leq 1/2$ it holds $s+\nu-\tfrac{1}{2}\leq s$, therefore, using \eqref{ineq:regularityest-1} for each $\eta_j$ and summing over $j$, we get
    \begin{multline}\label{ineq:regularityest-3}
        \sum_{j=1}^N \norm{\Theta_j(x, D_x)u}_{L^2([0, T], H^{s+\nu}(\M))}^2\\
        \lesssim \norm{\bC u}_{L^2([0, T], H^{s+\nu}(\M))}^2+\norm{u}_{C^0([0, T], H^s(\M))}^2+\norm{h}_{L^2([0, T], H^s(\M))}^2.
    \end{multline}
    By plugging \eqref{ineq:regularityest-3} into \eqref{ineq:regularityest-2}, we get the desired estimate \eqref{cor:ineq:regularity-estimate}.
\end{proof}

\subsubsection{Nonlinear propagation} We are now in position to prove the nonlinear propagation result.
\begin{proof}[Proof of \cref{thm:propregularity-highreg}]
    First, assume that $0<\nu\leq 1/2$. With the aid of Duhamel's formula, let us split the solution $u$ into its linear and nonlinear part as
    \begin{align*}
        u(t)=e^{it\Delta_g}u_0-i\int_0^t e^{i(t-s)\Delta_g}f(u(s))ds=u_{\lin}(t)+u_{\Nlin}(t).
    \end{align*}    
    For each $n\in \N$, let us consider $\J_n:=\big(1-\tfrac{1}{n^2}\Delta_g)^{-1}\in \Psi^{-2}(\M)$ and let us introduce the regularization $u^n:=\J_nu$. Let us also denote by $u_{\lin}^n$ and $u_{\Nlin}^n$ the corresponding regularized linear and nonlinear part of $u^n$.

    As we will employ the observability later on, let us note that $\bC u^n=\bC\mc{J}_n u=\mc{J}_n\bC u+[\mc{J}_n,\bC]u$. Moreover, we have $[\mc{J}_n,\bC]=\frac{1}{n^2}\mc{J}_n[\Delta_g,\bC]\mc{J}_n$. Some computations lead us to $\norm{\mc{J}_n}_{\mc{L}(H^{s+\nu}(\M))}\leq 1$ and $\norm{\tfrac{1}{n^2}\J_n}_{\mc{L}(H^s(\M), H^{s+2}(\M))}\leq 1$. Further, since $[\Delta_g,\bC]\in \Psi^{1}(\M)$ and $H^{s+1}(\M)\hookrightarrow H^{s+\nu}(\M)$, we get that $[\Delta_g,\bC]\in {\mc{L}(H^{s+2}(\M), H^{s+\nu}(\M))}$. Therefore, as $u\in L^2([0, T], H^{s+\nu}(\M))$ due to the corresponding smoothing effect, we get uniformly in $n$
    \begin{align*}
      \norm{\bC u^n}_{L^2([0, T], H^{s+\nu})}&\lesssim\norm{\bC u}_{L^2([0, T], H^{s+\nu}(\M))}+\norm{\mc{J}_n[\Delta_g,\bC]\frac{\mc{J}_n}{n^2} u}_{L^\infty([0, T], H^{s+\nu}(\M))}\\ 
      &\lesssim\norm{\bC u}_{L^2([0, T], H^{s+\nu}(\M))}+\norm{u}_{L^\infty([0, T], H^s(\M))}.
    \end{align*}
    Since $u\in C^0([0, T], H^s(\M))$ and we are in an algebra, using \cref{lem:comp-reg} (recall that $f(0)=0$) we have $\norm{f(u)}_{L^2(H^s)}\leq C\norm{u}_{L^2(H^s)}$ with $C$ depending only on $\norm{u}_{L^\infty([0, T]\times \M)}$, which is finite due to Sobolev embedding. By \cref{cor:regularity-estimate}, we get
    \begin{align*}
        \norm{u}_{L^2([0, T], H^{s+\nu}(\M))}^2&\lesssim \norm{\bC u}_{L^2([0, T], H^{s+\nu}(\M))}^2+\norm{u}_{L^\infty([0, T], H^s(\M))}^2+\norm{f(u)}_{L^2([0, T], H^s(\M))}^2\\
        &\lesssim \norm{\bC u}_{L^2([0, T], H^{s+\nu}(\M))}^2+\norm{u}_{L^\infty([0, T], H^s(\M))}^2
    \end{align*}
    
    Let us first prove that $u_{\Nlin}^n$ is uniformly bounded in $L^\infty(H^{s+\nu})$. As before, since $u(t)\in H^{s+\nu}(\M)$ a.e. $t\in [0, T]$ and we are in an algebra, using \cref{lem:comp-reg} we have $\norm{f(u(t))}_{H^{s+\nu}(\M)}\leq C\norm{u(t)}_{H^{s+\nu}(\M)}$ a.e. $t\in [0, T]$, with $C$ only depending on $\norm{u}_{L^\infty([0, T]\times\M)}$. Since $L^2\hookrightarrow L^1$ in finite measure spaces, we obtain
    \begin{align*}
        \norm{u_{\Nlin}^n}_{L^\infty([0, T], H^{s+\nu}(\M))}&\leq \norm{u_{\Nlin}}_{L^\infty([0, T], H^{s+\nu}(\M))}\\
        &\lesssim \int_0^T \norm{f(u(t))}_{H^{s+\nu}(\M)}dt\lesssim \int_0^T \norm{u(t)}_{H^{s+\nu}(\M)}dt\lesssim \norm{u}_{L^2([0, T], H^{s+\nu}(\M))}.
    \end{align*}
     We are now in position to treat the linear part by employing the observability inequality. More precisely, we use the observability inequality for the linear Schrödinger equation \cref{thm:schr-obs}, for any $t\in [0, T]$ we have
    \begin{align*}
        \norm{u_{\lin}^n(t)}_{H^{s+\nu}(\M)}^2&=\norm{u_0^n}_{H^{s+\nu}}^2\\
        &\leq \mathfrak{C}_{\text{obs}}^2\int_0^T \norm{\bC u_{\lin}^n(t)}_{H^{s+\nu}}^2dt\\
        &\leq 2\mathfrak{C}_{\text{obs}}^2\int_0^T \norm{\bC u^n(t)}_{H^{s+\nu}}^2dt+2\mathfrak{C}_{\text{obs}}^2\int_0^T \norm{\bC u_{\Nlin}^n(t)}_{H^{s+\nu}}^2dt\\
        &\leq C\big(\norm{\bC u}_{L^2([0, T], H^{s+\nu}(\M))}^2+\norm{u}_{L^\infty([0, T], H^s(\M))}^2\big),
    \end{align*}
    where we used all the previous estimates to get the last inequality. As the previous estimate is valid for each $t\in [0, T]$, it follows that $u^n$ is uniformly bounded in $L^\infty([0, T], H^{s+\nu}(\M))$. Moreover, due to the fact that $\norm{u-u^n}_{L^\infty([0, T], H^{s}(\M))}\to 0$ and leveraging that $H^{s}(\M)$ is a Hilbert space, we get that $u(t)\in H^{s+\nu}(\M)$ and 
    \begin{align*}
        \norm{u(t)}_{H^{s+\nu}(\M)}\leq \liminf \norm{u^n(t)}_{H^{s+\nu}(\M)},
    \end{align*}
    for any $t\in [0, T]$, showing that $u$ is uniformly bounded in $C^0([0, T], H^{s+\nu}(\M))$.
    
    If $\nu>1/2$, then we pick $0<\nu'\leq 1/2$ so that $k\nu'=\nu$ for some $k\in\N$ and iterate the previous argument $k$-times with $\nu'$ as a parameter to conclude that $u\in C^0([0, T], H^{s+\nu}(\M))$ along with estimate \eqref{thm:ineq:propregularity}.     
\end{proof}

\subsection{Uniform observability at high-frequency} Our aim in this section is to verify Assumption \ref{NLS:assumC}. More precisely, we will prove an observability inequality at high-frequency for the following linear Schrödinger equation with potential
\begin{align}\label{NLS-UC-FT:eq:nls-main}
\left\{\begin{array}{cl}
i\partial_t w+\Delta_g w=\Q_nD\widetilde{f}(v)w & \text{ in } [0, T]\times \M,\\
w(0)=w_0,&
\end{array}\right.
\end{align}
where $\widetilde{f}:=\chi f$ with $\chi\in C^\infty(\M; \R)$, and $w_0\in \Q_nH^s(\M)$.

\subsubsection{Propagation of compactness} Upon a straightforward modification of \cite[Proposition 15]{DGL06} regarding the Sobolev regularity, we have the following result about the propagation of the microlocal defect measure which is valid in any dimension.

\begin{proposition}\label{prop:mdm} Let $s\geq 1$. Let $L=i\partial_t+\Delta_g+R_0$ where $R_0(t, x, D_x)$ is a tangential pseudo-differential operator of order $0$ and $(u_n)_n$ a sequence of bounded functions $\norm{u_n}_{L^\infty([0, T], H^s(\M))}\leq C$ satisfying
\begin{align*}
    \norm{u_n}_{L^\infty([0, T], H^{s-1}(\M))}\xrightarrow[n\to\infty]{} 0\ \text{ and }\ \norm{Lu_n}_{L^2([0, T], H^{s-1}(\M))}\xrightarrow[n\to\infty]{} 0.
\end{align*}
Then, there exist a subsequence $(u_k)_k$ of $(u_n)_n$ and a positive measure $\mu$ on $(0, T)\times S^*\M$ such that, for every tangential pseudo-differential operator $A=A(t, x, D_x)$ of order $2s$, with principal symbol $\sigma(A)=a_{2s}(t, x, \xi)$,
\begin{align*}
    \inn{A(t, x, D_x)u_k, u_k}_{L^2([0, T]\times M)}\xrightarrow[k\to\infty]{} \int_{(0, T)\times S^*\M} a_{2s}(t, x, \xi)d\mu(t, x, \xi).
\end{align*}
Moreover, if $\Phi_t$ denotes the bicharacteristic flow on $S^*\M$, one has, for every $t\in \R$,
\begin{align*}
    \Phi_t(\mu)=\mu,
\end{align*}
namely, $\mu$ is invariant by the bicharacteristic flow 'at fixed $t$'.
\end{proposition}

\subsubsection{Uniform observability} Using the microlocal defect measure we now prove an observability inequality uniform with respect to the potential for solutions at high-frequency, that is, we verify that Assumption \ref{NLS:assumC} holds true. First, we verify the observability with potentials.

\begin{proposition}\label{prop:weakobsSchr}
    Let $T>0$, $M>0$, $d\in \N$ and $s>d/2$ be fixed. Let $\omega$ be an open set satisfying the \ref{NLS:assumGCC} and let $\bC$ be as in \eqref{def:operatorC}. Then there exist $n_0\in \N$ and $C>0$ such that for any $V_1$, $V_2\in \B_{M}^{[0, T]}(H^{s}(\M))$ and $n\geq n_0$, the following observability inequality holds
    \begin{align*}
        \norm{w_0}_{H^{s}}^2\leq C\int_0^T \norm{\bC w(t)}_{H^{s}}^2dt,
    \end{align*}
    for any $w_0\in \Q_n H^{s}(\M)$, where $w\in C^0([0, T], \Q_n H^s(\M))$ is solution of 
    \begin{align}\label{eq:weakobsSchr}
        \left\{\begin{array}{cc}
            i\partial_t w+\Delta_g w=\Q_n(V_1w+V_2\overline{w}) &~ \text{ in } (0, T)\times\M, \\
            w(0)=w_0. &
        \end{array}\right.
    \end{align}
\end{proposition}
\begin{proof}
    Let us momentarily assume that $d\geq 2$. We proceed by contradiction. Let us assume that there exist $n_j\to +\infty$, a sequence of potentials $(V_{1,j})$, $(V_{2,j})\subset \B_{M}^{[0, T]}(H^{s}(\M))$ and a sequence of solutions $(w_j)_j$ of \eqref{eq:weakobsSchr} in $C([0, T], \Q_{n_j}H^{s}(\M))$ associated to $(V_{1, j})_j$ and $(V_{2, j})_j$ with $\norm{w_{0, j}}_{H^{s}}=1$ such that
    \begin{align*}
        \begin{array}{ccc}
            \displaystyle\int_0^T \norm{\bC w_j(t)}_{H^{s}(\M)}^2dt\xrightarrow[j\to\infty]{} 0.
        \end{array}
    \end{align*}
    The uniform bounds on $(w_{0, j})$, $(V_{1, j})$ and $(V_{2, j})$ yield that that $(w_j)_j$ is an uniformly bounded sequence in $C([0, T], H^{s}(\M))$. Then, since $w_j(t)\in \Q_{n_j}H^s(\M)$ a.e. $t\in [0, T]$ we have
    \begin{align*}
        \norm{w_j}_{L^\infty([0, T], H^{s-1}(\M))}\leq \ld_{n_j}^{-1/2}\norm{w_j}_{L^\infty([0, T], H^s(\M))}
    \end{align*}
    from which $w_j\to 0$ in $L^\infty([0, T], H^{s-1}(\M))$.
    
    Using that $H^{s}(\M)$ is an algebra, we get
    \begin{align*}
        \norm{\Q_{n_j}V_{1, j}w_j}_{L^\infty([0, T], H^{s}(\M))}\leq \norm{V_{1, j}w_j}_{L^\infty([0, T], H^{s}(\M))}\leq M\norm{w_j}_{L^\infty([0, T], H^{s}(\M))}
    \end{align*}
    We conclude that $(\Q_{n_j}V_{1, j}w_j)_j$ is uniformly bounded. Hence, in a similar way as before,
    \begin{align*}
        \norm{\Q_{n_j}V_{1, j}w_j}_{L^2([0, T], H^{s-1}(\M))}\leq \ld_{n_j}^{-1/2}\norm{\Q_{n_j}V_{1, j}w_j}_{L^2([0, T], H^{s}(\M))}
    \end{align*}    
    and we get that $\Q_{n_j}V_{1, j}w_j$ converges strongly to $0$ in $L^2([0, T], H^{s-1}(\M))$. Similarly, we obtain that $\Q_{n_j}V_{2, j}\overline{w_j}$ converges strongly to $0$ in $L^2([0, T], H^{s-1}(\M))$ as well. We thus have
    \begin{align*}
        \left\{\begin{array}{rl}
            i\partial_t w_j+\Delta w_j=h_j\xrightarrow[j\to\infty]{} 0 & \text{ in } L^2([0, T], H^{s-1}(\M)),  \\
            \bC w_j\xrightarrow[j\to\infty]{} 0 & \text{ in } L^2([0, T], H^{s}(\M)),
        \end{array}\right.
    \end{align*}
    with $h_j:=\Q_{n_j}(V_{1, j}w_j+V_{2, j}\overline{w_j})$. By \cref{prop:mdm} we can attach a microlocal defect measure $\mu$ to $(w_j)_j$ in $L^2([0, T], H^{s}(\M))$. Further, since $\omega_0$ satisfies the \ref{NLS:assumGCC} and $\bC w_j\to 0$ in $L^2([0, T], H^{s}(\M))$, we get that $\mu=0$ on $S^*\omega_0$ and then, by invariance under the bicharacteristic flow, $\mu\equiv 0$ on $S^*\M$. Therefore $w_j$ converges to $0$ in $L_{\loc}^2([0, T], H^{s}(\M))$. 

    Let us pick $t_0\in [0, T]$ such that $w(t_0)$ goes strongly to $0$ in $H^s(\M)$. By uniqueness and linear estimates, we get that $w_j$ converges strongly to $0$ in $C^0([0, T], H^s(\M))$, which contradicts $\norm{w_j(0)}_{H^s(\M)}=1$.

    Finally, the case $d=1$ follows by applying the above reasoning to $\widetilde{w}=\Ld_{s-1}w$. Since $\Ld_{s-1}$ and $i\partial_t+\Delta_g$ commute, this allow us to apply \cref{prop:mdm} with $r=1$ and the rest goes in a similar way.
\end{proof}

As a corollary we verify that Assumption \ref{NLS:assumC} holds true.

\begin{corollary}\label{cor:weakobsassumC}
    Let $T>0$, $R>0$, $d\in \N$ and $s>d/2$. Let $\mathbb{V}=\B_{R}^{[0, T]}(H^s(\M))$ and $\chi\in C^\infty(\M, \R)$. Assume that $\omega$ satisfies the \ref{NLS:assumGCC} and let $\bC$ be as defined in \eqref{def:operatorC}. 
    Under these notations, there exists $n_0\in\N$ such that Assumption \ref{NLS:assumC} holds true with $\fk=-i\chi f$.
\end{corollary}
\begin{proof}
    Since multiplying by a smooth function is a bounded linear operator from $H^s(\M)$ into itself, let us assume for simplicity that $\chi\equiv 1$. If $z=x+iy\in\mathbb{C}$, let us define
    \begin{align*}
        \partial_z f:=\frac{1}{2}(\partial_x-i\partial_y)f\ \text{ and }\ \partial_{\overline{z}} f:=\frac{1}{2}(\partial_x+i\partial_y)f.
    \end{align*}
    By a slight variation of \cref{prop:NLS-assumF}, we get that for any $v\in \B_{R}^{[0, T]}(H^s(\M))$, both $\partial_z f(v)$ and $\partial_{\overline{z}} f(v)$ belong to $\B_{M}^{[0, T]}(H^s(\M))$ where $M$ is a constant that depends on $R$, coming from the composition estimates of $\partial_z f(v)$ and $\partial_{\overline{z}}f(v)$. Given that we can write
    \begin{align*}
        Df(v)w=\partial_{z}f(v)w+\partial_{\overline{z}}f(v)\overline{w},
    \end{align*}
    the conclusion follows as a direct application of \cref{prop:weakobsSchr} by taking $V_1=\partial_{z}f(v)$ and $V_2=\partial_{\overline{z}}f(v)$.
\end{proof}

\begin{remark}
    For instance, if $P'(z)=z$, which corresponds to a cubic nonlinearity $f(z)=|z|^2z=z^2\overline{z}$, then $\partial_z f(v)=2v$ and $\partial_{\overline{z}}f(v)=v^2$, from which follows that $Df(v)w=2|v|^2w+v^2\overline{w}$.
\end{remark}

\subsection{Finite determining modes} If we drop the analyticity on the nonlinearity, we can still show that the property of finite determining modes holds for the observed nonlinear Schrödinger equation \eqref{eq:NLS} with regular enough nonlinearity.

\begin{proposition}\label{prop:findetmodesSchrodinger}
 Let $d\in\N$ and $s>d/2$. With the notations of \cref{s3:preliminaries}, assume $\omega$ satisfies \ref{NLS:assumGCC} and that $f\in C^\infty(\mathbb{C}, \mathbb{C})$. For any $R_0>0$, there exists $n\in \N$ such that the following holds. Let $h\in  L^1([0,T], H^{s}(\M))$ and $g\in L^2([0, T], H^s(\M))$. Let $u$ and $\widetilde{u}$ be two solutions on $(0,T)$ of 
  \begin{align*}
    \left\{\begin{array}{rl}
        i\partial_t u+\Delta_g u=f(u)+h  &\ \text{ in } [0, T]\times \M,\\
        u=g &\ \text{ in } [0, T]\times \omega,
    \end{array}\right.
\end{align*}
such that $\norm{u(t)}_{H^s}\leq R_{0}$ and $\norm{\widetilde{u}(t)}_{H^s}\leq R_{0}$ for all $t\in [0, T]$. If $\P_n u(t)=\P_n \widetilde{u}(t)$ for all times $t\in [0, T]$, then $u(t)\equiv \widetilde{u}(t)$ for all $t\in [0, T]$.
\end{proposition}
\begin{proof}
We have already established in \cref{s3:preliminaries} that $A$ satisfies Assumption \ref{NLS:assumA}. If we pick $s^*\in (d/2, s)$, then $z\in \B_{R_0}^{[0, T]}(H^{s^*}(\M), \A)$ where $\A$ is a bounded ball in $H^{s}(\M)$. Moreover, $\A$ is compact in $H^s$ by Sobolev embedding. From \cref{prop:NLS-assumF}, Assumption \ref{assumF} is satisfied with $\fk:=-if$. Since $\omega$ satisfies the \ref{NLS:assumGCC}, Assumption \ref{NLS:assumC} is satisfied after \cref{cor:weakobsassumC} with $\bC$ defined as in \eqref{def:operatorC} and $\mathbb{V}=\B_{3R_0}^{[0, T]}(H^{s^*}(\M), \A)$. Then, we apply the abstract \cref{prop:finite-det-modes} with $\sigma^*=s^*/2$ instead of $\sigma$ to obtain $\Q_n(\widetilde{u}-u)(t)=0$ on $H^{s*}$ for every $t\in [0, T]$. Since $\Q_n$ is a linear bounded operator in both $H^{s^*}$ and $H^s$ with $H^s\hookrightarrow H^{s^*}$, the equality holds in $H^s$ as well.
\end{proof}

\subsection{Propagation of analyticity} Here we prove the main result for the NLS \cref{thm:NLS-analytic}.

\begin{proof}[Proof of \cref{thm:NLS-analytic}]
    Let $\omega_1\Subset\omega$ satisfying the \ref{NLS:assumGCC} \cite[Lemma A.14]{LL24} and let $\chi\in C_c^\infty(\M)$ be a smooth compactly supported function whose support is contained in $\omega$ and $\chi=1$ on $\omega_1$. 
    
    As the Schrödinger equation is observable from $\omega_1$ for any $T>0$ and being analytic is a local property, it is enough to prove that the restriction of $u$, solution to \eqref{eq:NLS}, restricted to any compact subinterval $[\delta, T-\delta]$ with $\delta>0$ is analytic. By translation in time, we can consider it on $[0, T-2\delta]\times\M$ and without loss of generality we can just relabel $T-2\delta$ by $T$ and assume that analyticity holds in a neighborhood of $(0, T)$. Hence, by hypothesis, $t\in (0, T)\mapsto \chi u(t)\in H^{s}(\M)$ is analytic, and so is $t\in (0, T)\mapsto \chi \partial_t u(t)\in H^{s}(\M)$. Let us also assume that $t\in [0, T]\mapsto u(t)\in H^s(\M)$ is bounded by some $R_0>0$. 
    
    By using the equation, we observe that on $(0, T)\times \omega$ we have
    \begin{align}\label{eq:ellipticSchr}
        \Delta_g(\chi u)=\chi f(u)+[\Delta,\chi]u-i\chi\partial_t u.
    \end{align}
    Since $u\in C^0([0, T], H^{s}(\M))$, we readily get that $[\Delta,\chi]u(t)\in H^{s-1}(\M)$ for each $t\in (0, T)$. Moreover, since $f$ is smooth, by \cref{lem:comp-reg} we have $f(u(t))\in H^{s}(\M)\hookrightarrow H^{s-1}(\M)$ for each $t\in (0, T)$. Thus, the right-hand side of \eqref{eq:ellipticSchr} belongs pointwise in time to $H^{s-1}$. Furthermore, since $t\in (0, T)\mapsto \chi u(t)\in H^{s}(\M)$ is analytic and bounded by $R_0$, by Cauchy's estimates \cref{appendix:prop:cauchy-est} applied to $\chi\partial_t u$ and global elliptic  estimates applied to \eqref{eq:ellipticSchr}, we get that there exists a constant $C=C(R_0)>0$ such that
    \begin{align*}
         \norm{\chi u}_{L^2([0, T], H^{1+s}(\M))}\leq C.
    \end{align*}
     We can invoke \cref{thm:propregularity-highreg} to obtain that $u\in \B_{R_1}^{[0, T]}(H^{1+s}(\M))$ for some $R_1>0$.
    
    Let $s^*\in (d/2, s)$ and let us consider
    \begin{align*}
        \A:=\{v\in H^{s}(\M)\ |\ \norm{v}_{H^{s}(\M)}\leq R_0\},
    \end{align*}
    which is a compact subset of $H^{s^*}(\M)$ by Sobolev embedding. Observe that $u(t)\in \A$ for all $t\in [0, T]$. Furthermore, by using the equation we see that $\partial_t u$ in particular belongs to $\B_{R_2}^{[0, T]}(H^{s-1}(\M))$ for some $R_2>0$. Let us then introduce
    \begin{align*}
        \K:=\{w\in \B_{R_0}^{[0, T]}(H^{s}(\M))\ |\ \norm{\partial_t w}_{L^2(H^{s-1})}\leq R_2\},
    \end{align*}
    which is compact in $C^0([0, T], H^{s^*}(\M))$ as a consequence of Aubin-Lions lemma (see \cref{app:thm:aubinlions}). Therefore, $u$ belongs to the compact set $\B_{R_0}^{[0, T]}(H^{s^*}(\M), \A)\cap \K$. Note that $\A$ and $\K$ are both stable under the projector $\P_n$, see \cref{rk:Acpctstable}.
    
    By compactness \cite[Lemma A.14]{LL24}, once again, let $\omega_1\Subset\omega_0$ satisfying the \ref{NLS:assumGCC} and let $\bC$ be defined as in \eqref{def:operatorC} accordingly. Observe that $z=(1-\chi)u=0$ on $\omega_1$, and thus $\bC z=0$ on $[0, T]\times\M$. Then, by the previous discussion, $z\in \B_{R_0}^{[0, T]}(H^{s^*}(\M), \A)\cap \K$ and it solves
    \begin{align*}
        \left\{\begin{array}{cl}
            \partial_t z=Az+\fk(z+\bh_1)+\bh_2 &~ \text{ in } [0, T], \\
            \bC z(t)=0 &~ \text{ for } t\in [0, T],
        \end{array}\right.
    \end{align*}
    with $A=i\Delta_g$, $\bC=b_{\omega_0}$, $\fk=-i(1-\chi)f$ and
    \begin{align*}
        \left\{\begin{array}{cccl}
            \bh_1: & t\in [0, T] & \longmapsto & \chi u(t)\in H^{s^*}(\M),  \\
            \bh_2: & t\in [0, T] & \longmapsto & [\Delta_g, \chi]u(t)\in H^{s^*}(\M).
        \end{array}\right.
    \end{align*}
    By the regularity properties obtained on $u$, we see that $\bh_1\in \B_{R_0}^{[0, T]}(H^{s^*}(\M), \A)\cap \K$ and $\bh_2\in \B_{R_3}^{[0, T]}(H^{s^*}(\M), \A)$, for some $R_3>0$. By hypothesis, an application of \cref{app:thm:h-ext} and compactness, there exists $\mu>0$ so that both $\bh_1$ and $\bh_2$ admit holomorphic extensions
    \begin{align*}
        \left\{\begin{array}{cccl}
            \bh_1: & t\in (0, T)+i(-\mu, \mu) & \longmapsto & \chi u(t)\in H^{s^*}(\M; \mathbb{C}^2),  \\
            \bh_2: & t\in (0, T)+i(-\mu, \mu) & \longmapsto & [\Delta_g, \chi]u(t)\in H^{s^*}(\M; \mathbb{C}^2).
        \end{array}\right.
    \end{align*}
    Moreover, given that $\Re\bh_1\in \B_{R_0}^{[0, T]}(H^{s^*}(\M))$, by shrinking $\mu>0$ if necessary, by continuity and compactness, we can assume that $\Re\bh_1(z)\in \mathbb{B}_{2R_0}(H^{s^*}(\M))$ for every $z\in [0, T]+i[-\mu, \mu]$.

    By \cref{prop:NLS-assumF}, $\fk$ satisfies Assumption \ref{NLS:assumFholom}. Then, by \cref{cor:weakobsassumC}, for any $0<\widetilde{T}<T$ there exists $n_0\in \N$ such that Assumption \ref{NLS:assumC} is satisfied with $\widetilde{T}$, $2R_0$, $n_0$, $\B_{6R_0}^{[0, \widetilde{T}]}(H^{s^*}(\M))$ and $\fk=-i(1-\chi) f$ (with $\chi$ of the corollary replaced by $1-\chi$). In particular, since a function with value in $\A+\A$ is bounded by $3R_0$ in the $H^{s^*}(\M)$-norm, Assumption \ref{NLS:assumC} also holds on the closed subset $\B_{6R_0}^{[0, \widetilde{T}]}(H^{s^*}(\M), \A+\A)\subset \B_{6R_0}^{[0, \widetilde{T}]}(H^{s^*}(\M))$.
    
    We are then in position to apply \cref{thm:mainabs-analytic} with $\sigma$, $(T, T^*)$ and $(R_0, R_1)$ of the theorem replaced by $s^*/2$, $(\widetilde{T}, T)$ and $(2R_0, R_3)$, respectively. We deduce that $t\in (0, T)\mapsto z(t)\in H^{s^*}(\M)$ is real analytic, and so is $t\in (0, T)\mapsto u(t)\in H^{s^*}(\M)$.

    Since $t\in (0, T)\mapsto u(t)\in H^{s^*}(\M)$ is analytic, so are $t\in (0, T)\mapsto \partial_t u(t)\in H^{s^*}(\M)$ (by \cref{prop:diff-analytic}) and $t\in (0, T)\mapsto u(t)-f(u(t))\in H^{s^*}(\M)$ (by \cref{appendix:thm:chain-rule}). Using the equation, we get that $t\in (0, T)\mapsto u(t)\in H^{2+s^*}(\M)$ is analytic.
\end{proof}

\subsection{Unique continuation} In this section we consider $u$ solution of the system
\begin{align}\label{eq:NLSucp}
        \left\{\begin{array}{cl}
            i\partial_tu+\Delta u=f(u)     &~~ \text{ in } (0, T)\times\M,  \\
            \partial_t u=0   &~~ \text{ on } (0, T)\times\omega,
        \end{array}\right.
\end{align}
where, we recall from case \ref{NLS:reg:2}, that the nonlinearity $f(u)=P'(|u|^2)u$ satisfies:
\begin{enumerate}
     \item if $d=2$ then $P$ is a polynomial function with real coefficients, satisfying $P(0)=0$ and the defocusing assumption $P'(r)\xrightarrow[r\to+\infty]{}+\infty$;
    \item if $d=3$, then $P'(r)=\al r+\beta$ with $\al>0$, $\beta\geq 0$, corresponding to the cubic nonlinearity.
\end{enumerate}
The purpose of this section is to prove \cref{thm:UCP-NLSdim2} and \cref{thm:UCP-NLSdim3}. We also prove the unique continuation result in an unbounded domain \cref{prop:UCP-unbounded} at the end of this section.

\subsubsection{On unique continuation for linear Schrödinger equation} We now recall the following unique continuation result in the context of Schrödinger equations due to Tataru-Robbiano-Zuily-Hörmander. We refer to \cite[Theorem 6.5]{LL19} for a quantitative statement that implies unique continuation.

\begin{theorem}\label{thm:UCP-linearschrodinger}
    Let $T>0$. Let $\M$ be a compact Riemannian manifold with (or without) boundary, $\Delta_g$ the Laplace-Beltrami operator on $\M$, and
    \begin{align*}
        P=i\partial_t+\Delta_g+V
    \end{align*}
    with $V\in L^\infty((0, T), W^{2, \infty}(\M))$. Assume that $V$ depends analytically on the variable $t\in (0, T)$.

     Let $\omega$ be a nonempty open subset of $\M$. Let $u_0\in H^2(\M)\cap H_0^1(\M)$ and associated solution $u$ of
    \begin{align*}
        \left\{\begin{array}{ll}
            Pu=0&\ \text{ in }\ (0, T)\times\Int{(\M)}\\
            u_{|_{\partial\M}}=0 &\ \text{ in }\ (0, T)\times \partial\M\\
            u(0)=u_0. &\ 
        \end{array} \right.
    \end{align*}
    Then, if $u$ satisfies $u=0$ on $[0,T]\times \omega$, then $u=0$ on $[0,T]\times \M$.
\end{theorem}

\begin{remark}
    A sharper unique continuation result was obtained by Filippas, Laurent, and Léautaud \cite{FLL24}, in which the analyticity assumption was relaxed to the Gevrey $2$ class.
\end{remark}

\subsubsection{Unique continuation for the nonlinear equation} We now come to prove \cref{thm:UCP-NLSdim2} and \cref{thm:UCP-NLSdim3}.

\begin{proof}[Proof of \cref{thm:UCP-NLSdim2}]
    Let $u$ be a solution of \eqref{eq:NLSucp}  which belongs to $C^0([0, T], H^1(\M))$ with finite Strichartz norms. Since $u$ solves the elliptic equation $\Delta_g u=f(u)$ in $(0, T)\times\omega$ and $f$ is subcritical, by elliptic regularity and bootstrap (see \cite[Theorem 9.19]{GT01}) we improve the regularity of $u$ in $(0, T)\times \omega$, which then allows us to apply the propagation of regularity result \cref{prop:propregularityYT} to obtain that $u$ is uniformly bounded in $C^0([0, T], H^s(\M))$ for any $s\in (1, 2]$, but fixed. We are then in position to apply \cref{thm:NLS-analytic} to obtain that $t\in (0, T)\mapsto u(t,\cdot)\in H^2(\M)$ is analytic.

   Actually, \cref{prop:propregularityYT} implies that $u$ belongs to $C^\infty\big((0, T)\times\M\big)$. Now, set $z=\partial_t u$ and observe that it solves
    \begin{align*}
        \left\{\begin{array}{cl}
            i\partial_t z+\Delta_g z=f'(u)z &~~ \text{ on }~ (0, T)\times \M, \\
            z=0 &~~ \text{ on }~ (0, T)\times \omega.
        \end{array}\right.
    \end{align*}
    Since $(t, x)\mapsto f'(u(t, x))$ is smooth and analytic in $t$, and thus bounded, we can apply \cref{thm:UCP-linearschrodinger} to obtain that $z\equiv 0$. This means that $u$ is constant in the time variable and hence it solves
    \begin{align*}
        -\Delta_g u+f(u)=0,\ \ x\in \M.
    \end{align*}
    Multiplying by $\overline{u}$ the above equation and integrating by parts, in the case that $P'(r)\geq C>0$ for every $r\geq 0$, we get
    \begin{align*}
        0\leq \int_\M |\nabla u|^2dx=-\int_\M P'(|u|^2)|u|^2dx\leq -C\int_\M |u|^2dx
    \end{align*}
    and thus $u$ must be equal to $0$.
\end{proof}

The proof in dimension $3$ is similar, we just point out the main differences.

\begin{proof}[Proof of \cref{thm:UCP-NLSdim3}]
    If $u$ is a solution of \eqref{eq:NLSucp} which belongs to $X_T^{1, b}$ with $b>1/2$, due to elliptic regularity arguments, $u$ is smooth in $(0, T)\times\omega$ and we can thus invoke the propagation of regularity result \cref{prop:propregularityXsb} to obtain that $u$ belongs to $X_T^{1+\nu, b}$ for $\nu\in (1/2, 1]$. By Sobolev embedding, it also belongs to $C^0([0, T], H^{1+\nu}(\M))$ and hence $t\in (0, T)\mapsto u(t, \cdot)\in H^{1+\nu}(\M)$ is analytic as a consequence of \cref{thm:NLS-analytic}. By setting $z=\partial_t u$, we can apply \cref{thm:UCP-linearschrodinger} to obtain $\partial_t u=0$ in $(0, T)\times \M$. Multiplying the resulting elliptic equation by $\overline{u}$ and integrating by parts we readily get
    \begin{align*}
        0\leq \int_\M |\nabla u|^2dx+\al\int_\M|u|^2dx+\beta\int_\M|u|^4dx=0
    \end{align*}
    and thus $u\equiv 0$.
\end{proof}

\subsubsection{Unique continuation on unbounded domains}\label{S3:UCP:unbounded} Here we provide an example of unique continuation for the NLS in an unbounded domain. Let $(\R^2, g)$ where the Riemannian metric $g$ satisfies
\begin{align*}
    \left\{\begin{array}{cl}
        \forall x\in \R^2  &~~ m\Id\leq g(x)\leq M\Id  \\
        \forall\al\in\N^2,\ \exists C_\al>0, \forall x\in\R^2 &~~ |\partial^\al g(x)|\leq C_\al. 
    \end{array}\right.
\end{align*}
After \cite[Theorem 5]{BGT04}, without any further geometric assumption $g$, it is known that the Schrödinger equation enjoys the same Strichartz estimates as \cref{thm:strichartz}. Let us consider
\begin{align}\label{eq:NLSunbounded}
    \left\{\begin{array}{cl}
        i\partial_t u+\Delta_g u=P'(|u|^2)u     &~~ (0, T)\times \R^2,  \\
        \partial_t u=0     &~~ (0, T)\times\omega, 
    \end{array}\right.
\end{align}
where $\omega$ satisfies the \ref{NLS:assumGCC}. After \cite[Remark A.5]{BGT04}, we consider solutions to \eqref{eq:NLSunbounded} which belong to the analogous space $Y_T$ in the same fashion as \cref{thm:wellposedH1} in the compact case.

\begin{proof}[Proof of \cref{prop:UCP-unbounded}]
    To ease notation, from now on we will write $f(u)=P'(|u|^2)u$. Up to making $R>0$ larger, we can assume that $\R^2\setminus B(0, R)\Subset \omega$. Since $\partial_t u$ vanishes on $(0, T)\times\big(\R^2\setminus B(0,R)\big)\subset (0, T)\times \omega $, we have $-\Delta u+f(u)=0$ with $u\in H^1(\R^2\setminus B(0, R))$. Since $f$ is subcritical, we can use elliptic regularity and bootstrap (see, for instance, \cite[Theorem 9.19]{GT01}) to show that $u=u(x)$ belongs to $C^{4}$ in the set $\R^2\setminus B(0, R)$. By Sobolev embedding and iteration, it belongs to $H^k$ on $\R^2\setminus B(0, R)$ for every $k\in\N$.

Let us consider $r_1>R$ and a cutoff function $\chi\in C_c^\infty(\R^2)$ such that
\begin{itemize}
    \item $\chi=1$ in $B(0, R)$,
    \item $\chi=0$ in $\R^2\setminus B(0, r_1)$,
    \item $\supp\nabla\chi\subset B(0, r_1)\setminus B(0, R)$.
\end{itemize}
We follow the same strategy as before. As $(1-\chi)u$ is supported in $\R^2\setminus B(0, R)$ and does not depend on the time variable, it is analytic as a map from $t\in (0, T)$ into $H^{s}(\R^2)$ for any $s\in (d/2, 2]$. By writing $u=\chi u+(1-\chi)u$, it remains to prove that $z:=\chi u$ is analytic.

\medskip
\paragraph{\emph{Step 1. Reduction to a fundamental domain.}} Let us take $R_1>r_2>r_1>R$ and let us introduce the square
\begin{align*}
    T_{R_1}=\{(x_1, x_2)\in \R^2\ |\ x_i\in [-R_1, R_1],\ i=1, 2\},
\end{align*}
which obviously contains the balls of radius $R$, $r_1$ and $r_2$. First, let us consider cutoff function $\psi\in C^\infty(\R^2, [0, 1])$ such that
\begin{itemize}
    \item $\psi=1$ in $B(0, r_1)$,
    \item $\psi=0$ in $T_{R_1}\setminus B(0, r_2)$,
    \item $\supp\nabla\psi\subset B(0, r_2)\setminus B(0, r_1)$.
\end{itemize}
If $g_{euc}$ denote the euclidean metric in $\R^2$, we set a new metric by $g_1=\psi g+(1-\psi)g_{euc}$. This metric is smooth satisfies $g_1=g$ on $\supp\chi\cup\supp\nabla\chi\subset B(0, r_1)$ and coincides with the euclidean metric $g_{euc}$ outside of the ball $B(0, r_2)$.

Let $\omega_1=T_{R_1}\cap \omega$. Let $\widetilde{\chi}\in C^\infty(\R^2)$ be another cutoff with the same properties as $\chi$ and $\widetilde{\chi}=1$ on $\supp(\chi)$. Since the operator $f$ is local, we have $\chi f(u)=\chi f(\widetilde{\chi}u)$. Also, since $g_1=g$ on $B(0, r_1)$, which is where $z$ is supported, in $T_{R_1}$ we have $\Delta_g z=\Delta_{g_1} z$ and thus the local Sobolev norm coincides there. Then $z:=\chi u$ satisfies
\begin{align*}
    \left\{\begin{array}{rl}
        i\partial_t z+\Delta_{g_1} z=\chi f(z+\bh_1)+\bh_2 &\ (0, T)\times T_{R_1}, \\
        \partial_t z=0 &\ (0, T)\times \omega_1,\\ 
    \end{array}\right.
\end{align*}
where we have set
\begin{align*}
    \left\{\begin{array}{cccl}
        \bh_1: & t\in [0, T] & \longmapsto & (1-\chi)\widetilde{\chi}u(t)\in H^{2}(T_{R_1}),  \\
        \bh_2: & t\in [0, T] & \longmapsto & [\Delta_g, \chi]u(t)\in H^{2}(T_{R_1}).
    \end{array}\right.
\end{align*}
This is indeed consistent since both $(1-\chi)\widetilde{\chi}u$ and $[\Delta_g, \chi]u$ are supported in $B(0, r_1)\setminus B(0, R)\Subset T_{R_1}$ where they do not depend on the time variable and enjoy higher regularity in space, say $H^2(\R^2)$. 

By hypothesis $u$ has finite Strichartz norms, which we can use in a similar way to what is done in the compact boundaryless case (see \cite[Remark A.5]{BGT04}), to obtain that $P'(|u|^2)u\in L^2([0, T], H^1(\R^2))$ and in particular, $\chi P'(|u|^2)u\in L^2([0, T], H^1(\R^2))$. The latter translates into $\chi f(z+\bh_1)\in L^2([0, T], H^1(T_{R_1}))$ (recall that $\chi=0$ outside $B(0, r_1)$).

\medskip
\paragraph{\emph{Step 2. Periodic extension}} Since outside the ball $B(0, r_2)$ and up to the boundary of $T_{R_1}$ we have $g_1=g_{euc}$ and the euclidean metric is translation-invariant, we can construct a manifold $(\T^2, \mathfrak{g})$ with fundamental domain $T_{R_1}$ and equipped with the inherited metric $\mathfrak{g}$ whose projection coincides $g_1$ in the fundamental domain. Note that our choice of cutoffs $\psi$, $\chi$ and $\widetilde{\chi}$ are consistent, in the sense that the solution $z$ is supported in $B(0, r_1)$ where the metric $g$ is active and vanishes elsewhere, and $g_1$ transitions smoothly in regions. We can then consider a periodic extension of $z$, that is, $z^P: [0, T]\times\T^2\to \mathbb{C}$ is defined by $z^P(\cdot, x)=z(\cdot, x^*)$ where $x^*$ is the unique representative of $x\mod T_{R_1}$, and since $\supp z\Subset T_{R_1}$, the periodization $z^P$ belongs to $C^0([0, T], H^1(\T^2))$. Reasoning likewise for the remaining functions, we see that $z^P$ solves
\begin{align*}
    \left\{\begin{array}{rl}
        i\partial_t z^P+\Delta_{\mathfrak{g}} z^P=\chi^P f(z^P+\bh_1^P)+\bh_2^P &\ (0, T)\times \T^2, \\
        \partial_t z^P=0 &\ (0, T)\times \widetilde{\omega}_1,\\ 
    \end{array}\right.
\end{align*}
where $\widetilde{\omega}_1$ is the  periodization of $\omega_1$ into $\T^2$ and
\begin{align*}
    \left\{\begin{array}{cccl}
        \bh_1^P: & t\in [0, T] & \longmapsto & \big((1-\chi)\widetilde{\chi}u\big)^P(t)\in H^{2}(\T^2),  \\
        \bh_2^P: & t\in [0, T] & \longmapsto & \big([\Delta_g, \chi]u\big)^P(t)\in H^{2}(\T^2).
    \end{array}\right.
\end{align*}

First, since $\omega$ satisfies the \cref{NLS:assumGCC} in $(\R^2, g)$ and $R^2\setminus B(0, R)\Subset \omega$, so does $\widetilde{\omega}_1$ on $(\T^2,\mathfrak{g})$. Indeed, on the set $K=T_{R_1}\setminus\omega_1$, the active metric is $g$ and thus every geodesic starting there must enter $\omega_1$ in finite time and any other geodesic starting outside $K$ is already on the observation zone. Second, since the local Sobolev $H^1$-norm in the support of $\chi f$ is determined by $g$, we get that $\chi f(z^P+\bh_1^P)\in L^2([0, T], H^1(\T^2))$.
Since the metric $g_1=g$ on the support $\chi f$, the local $H^1$-norms computed with $g$ and $g_1$ coincides there we have
\begin{align*}
    \norm{\chi f(z^P+\bh_1^P)}_{L^2([0, T], H^1(\T^2))}&=\norm{\chi f(z+\bh_1)}_{L^2([0, T], H^1(T_{R_1}))}\\
    &\leq \norm{\chi f(u)}_{L^2([0, T], H^1(\R^2))}.
\end{align*}
This suffices to apply a slight variation of \cref{prop:propregularityYT} with extra regular parameter (see \cite[Section 3]{DGL06}), which allows us to propagate regularity, obtaining $z^P\in C^0([0, T], H^2(\T^2))$. We are then in the configuration of \cref{thm:mainabs-analytic}. As we did in the proof of \cref{thm:NLS-analytic}, we get that $t\in(0, T)\mapsto z^P(t)\in H^2(\T^2)$ is real analytic and thus we can proceed we did in \cref{thm:UCP-NLSdim3} to obtain that $z^P$ is independent of $t$. By projection into the fundamental domain, we obtain that $t\in(0, T)\mapsto \chi u(t)\in H^2(T_{R_1})$ is real analytic.

Summarizing, by going back to $u=\chi u+(1-\chi)u$, we have proved that $t\in (0, T)\mapsto u(t)\in H^2(\R^2)$ is constant and thus $-\Delta_g u+P'(|u|^2)u=0$ in $\R^2$. For any $x_0\in \R^2$ and $r>0$ we have then
\begin{align*}
    0\leq \int_{B(x_0, r)} |\nabla u|^2dx=-\int_{B(x_0, r)} P'(|u|^2)|u|^2dx\leq -C\int_{B(x_0, r)}|u|^2dx.
\end{align*}
Therefore, $u=0$ on $B(x_0, r)$ and being both $x_0$ and $r$ arbitrary, we conclude that $u\equiv 0$.
\end{proof}

\appendix
\addtocontents{toc}{\protect\setcounter{tocdepth}{1}}
\setcounter{theorem}{0}
\renewcommand{\thetheorem}{\Alph{section}\arabic{theorem}}

\section{Analysis tools}\label{A:eveq}

\subsection{ODEs in Banach spaces} We now introduce the two different notions of ODEs in Banach spaces used in the present article. Let us consider the framework of \cref{sec:NLSabstract-construction} and let $I\subset \R$ be a nonempty interval and take $s_0\in I$. 

For any $s\in \R$, we can easily extend $e^{sA}$ to $C^{0}([0,T],X^{\sigma})$ by the formula
\begin{align*}
\left[e^{sA}V\right](t)=e^{sA}V(t),
\end{align*}
for $V\in C^{0}([0,T],X^{\sigma})$. If $H\in L^{1}\left(I,C^{0}([0,T],X^{\sigma})\right)$, we say that $\xi\in C^{0}\left(I,C^{0}([0,T],X^{\sigma})\right)$ satisfies 
\begin{align}\label{app:ode:ode-banach-1}
    \left\{\begin{array}{lc}
        \dfrac{d}{ds}\xi(s)=A \xi(s)+H(s), & s\in I, \\
        \xi(s_0)=\xi_0, & 
    \end{array}\right.
\end{align}
with $\xi_0\in C^{0}([0,T],X^{\sigma})$, if it satisfies
\begin{align}\label{app:ode:ode-banach-1-duhamel}
\xi(s)&=e^{(s-s_0)A}\xi_0+\int_{s_0}^{s}e^{(s-w)A}H(w)dw,\ \forall s\in I,
\end{align}
with equality in $C^{0}([0,T],X^{\sigma})$.
\begin{lemma}
\label{lmDuhamelCauchy}
If $H\in L^{1}\left(I,C^{0}([0,T],\P_{n}X^{\sigma})\right)$ and $\xi\in C^{0}\left(I,C^{0}([0,T],\P_{n}X^{\sigma})\right)$ for some $n\in\N$ and satisfies 
$ \dfrac{d}{ds}\xi(s)=A \xi(s)+H(s)$ in the previous sense. Then, it satisfies this equation in the sense of Cauchy-Lipschitz.
\end{lemma}
\begin{proof}
It follows as an application of Duhamel's formula.
\end{proof}
\begin{lemma}\label{lmtranslat}
For $T_1<T_2$, let us consider $T\in (0, T_2-T_1)$, $\eta\in (0, T_2-T-T_1)$ and $I:=[T_1-\eta, T_2-T-\eta]$. Let $G\in C^{0}([T_{1},T_{2}],X^{\sigma})$ and assume that $V\in C^{0}([T_{1},T_{2}],X^{\sigma})$ is a mild solution of 
\begin{align*}
    \left\{\begin{array}{lc}
    \dfrac{d}{dt}V(t)=A V(t)+G(t) &\ \text{ for }\ t\in [T_1, T_2]\\
   V(T_1)=V_0. & 
    \end{array}\right.
\end{align*}
If we define $\xi,H\in C^{0}\left(I, C^{0}([0,T],X^{\sigma})\right)$ by $\xi(s)=V^s$ and $H(s)=G^s$ with 
\begin{align*}
    \xi(s)(t)=V^{s}(t)=V(t+s+\eta)\ \text{ and }\ H(s)(t)=G^{s}(t)=G(t+s+\eta),
\end{align*}
for all $s\in I, t\in [0,T]$, then, for any $s_0\in I$, $\xi$ is solution in the sense of \eqref{app:ode:ode-banach-1-duhamel} of
\begin{align*}
    \left\{\begin{array}{lc}
        \dfrac{d}{ds}\xi(s)=A \xi(s)+H(s), & s\in I, \\
        \xi(s_0)=\xi_0, & 
    \end{array}\right.
\end{align*}
with $\xi_0=V^{s_0}=V(\cdot+s_0+\eta)$.
\end{lemma}
\begin{proof} 
By Duhamel's formula, for all $t\in [T_1, T_2]$ we have
\begin{align*}
V(t)&=e^{(t-T_1)A}V_{0}+\int_{T_1}^{t}e^{(t-\tau)A}G(\tau)d\tau.
\end{align*}
with $V(T_1)=V_0$. Pick $s_0\in I$. First observe that, for any $t\in [0, T]$,
\begin{align*}
    V(t+s_0+\eta)=e^{(t+s_0+\eta-T_1)A}V_{0}+\int_{T_1}^{t+s_0+\eta}e^{(t+s_0+\eta-\tau)A}G(\tau)d\tau
\end{align*}
Then, for $s\in I$ and $t\in [0, T]$,
\begin{align*}
V^{s}(t)&=V(t+s+\eta)=e^{(t+s+\eta-T_1)A}V_{0}+\int_{T_1}^{t+s+\eta}e^{(t+s+\eta-\tau)A}G(\tau)d\tau\\
&=e^{(s-s_0)A}V(t+s_0+\eta)+\int_{s_0}^{s}e^{(s-w)A}G(t+w+\eta)dw\\
&=e^{(s-s_0)A}V^{s_0}(t)+\int_{s_0}^{s}e^{(s-w)A}G^{w}(t)dw.
\end{align*}
So, since this is true for any $t\in [0,T]$, it gives 
\begin{align*}
V^{s}=e^{(s-s_0)A}V^{s_0}+\int_{s_0}^{s}e^{(s-w)A}G^{w}dw,\ \forall s\in I.
\end{align*}
By hypothesis, this equality holds in $C^0([0, T], X^\sigma)$, and is exactly \eqref{app:ode:ode-banach-1-duhamel}, as we wanted to prove.
\end{proof}

\subsection{Complex analysis in Banach spaces}\label{s:appcomplex} Let $E$ and $F$ be Banach spaces over the same field $\mathbb{K}$, with $\mathbb{K}$ being either $\R$ or $\mathbb{C}$. Along this appendix, we will introduce several notions of differentiability and analyticity needed to unify the different results used throughout the present article.

We first start with a notion of differentiability in Banach spaces, often referred to as Fr\'echet differentiability.

\begin{definition}
\label{def-Kdifferentiable}
    Let $U$ be an open subset of $E$. A mapping $f: U\to F$ is said to be $\mathbb{K}$-differentiable (or just differentiable) if for each point $x\in U$ there exists a mapping $A\in \mc{L}(E, F)$ such that
    \begin{align*}
        \lim_{h\to 0} \dfrac{\norm{f(x+h)-f(x)-Ah}_F}{\norm{h}_E}=0.
    \end{align*}
    Such map $A$ is called the derivative of $f$ at $x$ and it is denoted by $Df(x)$.
\end{definition}

We can rephrase the above definition as follows: for each $x\in U$ there exists a mapping $A\in \mc{L}(E, F)$ such that
\begin{align*}
    f(x+h)=f(x)+Ah+o(h)
\end{align*}
where $o(h)/\norm{h}_E\to 0$ as $h\to 0$. With this formulation at hand, we state the chain rule in this setting.

\begin{theorem}{\cite[Theorem 13.6]{Muj86}}\label{appendix:thm:chain-rule} (Chain rule)
    Let $E$, $F$ and $G$ be Banach spaces over $\mathbb{K}$. Let $U\subset E$ and $V\subset F$ be two open sets and let $f: U\to F$ and $g: V\to G$ be two differentiable mappings with $f(U)\subset V$. Then the composite mapping $g\circ f: U\to G$ is differentiable as well and $D(g\circ f)(x)=Dg(f(x))\circ Df(x)$ for every $x\in U$.
\end{theorem}

We now come to introduce the different notions of holomorphic or analytic maps that have been used throughout the present article. A mapping $P: E\to F$ is said to be an $k$-homogeneous polynomial if there exists a $k$-linear mapping $A: E^k\to F$ such that $P(x)=A(x,\ldots, x)$ for every $x\in E$. We represent by $\mc{P}(^kE, F)$ the Banach space of all continuous $k$-homogeneous polynomials from $E$ into $F$ under the norm
\begin{align*}
    \norm{P}_{\mc{P}(^kE, F)}=\sup\{\norm{P(x)}_F\ |\ x\in E,\ \norm{x}_E\leq 1\}.
\end{align*}
A series $\sum_{k=0}^\infty f_k$ of homogeneous polynomials $f_k\in \mc{P}(^k E, F)$ will shortly be called a formal series from $E$ to $F$. The space of all formal series with continuous terms will be denoted by $S(E, F)$. We say that a formal series $\sum_{j=0}^\infty f_j$ converges in a set $U\subset E$ if for every $x\in U$ the series $\sum_{j=0}^\infty f_j(x)$ is convergent.

\begin{definition}
    Let $U$ be an open subset of $E$ and $\mathbb{K}=\mathbb{C}$ (resp. $\R$). A continuous mapping $f: U\to F$ is said to be \emph{holomorphic} (resp. \emph{analytic}) if for each $x\in U$ there exist a series $\sum_{j=0}^\infty f_j\in S(E, F)$ such that
    \begin{align*}
        f(x+h)=\sum_{j=0}^\infty f_j(h)
    \end{align*}
    for all $h$ in a neighborhood of $0\in E$. We shall denote by $\H(U, F)$ the vector space of all holomorphic mapping from $U$ into $F$.
\end{definition}

\begin{remark}
    The sequence $(f_j)$ which appears in the above definition is uniquely determined by $f$ and $x$. We then shall write $f_j=f_j(x)$ for every $j\in \N_0$.
\end{remark}

The previous definition has been taken from \cite{BS:71-analytic} and \cite{Muj86}. Observe that here we have reserved the concept \emph{holomorphic} for the complex case and \emph{analytic} for the real case. When going through the literature, it is often the case that holomorphicity is introduced with a different definition. We will introduce these notions and then we will establish that they are equivalent. From now on, assume that $\mathbb{K}=\mathbb{C}$, unless we say otherwise.

\begin{definition}
A mapping $f: U\to F$ is said to be:
\begin{enumerate}
    \item \emph{weakly holomorphic} if $\psi\circ f$ is holomorphic for every $\psi\in F^*$, where $F^*$ is the dual space of $F$;
    \item \emph{G-holomorphic} if for all $x\in U$ and $h\in E$, the mapping $\zeta\mapsto f(x+\zeta h)$ is holomorphic on the open set $\{\zeta\in \mathbb{C}\ |\ x+\zeta h\in U\}$.
\end{enumerate}
\end{definition}

The following theorem shows that one of the most important features of the complex analysis still holds when working with functions between complex Banach spaces.

\begin{theorem}{\cite[Theorem 8.12, Theorem 8.7, Theorem 13.16]{Muj86}}\label{appendix:thm:equiv-holom}
    Let $U$ be an open subset of $E$, and let $f: U\to F$. The following statements are equivalent:
    \begin{enumerate}
        \item $f$ is $\mathbb{C}$-differentiable,
        \item $f$ is holomorphic,
        \item $f$ is weakly holomorphic,
        \item $f$ is continuous and $G-$holomorphic.
    \end{enumerate}
\end{theorem}

For a given $x\in U$ and $h\in E$, let us denote by $\rho(x, h)$ the supremum of all numbers $\rho$ such that $|\zeta|\leq \rho$ implies $x+\zeta h\in U$.

\begin{theorem}{\cite[Theorem 7.1, Corollary 7.3]{Muj86}}\label{appendix:prop:cauchy-form} (Cauchy integral formula)
    Let $U$ be an open subset of $E$, and let $f\in \H(U, F)$. Let $x\in U$, $h\in E$ and $r<\rho(x, h)$. Then for each $\ld\in \mathbb{D}(0, r)$ we have
    \begin{align*}
        f(x+\ld h)=\dfrac{1}{2\pi i}\int_{|\zeta|=r} \dfrac{f(x+\zeta h)}{\zeta-\ld}d\zeta,
    \end{align*}
    where $|\zeta|=r$ denotes a circle of radius $r$ an center at the origin in the complex plane. Moreover, for each $j\in \N$ we have
    \begin{align*}
        f_j(x)(h)=\dfrac{1}{2\pi i}\int_{|\zeta|=r} \dfrac{f(x+\zeta h)}{\zeta^{j+1}}d\zeta.
    \end{align*}
\end{theorem}

Let $f\in \H(U, F)$. We can expand $f(x+\ld h)$ as
\begin{align*}
    f(x+\ld h)=\sum_{j=0}^\infty f_j(x)(\ld h)=\sum_{j=0}^\infty \ld^j f_j(x)(h),
\end{align*}
which holds uniformly for $|\ld|\leq r$ with $0\leq r<\rho(x, h)$. For $x\in U$ we may define the $n$th variation $\delta^n f(x, h)$ of $f(x)$ with increment $h$ as
\begin{align*}
    \delta^n f(x, h)=\left[\dfrac{d^n}{d\zeta^n} f(x+\zeta h)\right]_{\zeta=0}.
\end{align*}
It can be seen that $\delta^n f(x, h)$ is homogeneous of degree $n$ in $h$. Moreover, looking at the Taylor development of the holomorphic map $\ld\in \mathbb{D}(0, r)\mapsto f(x+\ld h)\in F$, in view of the previous result, it follows that
\begin{align}\label{appendix:eq:cauchy-formula-derivatives}
    \delta^n f(x, h)=\dfrac{n!}{2\pi i}\int_{|\zeta|=r} \dfrac{f(x+\zeta h)}{\zeta^{m+1}}d\zeta.
\end{align}
\begin{remark}
    Formula \eqref{appendix:eq:cauchy-formula-derivatives} does not depend on the chosen $r<\rho(x, h)$.
\end{remark}

The above discussion leads us to the classical Cauchy estimates.

\begin{proposition}{\cite[Corollary 7.4]{Muj86}}\label{appendix:prop:cauchy-est} (Cauchy estimates)
    Let $U$ be an open subset of $E$, and let $f\in \H(U, F)$. Let $x\in U$, $h\in E$ and $r<\rho(x, h)$. Then for each $n\in \N$ we have
    \begin{align*}
        \norm{\delta^n f(x)(h)}\leq r^{-n}\sup_{|\zeta|=r}\norm{f(x+\zeta h)}.
    \end{align*}
\end{proposition}

\begin{remark}
Actually, it is possible to have the Cauchy estimates locally around any point $x\in U$ or even uniformly in a ball (by assuming that $f$ is bounded in there). Let us argue for the former case, the latter being similar. By continuity, there exists $r_x>0$ such that $\norm{f(z)}\leq M$ for all $z\in U$ such that $\norm{z-x}\leq r_x$, where $M=M(x)>0$ is a bound that depends on $x$. Let $h\in E$. Thus, for $z$ such that $\norm{z-x}\leq r_x/2$, we have $z+\zeta h\in U$ for any $|\zeta|\leq \tfrac{r_x}{2\norm{h}}$, since
\begin{align*}
    \norm{z+\zeta h-x}\leq \norm{z-x}+\norm{h}\leq \dfrac{r_x}{2}+\dfrac{r_x}{2\norm{h}} \norm{h}=r_x
\end{align*}
and so $z+\zeta h\in B(x, r_x)\subset U$. Due to the Cauchy estimates
\begin{align*}
    \norm{\delta^n f(z)(h)}&\leq \left(\dfrac{2\norm{h}}{r_x}\right)^n \sup_{|\zeta|=\tfrac{r_x}{2\norm{h}}}\norm{f(z+\zeta h)}\leq M \left(\dfrac{2\norm{h}}{r_x}\right)^m.
\end{align*}
The previous estimate holds uniformly on $\norm{z-x}\leq r^*$, for any $r^*<\tfrac{r_x}{2}$.
\end{remark}

\subsubsection{Some regularity results} Here we state some useful regularity results that are used throughout \cref{sec:NLSabstract-construction}. First, we have the following characterization of holomorphic mappings whose domain is an open set in a product of Banach spaces. 

\begin{proposition}{\cite[Proposition 8.10]{Muj86}}\label{appendix:thm:holom-sev-var}
    Let $E_1,\ldots, E_n$ and $F$ Banach spaces, and let $U$ be an open subset of $E_1\times\ldots \times E_n$. Then a mapping $f: U\to F$ is holomorphic if and only if $f$ is continuous and $f(\zeta_1,\ldots, \zeta_n)$ is holomorphic in each $\zeta_j$ when the other variables are held fixed.
\end{proposition}

We can say the following in regards to the regularity of the $n$-th variation of $f$.

\begin{proposition}\cite[Proposition 6.4]{BS:71-analytic}\label{prop:diff-analytic}
    Assume $\mathbb{K}=\mathbb{C}$ (resp. $\R$). If $f: U\to F$ is holomorphic (resp. analytic), then for every $n\in \N$ the function
    \begin{align*}
        \delta^n f: (x, h)\in U\times E\longmapsto \delta^n f(x)(h)\in F
    \end{align*}
    is holomorphic (resp. analytic).
\end{proposition}

\begin{proposition}\cite[Exercise 8.E.]{Muj86}
    Let $E$, $F$, $G$ be Banach spaces, let $U$ be an open subset of $E$, and let $f: U\to \mc{L}(F, G)$. The following conditions are equivalent:
    \begin{enumerate}
        \item $f$ is holomorphic.
        \item The mapping $x\in U\mapsto f(x)(y)\in G$ is holomorphic for each $y\in F$.
        \item The function $x\in U\mapsto \eta(f(x)(y))\in \mathbb{C}$ is holomorphic for each $y\in F$ and $\eta\in G'$.
    \end{enumerate}
\end{proposition}

Let $\Isom(E, F)\subset \mc{L}(E, F)$ be the space of invertible linear continuous maps from $E$ into $F$. Let $\mathfrak{I}: \Isom(E, F)\to \mc{L}(F, E)$ be the map $\mathfrak{I}(u)=u^{-1}\in\Isom(F, E)$. As a consequence of the Neumann series  (see \cite[Theorem 1.7.3]{Car67}) we can establish that $\mathfrak{I}$ is an analytic map in suitable neighborhoods of bijective maps. Indeed, for any $L\in \mc{L}(E, F)$ on an $\veps$-neighborhood of a bijection $T\in \mc{L}(E, F)$ with $0<\veps<1/\norm{T^{-1}}$, then $L^{-1}\in \L(F, E)$ and
\begin{align*}
    L^{-1}=\big(I-T^{-1}(T-L)\big)^{-1}T^{-1}=\sum_{k=0}^\infty \big(T^{-1}(T-L)\big)^kT^{-1}=\sum_{k=0}^\infty m_k(T-L)^k,
\end{align*}
where $m_k$ is defined by
\begin{align*}
    m_k(L_1,\ldots, L_k)=\dfrac{1}{k!}\sum_{\pi\in S_k}T^{-1}\circ L_{\pi(1)}\circ T^{-1}\circ L_{\pi(2)}\circ\ldots \circ T^{-1}\circ L_{\pi(k)}\circ T^{-1},
\end{align*}
with the summation being taken over all $k!$ permutations of $\{1,\ldots, n\}$. This shows that $\mathfrak{I}: L\mapsto L^{-1}$ from $\Isom(E, F)$ into $\Isom(F, E)$ is $\mathbb{K}$-analytic on the neighborhood of $T$ given by the previous result.

\begin{lemma}\label{app:lem:bij2} Suppose $T\in \L(E, F)$ is a bijection. Then for any $0<\veps<1/\norm{T^{-1}}$ such that if $L\in \L(E, F)$ and $\norm{T-L}<\veps$, then $L^{-1}\in \L(F, E)$. Moreover, $\mathfrak{I}: L\mapsto L^{-1}$ as a map from $\Isom(E, F)$ into $\Isom(F, E)$ is $\mathbb{K}$-analytic on any of these $\veps$-neighborhood centered at $T$.
\end{lemma}

\subsubsection{On the complexification} Let $E$ and $F$ be real Banach spaces. The canonical complexification $E_\mathbb{C}=E+iE$ is a complex Banach space equipped with the norm whose square is $\norm{x+iy}_{E_\mathbb{C}}^2=\norm{x}_E^2+\norm{y}_E^2$. If $A\in\mc{L}(E, F)$ is a bounded linear operator, its complexification is
\begin{align*}
    A_\mathbb{C}(x+iy):=A(x)+iA(y),\ x, y\in E.
\end{align*}
Note that $A_\mathbb{C}\in \mc{L}(E_\mathbb{C}, F_\mathbb{C})$, where the latter denotes the space of $\mathbb{C}$-linear bounded operators from $E_\mathbb{C}$ into $F_\mathbb{C}$ with the inherited complex structure.

\begin{lemma}\label{app:lem:Lopcomplex}
    It holds $\mc{L}(E_\mathbb{C}, F_\mathbb{C})\simeq\mc{L}(E, F)_\mathbb{C}$ as complex Banach spaces and the map
    \begin{align*}
        \left\{\begin{array}{crcl}
            \Psi_{E\to F}: & \mc{L}(E, F)_\mathbb{C}  & \longrightarrow & \mc{L}(E_\mathbb{C}, F_\mathbb{C})  \\
            &   A+iB     &   \longmapsto   &    A_\mathbb{C}+iB_\mathbb{C},
        \end{array}\right.
    \end{align*}
    is $\mathbb{C}$-linear isometric isomorphism in between Banach spaces.
\end{lemma}
\begin{proof}
    Since the spaces $E$ and $F$ are fixed, we drop the subscript $E\to F$. Under the identification $E_\mathbb{C}=E+iE$ and $F_\mathbb{C}=F+iF$, we see that $\Psi$ acts as
    \begin{align*}
        \begin{pmatrix}
            A & -B \\ B & A
        \end{pmatrix}: (x, y)\longmapsto (Ax-By, Bx+Ay).
    \end{align*}
    With this at hand is not hard to see with some algebra that $\Psi(S)\Psi(T)=\Psi(ST)$ for any $T\in \mc{L}(E, F)_\mathbb{C}$ and $S\in \mc{L}(F, G)_\mathbb{C}$.

    Let us denote by $\iota_E: E\to E_\mathbb{C}$ the embedding $\iota(x)=x+i0$ and define $\iota_F$ similarly. Let $\mc{E}_1^F$, $\mc{E}_2^F\in \mc{L}(F_\mathbb{C}, F)$ be the bounded $\mathbb{R}$-linear projectors
    \begin{align*}
        \mc{E}_1^F(u+iv)=u\ \text{ and }\ \mc{E}_2^F(u+iv)=v.
    \end{align*}
    Given $T\in \mc{L}(E_\mathbb{C}, F_\mathbb{C})$, we define $A:=\mc{E}_1^FT\iota_E$ and $B:=\mc{E}_2^FT\iota_E$, both of them belonging to $\mc{L}(E, F)$. For any $x, y\in E$ we have
    \begin{align*}
        T(x)=Ax+iBx\ \text{ and }\ T(iy)=iT(y)=-By+iAy.
    \end{align*}
    Hence, the $\mathbb{C}$-linearity of $T$ forces
    \begin{align*}
        T(x+iy)=(Ax-By)+i(Bx+Ay),\ \forall x, y\in E.
    \end{align*}
    This means that $T=\Psi(A+iB)$ and thus $\Psi$ is onto.
    
    For $T=\Psi(A+iB)$ and under the block-matrix identification of $\Psi$, some algebra along with a trigonometric change of variable, lead us to
    \begin{align*}
        \norm{\Psi(A+iB)}_{\mc{L}(E_\mathbb{C}, F_\mathbb{C})}=\sup_{\substack{x, y\in E\\ x, y\neq 0}}\dfrac{\sqrt{\norm{Ax-By}_F^2+\norm{Bx+Ay}_F^2}}{\sqrt{\norm{x}_E^2+\norm{y}_E^2}}=\sup_{\theta\in [0, 2\pi]}\norm{A\cos\theta-B\sin\theta}_{\mc{L}(E, F)}.
    \end{align*}
    Since the canonical complexification norm on $\mc{L}(E, F)_\mathbb{C}$ is defined so that
    \begin{align*}
        \norm{A+iB}_{\mc{L}(E, F)_\mathbb{C}}:=\sup_{\theta\in [0, 2\pi]}\norm{A\cos\theta-B\sin\theta}_{\mc{L}(E, F)},
    \end{align*}
    we conclude that $\Psi$ is an isometric isomorphism of Banach spaces.
\end{proof}

Let $(E, \inn{\cdot, \cdot}_E)$ be a real Hilbert space. Its canonical complexification $E_\mathbb{C}=E+iE$ is a complex Hilbert space equipped with the inner product
\begin{align*}
    \inn{u, v}_{E_\mathbb{C}}=\inn{x_1, y_1}_E+\inn{x_2, y_2}_E+i\big(\inn{x_2, y_1}_E-\inn{x_1, y_2}_E\big),
\end{align*}
with $x=x_1+ix_2$ and $y=y_1+iy_2$. If $(F, \inn{\cdot, \cdot}_F)$ is another Hilbert space, we can introduce the bounded linear map $\adj: L\in \mc{L}(E, F)\mapsto L^*\in\mc{L}(F, E)$, see \cite[Remark 16]{Brez11}. Here the adjoint is taken with respect to the real structure of $E$ and $F$.

\begin{lemma}\label{app:lem:adjext}
    The map $\adj$ admits a holomorphic extension $\widetilde{\adj}$ from $\mc{L}(E_\mathbb{C}, F_\mathbb{C})$ into $\mc{L}(F_\mathbb{C}, E_\mathbb{C})$.
\end{lemma}
\begin{proof}
    Since $\adj$ is bounded and linear, it can be extended as a holomorphic $\mathbb{C}$-linear map from $\mc{L}(E, F)_\mathbb{C}$ into $\mc{L}(F, E)_\mathbb{C}$, by $\adj(L_1+iL_2)=L_1^*+iL_2^*$ where the adjoint is taken with respect to the real underlying structures of $E$ and $F$. Then $\widetilde{\adj}:=\Psi_{F\to E}\circ\adj\circ \Psi_{E\to F}^{-1}$ is the desired holomorphic extension.
\end{proof}

Through complexification of the underlying spaces, the following theorem permits to treat a real analytic function as a restriction of some holomorphic function.

\begin{theorem}\cite[Theorem 7.2]{BS:71-analytic}\label{app:thm:h-ext}
    Assume $\mathbb{K}=\R$. For any analytic function $f: U\to F$ one may find an open subset $V$ of $E_{\mathbb{C}}$ and a holomorphic function $\Tilde{f}: V\to F_\mathbb{C}$ such that $U\subset V$ and $\Tilde{f}_{|_U}=f$.
\end{theorem}

\subsection{Analysis tools} The following result, known as the Uniform Contraction Principle, elucidates the regularity that can be obtained for a parameter-dependent fixed point.

\begin{theorem}{\cite[Theorem 2.2]{CH82}}\label{app:thm:uniffixedpoint}
    Let $U$, $V$ be open sets in Banach spaces $X$, $Y$, let $\overline{U}$ be the closure of $U$, $T:\overline{U}\times V\to \overline{U}$ a uniform contraction on $\overline{U}$ and let $g(y)$ be the unique fixed point of $T(\cdot, y)$ in $\overline{U}$. If $T\in C^k(\overline{U}\times V, X)$, $0\leq k<\infty$, then $g(\cdot)\in C^k(V, X)$. If there is a neighborhood $U_1$ of $\overline{U}$ such that $T$ is analytic from $U_1\times V$ to $X$, then the mapping $g(\cdot)$ is analytic from $V$ to $X$.
\end{theorem}

Observe that the definition of analyticity used in \cite{CH82} combines $G$-analyticity and weakly analyticity. From \cref{appendix:thm:equiv-holom}, these notions of analyticity are equivalent.

We recall the following classical Aubins-Lions lemma.

\begin{theorem}\label{app:thm:aubinlions}\cite[Theorem II.5.16]{BF13}
    Let $B_0\subset B_1\subset B_2$ be three Banach spaces. We assume that the embedding of $B_1$ in $B_2$ is continuous and that the embedding of $B_0$ in $B_1$ is compact. Let $p$, $r$ such that $1\leq p, r\leq +\infty$. For $T>0$, we define
    \begin{align*}
        E_{p, r}=\left\{v\in L^p([0, T], B_0)\ |\ \dfrac{dv}{dt}\in L^r([0, T], B_2)\right\}.
    \end{align*}
    \begin{enumerate}
        \item If $p<+\infty$, the embedding of $E_{p, r}$ in $L^p([0, T], B_1)$ is compact.
        \item If $p=+\infty$ and if $r>1$, the embedding of $E_{p, r}$ in $C^0([0, T], B_1)$ is compact.
    \end{enumerate}
\end{theorem}

\section{Pseudodifferential operators}\label{A:pseudodiff}

Let $\M$ be a compact boundaryless smooth connected Riemannian manifold of dimension $d$. For each $x\in \M$ we denote by $T_x\M$ the tangent space to $\M$ at $x$ and by $T_x^*\M$ its dual space, the cotangent space to $\M$ at $x$. Let $\pi: T\M\to\M$ and $\pi: T^*\M\to\M$ denote the canonical projections into the manifold. We denote by $\inn{\cdot,\cdot}_x=\inn{\cdot, \cdot}_{T_x^*\M, T_x\M}$ the duality bracket at $x$. The manifold $\M$ is equipped with a Riemannian metric $g$, meaning that for any $x\in \M$, $g_x$ is a positive definite quadratic form on $T_x\M$ depending smoothly on $x$. This Riemannian metric induces an isomorphism $T_x\M\to T_x^*\M$ defined as $v\mapsto v^\flat:=g_x(v, \cdot)$, with inverse $v=(v^\flat)^\sharp$. The metric $g$ on $T\M$ induces a metric $g^*$ on $T^*\M$, canonically defined by $g_x^*(\xi, \eta)=g_x(\xi^\sharp, \eta^\sharp)$ for $x\in \M$ and $\xi$, $\eta\in T_x^*\M$. We denote by $S^*\M$ the Riemannian cosphere bundle over $\M$, with fiber $x\in \M$ given by $\{\xi\in T\M\ |\ g_x^*(\xi, \xi)=1\}$.

For classical references on pseudodifferential operators, see \cite{Hor85ALPDOIII, Shu01}. Here below we follow the presentation of \cite{Lef25}.

\subsection{Definitions} Let $X\subset \R^d$ be an open subset. We say that $a\in S_{\loc}^m(X\times\R^d)$ if $\vp(x)a(x, \xi)$ belongs to the usual class of symbols $S^m(\R^d\times\R^d)$ for every $\vp\in C_c^\infty(X)$. We consider a quantization $\Op(a)$, defined as a map from $C_c^\infty(X)$ into $\D'(X)$ whose Kernel is given by
\begin{align*}
    K(x, y)=\dfrac{1}{(2\pi)^d}\int_{\R_\xi^d} e^{i(x-y)\cdot\xi}a(x, \xi)d\xi.
\end{align*}
The set consisting on such quantizations $\Op(a)$ of symbols $a\in S^m(X\times\R^d)$ is called the class of \emph{pseudodifferential operators} of order $m$ and is denoted by $\Psi^m(X)$. This choice of quantization is not unique. The class of smoothing pseudodifferential operators $\Psi^{-\infty}(X)$ corresponds to the quantization of symbols $a\in S^{-\infty}(X\times\R^d):=\cap_{m\in\R}S^m(X\times\R^d)$.

Any diffeomorphism $\kappa: U\subset\M\to X\subset\R^d$ induces a map $\kappa^*: C^\infty(T^*\M)\to C^\infty(T^*U)$ defined by $\widetilde{\kappa}^*\vp(x, \xi)=\vp(\kappa(x), d\kappa(x)^{-\tr}\xi)$.

\begin{definition}
    The class $\Psi^m(\M)$ of pseudodifferential operators of order $m$ is defined as the set of continuous linear operators $A=A(x, D_x): C^\infty(\M)\to C^\infty(\M)$ such that:
    \begin{itemize}
        \item For any $\chi$, $\chi'\in C^\infty(\M)$ with disjoint support, $\chi A\chi'$ is smoothing.
        \item For every chart $(\kappa, U)$, for every $\chi, \chi'\in C_c^\infty(U)$, the operator $A_{\kappa, \chi, \chi'}:=\kappa_*\chi'A\chi\kappa^*$ belongs to $\Psi^m(X)$.
    \end{itemize}
    The set $\Psi_{\phg}^m(\M)$ of polyhomogeneous pseudodifferential operators of order $m$ is defined likewise by considering instead the class $\Psi_{\phg}^m(X)$ in the latter property.
\end{definition}

The symbol space $S^m(T^*\M)$ consists of those functions $a\in C^\infty(T^*\M)$ such that for any chart $\kappa: U\to X\subset\R^d$, $(\kappa^{-1})^*a\in S^m(X\times\R^d)$. 

We say that a family $(\kappa_i, U_i)_{i=1}^N$ is a family of cutoff charts if $\bigcup_{i=1}^N U_i=\M$ covers $\M$. For a given family of cutoff charts, we consider a partition of unity $\sum_{i=1}^N\chi_i=1$ subordinated to that cover, as well as other cutoff functions $\chi_i'\in C_c^\infty(U_i)$ such that $\supp\chi_i\Subset \{\chi_i'=1\}$. A quantization procedure is a map $\Op: S^m(T^*\M)\to \Psi^m(\M)$ given by
\begin{align*}
    \Op(a)u=\sum_{i=1}^N\kappa_i^*\big((\kappa_i)_*\chi_i'\Op_{\R^d}((\widetilde{\kappa}_i)_*(\chi_i'a)\big)(\kappa_i)_*\chi_i u)
\end{align*}
where $\Op_{\R^n}$ is a previously chosen quantization on $\R^d$. This way, by \cite[Proposition 5.2.14]{Lef25}, every $A\in \Psi_{\phg}^m(\M)$ is of the form $\Op(a)+R$ where $a\in S^m(T^*\M)$ and $R\in \Psi^{-\infty}(\M)$. Also, given $A\in\Psi^m(\M)$, its principal symbol $\sigma_A\in S^m(T^*\M)$ is well-defined and belongs to $S^m(T^*\M)/S^{m-1}(T^*\M)$.

We recall the algebra of pseudodifferential operators in the following proposition.

\begin{proposition}\cite[Theorem 5.2.16, Lemma 5.2.17]{Lef25}\label{app:prop:psido-algebra}
    The following holds:
    \begin{enumerate}
        \item If $A\in \Psi^{m_1}(\M)$ and $B\in\Psi^{m_2}(\M)$, then $AB\in\Psi^{m_1+m_2}(\M)$ and $\sigma_{A\circ B}=\sigma_A\sigma_B=\sigma_{B\circ A}$. Additionally, $[A, B]\in\Psi^{m_1+m_2-1}$ and $\sigma_{[A, B]}=\frac{1}{i}\{\sigma_A, \sigma_B\}$.
        \item If $C\in\Psi^{m}(\M)$, then $C^*\in \Psi^{m}(\M)$ and $\sigma_{C^*}=\overline{\sigma_{C}}$.
    \end{enumerate}  
\end{proposition}

Regarding their mapping properties, we have the following.

\begin{proposition}\cite[Theorem 5.4.9]{Lef25}\label{app:prop:sobcont}
    Let $A\in \Psi^m(\M)$. Then, for every $s\in \R$, $A: H^{s+m}(\M)\to H^{s}(\M)$ is bounded. In particular, if $K\in \Psi^{-\infty}(\M)$, then for all $s$, $t\in \R$, $K: H^s(\M)\to H^t(\M)$ is bounded.    
\end{proposition}

\subsubsection{Invertibility and positivity} Let $T_0^*\M$ denote the cotangent bundle of $\M$ with the zero section removed.

\begin{definition}
    An operator $A\in \Psi^m(\M)$ is elliptic at $(x_0, \xi_0)\in T_0^*\M$ if there exists $C>0$ and a conic neighborhood $V\subset T^*\M$ of $(x_0, \xi_0)$ such that for all $(x, \xi)\in V$ and $|\xi|_x\geq C$
    \begin{align*}
        |\sigma_A(x, \xi)|\geq \langle\xi\rangle^m/C,
    \end{align*}
    We say that $A$ is elliptic if it is elliptic on $T^*\M$.
\end{definition}

The following lemma states the existence of a local parametrix.

\begin{lemma}\label{lem:localparametrix}\cite[Lemma 5.3.11]{Lef25}
    Let $A\in \Psi^m(\M)$ be elliptic at $(x_0, \xi_0)\in T_0^*\M$. Then, there exists $B\in\Psi^{-m}(\M)$ elliptic at $(x_0, \xi_0)\in T^*\M$, $\chi\in S^0(T^*\M)$ equal to $1$ in a conic neighborhood of $(x_0, \xi_0)$, and $K_1$, $K_2\in \Psi^{-\infty}(\M)$ such that
    \begin{align*}
        AB=\Op(\chi)+K_1,\ \ BA=\Op(\chi)+K_2.
    \end{align*}
\end{lemma}

\begin{theorem}{(The sharp G{\aa}rding's inequality)} \cite[Theorem 6.1.9]{Lef25}\label{app:thm:garding}
    Let $A\in\Psi_{\phg}^m(\M)$ with $m\geq 0$ and assume that $\Re(\sigma_A)\geq 0$. Then, there exists a constant $C>0$ such that
    \begin{align*}
        \Re\inn{Au, u}_{L^2(\M)}\geq -C\norm{u}_{H^{\frac{m-1}{2}}(\M)}^2.
    \end{align*}
\end{theorem}

\subsubsection{Symbol transport} Let $p(x, \xi)=|\xi|_x^2=g_x^*(\xi, \xi)\in C^\infty(T_0^*\M)$ and denote by $H_p$ and $\Phi_t$ the associated Hamiltonian vector field and flow, respectively, meaning that
\begin{align*}
    \dfrac{d}{dt}\Phi_t(\rho)=H_p(\Phi_t(\rho)),\ \quad \Phi_0(\rho)=\rho\in T^*\M,
\end{align*}
which, in local charts can be expressed by $H_p=\nabla_\xi p\cdot\nabla_x-\nabla_x p\cdot \nabla_\xi$. The following result tell us how symbols can be transported along the Hamiltonian flow associated to $p$.

\begin{lemma}\cite[Lemma 3.1]{Lau14}\label{lem:propagation-symbol}
    Let $\rho_0\in T_0^*\M$ . Then, for any $\rho_1=\Phi_t(\rho_0)$ and $V_1$ a small conic neighborhood of $\rho_1$, there exists a neighborhood $V_0$ of $\rho_0$ such that for any symbol $\mathfrak{c}=\mathfrak{c}(x, \xi)$ homogeneous of order $s$ supported in $V_0$, there exists another symbol $\mathfrak{b}=\mathfrak{b}(x,\xi)$ homogeneous of order $s-1$ such that
    \begin{align*}
        H_p \mathfrak{b}(x, \xi)=\mathfrak{c}(x, \xi)+\mathfrak{r}(x, \xi)
    \end{align*}
    where $\mathfrak{r}$ is of order $s$ supported in $V_1$.
\end{lemma}


\end{document}